\documentclass[a4paper]{article}

\usepackage[french,english]{babel}
\usepackage[utf8]{inputenc}

\usepackage[top=1.5cm, bottom=1.5cm, left=2.5cm, right=2.5cm]{geometry}
\usepackage{enumerate}

\usepackage{tikz}
\usetikzlibrary{patterns}

\usepackage{amsmath}
\usepackage{amsthm}
\usepackage{amssymb}
\usepackage{dsfont}
\usepackage{stmaryrd}
\usepackage{setspace}
\usepackage{mathrsfs}
\usepackage{latexsym}

\usepackage[colorlinks=true,linkcolor=blue,citecolor=blue]{hyperref}

\theoremstyle{definition}
\newtheorem{theo}{Theorem}[section]
\newtheorem{de}[theo]{Definition}

\newtheorem{prop}[theo]{Proposition}
\newtheorem{propde}[theo]{Proposition-Definition}
\newtheorem{lem}[theo]{Lemma}
\newtheorem{cor}[theo]{Corollary}
\newtheorem{ex}[theo]{Example}
\newtheorem{exs}[theo]{Examples}
\newtheorem{rem}[theo]{Remark}
\newtheorem{rems}[theo]{Remarks}
\numberwithin{equation}{section}

\newcommand{\A}{\mathscr{A}}			
\newcommand{\B}{\mathscr{B}}			
\newcommand{\C}{\mathscr{C}}			
\newcommand{\D}{\mathscr{D}}			
\newcommand{\qd}{\mathfrak{d}}			
\newcommand{\E}{\mathcal{E}}			
\newcommand{\F}{\mathscr{F}}			
\newcommand{\G}{\mathscr{G}}			
\newcommand{\N}{\mathbf{N}}				
\newcommand{\bigO}{\mathcal{O}}			
\newcommand{\PP}{\mathscr{P}}			
\newcommand{\myprob}{\mathbf{P}}		
\newcommand{\Q}{\mathbf{Q}}				
\newcommand{\R}{\mathbf{R}}				
\newcommand{\Sph}{\mathbf{S}}			
\newcommand{\T}{\mathcal{T}}			
\newcommand{\V}{\mathcal{V}}			
\newcommand{\Z}{\mathbf{Z}}				
\newcommand{\myesp}{\mathbf{E}}			
\newcommand{\eps}{\varepsilon}			

\newcommand{\lp}{\left(}
\newcommand{\rp}{\right)}

\newcommand{\ls}{\leqslant}
\newcommand{\gs}{\geqslant}
\newcommand{\pre}{\preccurlyeq}

\newcommand{\myindic}[1]{\mathds{1}_{#1}}	
\newcommand{\prob}[1]{\myprob\lp #1\rp}		
\newcommand{\esp}[1]{\myesp\left[ #1\right]}		
\newcommand\restr[2]{\ensuremath{\left.#1\right|_{#2}}}

\newcommand{\ds}{\displaystyle}

\newcommand{\ie}{{\it i.e. }}			

\newcommand{\speq}{~=~}						
\newcommand{\+}{\,+\,}						
\newcommand{\spm}{\,-\,}					
\newcommand{\spand}{\quad\text{and}\quad}	
\newcommand{\spas}{\quad\text{as}\quad}		

\begin{document}

\title{\textbf{Hölder regularity of a Wiener integral in abstract space}}
\date{\today}

\author{\textbf{Brice Hannebicque} \\ Université Paris-Saclay, CentraleSupélec, MICS and CNRS FR-3487. \\ \href{mailto:brice.hannebicque@centralesupelec.fr}{\texttt{brice.hannebicque@centralesupelec.fr}} 
   \and \textbf{Érick Herbin} \\ Université Paris-Saclay, CentraleSupélec, MICS and CNRS FR-3487. \\ \href{mailto:erick.herbin@centralesupelec.fr}{\texttt{erick.herbin@centralesupelec.fr}} }

\maketitle

\begin{abstract}
	In this article, we propose a way to consider processes indexed by a collection $\A$ of subsets of a general set $\T$. A large class of vector spaces, manifolds and continuous $\R$-trees are particular cases.  
	Lattice-theoretic and topological assumptions are considered separately with a view to clarifying the exposition. 
	We then define a Wiener-type integral $Y_A=\ds\int_A f\,\text dX$ for all $A\in\A$ for a deterministic function $f:\T\rightarrow\R$ and a set-indexed Lévy process $X$. 
	It is a particular case of Raput and Rosinski \cite{rajput_spectral_1989}, but our setting enables a quicker construction and yields more properties about the sample paths of $Y.$
	Finally, bounds for the Hölder regularity of $Y$ are given which indicate how the regularities of $f$ and $X$ contributes to that of $Y$. This follows the works of Jaffard \cite{jaffard_multifractal_1999} and Balança and Herbin \cite{balanca_2-microlocal_2012}.
\end{abstract}

\tableofcontents

\bigskip

The regularity of sample paths of stochastic processes have been an object of deep interest as early as the 1960s.
This field, far from having dried out, still is nowadays the place of a thriving research as diverse as the study of random fractals (see e.g. the survey of Xiao \cite{lapidus_random_2004} or Khoshnevian \emph{et al.} \cite{khoshnevisan_intermittency_2015}), wavelet theory (e.g. Ayache \emph{et al.} \cite{ayache_wavelet_2020}), regular parametrization of random curves (e.g. Lawler and Zhou \cite{lawler_sle_2013}) and many more. All those works have the common point that the regularity of sample paths is used to better understand the processes at play.

\medskip

Let us consider a Lévy process $\mathtt X=\{ \mathtt X_t:t\in\R_+ \}$. 
As expressed above, there are a lot of ways one can talk about the regularity of $\mathtt X.$ We are interested in the Hölder regularity of the sample paths of $\mathtt X.$
Jaffard \cite{jaffard_multifractal_1999}, followed by Balança \cite{balanca_fine_2014}, studied the process $\mathtt X$ in an extremely refined fashion.
In particular, they extended a classical result from Blumenthal and Getoor \cite{blumenthal_sample_1961} and Pruitt \cite{pruitt_growth_1981} which gives the pointwise Hölder exponent $\alpha_{\mathtt X}(t)$ at any $t\in\R_+$ depending on the Lévy-Khintchine triplet of $\mathtt X.$ 

\medskip

For a given, regular enough, function $f:\R_+\rightarrow\R$, we can define the `primitive process' $\mathtt Y$ given by
\begin{equation}	\label{intro:Y_1d}
	\forall t\in\R_+,\quad \mathtt Y_t \speq \int_{[0,t]} f\,\text d\mathtt X.
\end{equation} 

In the continuation of the previous works, we study the regularity of $\mathtt Y$ based on those of $\mathtt X$ and $f.$ Namely, under some additional hypotheses (see Corollary \ref{cor:regularity_Y_1d}), we prove in particular that, for all $t\in\R_+,$ the pointwise exponent of $\mathtt Y$ satisfies with probability one:
\begin{equation}	\label{intro:main_result} 
	\alpha_{\mathtt Y}(t) \speq \alpha_{\mathtt X}(t) \+ \alpha_f(t)\myindic{f(t)=0}.
\end{equation}
A result of this kind has already been established by Balança and Herbin \cite{balanca_2-microlocal_2012} for a stochastic integral with respect to a continuous semimartingale. 
Since $\mathtt X$ has in general a dense set of discontinuities owing to its Poissonian part, we had to take a different approach.
However, having a lot of discontinuities is not always a disadvantage when one wants to understand Hölder regularity. When they are plenty, they provide bounds on the Hölder exponent. 
Jaffard developed this idea in \cite{jaffard_old_1997} to determine the Hölder regularity of deterministic functions and successfully applied it in \cite{jaffard_multifractal_1999} to characterize the full multifractal spectrum of $\mathtt X$. We shall adapt those techniques to obtain (\ref{intro:main_result}) and related results. 

\medskip
	
Actually, we took the opportunity to study more general processes than $\R_+$-indexed ones. Several advantages of this approach will be discussed as we proceed.
Let us consider for now a generic set $\T$ as well as a function $f:\T\rightarrow\R.$ In order to be able to define an expression similar to (\ref{intro:Y_1d}), it is more natural to consider a Lévy process $X=\{X_A:A\in\A\}$ indexed by a collection $\A$ of subsets of $\T$ rather than by $\T$ itself. 
The existence and properties of $X$ have been studied by Herbin and Merzbach \cite{herbin_set-indexed_2013}, extending a work from Bass and Pyke \cite{bass_existence_1984}.
This enables to view $X$ as a stationary random measure, with respect to which it is possible to define an integral.
Hence, the goal of this article is to study the Hölder regularity of the sample paths of the set-indexed process $Y$ given by:
\begin{equation}	\label{intro:Y_A}
	\forall A\in\A,\quad Y_A \speq \int_A f\,\text dX
\end{equation}
and express it in terms of the regularity of $f$ and $X.$
When $\T=\R_+,$ $\A$ may be chosen as the collection of all segment $[0,t]$ for $t\in\R_+,$ so (\ref{intro:Y_1d}) is a particular case.

\medskip

The theory of set-indexed processes is exposed in \cite{capasso_set-indexed_2003, ivanoff_set-indexed_1999, merzbach_introduction_2003} by Ivanoff and Merzbach.
The original need was to generalize the theory of processes \emph{à deux indices} (\ie $\R_+^2$-indexed processes) pioneered by \cite{cairoli_stochastic_1975, merzbach_extension_1981, korezlioglu_theorie_1981, azema_differents_1983, walsh_convergence_1979, wong_extension_1977, wong_sample_1977} and many others.
Maybe unexpectedly, it turns out that this setting enables the study of $\T$-indexed processes where $\T$ is quite a general partially ordered set (see Proposition \ref{prop:lattice_correspondance} for the exact correspondance and \cite{ivanoff_compensator_2007} for a related approach).
For instance, it encompasses the multiparameter case $\T=\R_+^p$ (see Khoshnevisan \cite{khoshnevisan_multiparameter_2002} for a modern exposition) and processes indexed by continuous trees (see Evans \cite{evans_probability_2008}, Le Gall \cite{le_gall_random_2005}, Lin \cite{lin_tree-indexed_2014} and references therein), which both constitute important subjects in modern-day probability.

\medskip

In Section \ref{sec:set-indexed}, we expose the framework of set-indexed theory. 
The approach taken is a bit different from the classical one where the definition of the \emph{indexing collection} $\A$ gathers in one place both lattice-related and topological assumptions. We aim to simplify the exposition by splitting the properties between several subsections and separately investigating their consequences.

\medskip

In Section \ref{sec:construction_integral}, we rigorously define the integral (\ref{intro:Y_A}) in the same fashion as the classical Wiener integral. This appears as a particular case of the construction of Rajput and Rosinski \cite{rajput_spectral_1989}, but the stationarity of the increments of $X$ enables a quicker construction.
Thanks to the set-indexed framework, a Lévy-Itô decomposition for the sample paths of $Y$ is proven in the same spirit as Herbin and Merzbach \cite[Theorem 7.9]{herbin_set-indexed_2013}. 

\medskip

In Section \ref{sec:regularity_criterion}, we discuss several notions of Hölder regularity following Herbin and Richard \cite{herbin_local_2016}. 
Usually, determining the Hölder regularity of some function $h:\R_+\rightarrow\R$ at $t\in\R_+$ is to find the best $\alpha\gs0$ such that the estimate $|h(s)-h(t)| \ls |s-t|^\alpha$ holds true for $s$ sufficiently close to $t.$
So the focus is on bounding the increments of $h.$ 
However, when the partial order on $\T$ is not total, there is a variety of increments that one could want to bound. For $\T=\R_+^2$, two notions of increments are usually considered for $h:\R_+^2\rightarrow\R$ and $(s_1,s_2),(t_1,t_2)\in\R_+^2$: the `naive' increment $h(s_1,s_2)-h(t_1,t_2)$ and the \emph{rectangular increment} $h(s_1,s_2)-h(s_1,t_2)-h(t_1,s_2)+h(t_1,t_2)$. 
We provide corresponding definitions to tackle those different situations. 
We then further push ideas from \cite{jaffard_multifractal_1999, jaffard_old_1997} to obtain (deterministic) upper bounds for the Hölder regularity of a deterministic function $h:\A\rightarrow\R$ based on its pointwise jumps.

\medskip

In Section \ref{sec:regularity_thm}, we prove a 0-1 law (Theorem \ref{theo:0-1_law}) and then proceed to our main results (Theorems \ref{theo:regularity_Y_poissonian} and \ref{theo:loc_regularity_Y_poissonian}). 
They give upper and lower bounds on the Hölder regularities of $Y$ in terms of the Lévy-Khintchine triplet of $X$ and estimates on $f$. 
Simple examples are considered, already showing the crucial influence of the geometry of $\T$ on the regularity of $Y.$
Several phenomenons in the multiparameter case invisible in dimension one are investigated.

\medskip

Notations: 
$(\Omega,\F,\myprob)$ is a complete probability space,
$\T$ a non-empty measurable space 
and
$m$ a measure on $\T.$
As soon as their respective definitions are given, 
$\A$ will denote an indexing collection on $\T,$
$d_\A$ a pseudo-metric on $\A$ (abusively called metric in the sequel),
$X$ a set-indexed Lévy process
and
$Y$ will be given by (\ref{intro:Y_A}).
By $A\subseteq A'$ (resp. $A\subset A'$), we mean that $A$ is included (resp. strictly included) in $A'$.
$\bigO$ (resp. $o$) will denote the usual Landau's `big O' (resp. `small o') notation.
By $x\rightarrow a^{-}$ (resp. $x\rightarrow a^+$) for $a\in\R$, we mean that we take the limit as $x$ increases (resp. decreases) to $a.$

\section{Set-indexed framework} \label{sec:set-indexed}

When one wants to study general $\T$-indexed processes $\mathtt Z=\{ \mathtt Z_t:t\in\T \}$ while retaining the intuition that $\T$ represents a kind of `time', it becomes natural to endow it with an order relation $\pre$.
Stating that $s\pre t$ for $s,t\in\T$ would then mean that `$s$ happens before $t$'. 
As annouced above, this setting turns out to be equivalent to considering processes $Z=\{ Z_A:A\in\A \}$ for a specific collection $\A$ of subsets of $\T$ (Proposition \ref{prop:lattice_correspondance}).
The advantage is that it becomes easier and more natural to define the increment process $\Delta Z$ over bigger collections than $\A$, extending the previous notion of rectangular increments for two-parameter processes. This framework is presented in Section \ref{subsection:indexing_collection_poset}.

Since our goal is to deal with the regularity of the sample paths of $\A$-indexed processes, we require a topology on $\A$, or equivalently on $\T$ (see (\ref{eq:d_T})), by means of a metric $d_\A$ whose properties are exposed in Section \ref{subsection:indexing_collection_geodesic}.

At last, we present in Section \ref{subsection:indexing_collection_finite_dim} restrictions on $\A$ to be finite-dimensional in some sense. This will later help to establish much needed sample paths properties (Section \ref{subsection:levy-ito_decomposition}) and martingale inequalities (Section \ref{subsection:ptw_regularity}).

\subsection{Indexing collection as a poset} \label{subsection:indexing_collection_poset}

\subsubsection{$\A$ and other classes}

In the following, $\PP(\T)$ denotes the collection of all subsets of $\T$ and for any subcollection $\D\subseteq\PP(\T),\,$ $\D(u)$ denotes the collection of all finite unions of elements in $\D.$
The following definition is inspired from \cite[Definition 2.1]{herbin_local_2016}, which is itself a careful selection of the required properties of \cite[Definition 1.1.1]{ivanoff_set-indexed_1999}.

\begin{de}[Indexing collection] \label{de:indexing_collection}
	A class $\A\subseteq\PP(\T)$ is an \emph{indexing collection on $\T$} if the following properties hold:
	\begin{enumerate}
		\item (Stability under intersections).
		$\varnothing\in\A,\,$ $\A$ and $\A\setminus\{\varnothing\}$ are closed under countable intersections.
		
		\item (Separability from above). 
		There exists an increasing sequence of finite subcollections $\A_n=\{ A_1^n,...,A_{k_n}^n \}\subseteq\A$ ($n\in\N$) closed under intersections such that given the functions $g_n:\A\rightarrow\A_n\cup\{\T\}$ defined by
		\[ \forall A\in\A,\quad g_n(A) \speq \bigcap_{\substack{ A'\in\A_n\cup\{\T\}: \\ A\subseteq A' }}A', \]
		the elements of $\A$ may be approximated as follows: for all $A\in\A,$\, $A = \ds\bigcap_{n\in\N}g_n(A).$ 
		
		\item (SHAPE condition). For any positive integer $n$ and $A,A_1,...,A_n\in\A,\,$ if $A\subseteq\bigcup_{i\ls n}A_i,\,$ then $A\subseteq A_j$ for some $j\ls n.$
		
		\item (TIP assumption).
		The map
		\[\begin{array}{ccc}
			\T 	& \longrightarrow 	& \A\setminus\{\varnothing\}	\\
			t 	& \longmapsto 		& A(t)=\bigcap_{n\in\N}A_n(t),
		\end{array}\]
		where $A_n(t) = \ds\bigcap_{\substack{A\in\A_n\cup\{\T\}: \\ t\in A}}A,$ is one-to-one. 
		
		For all $A\in\A,$ the unique point $t\in\T$ such that $A=A(t)$ is called the \emph{tip} of $A.$
	\end{enumerate}	
\end{de}

First of all, let us state that such a definition is actually equivalent to a couple of simple axioms. 
The proof is just straightforward abstract nonsense.
Whenever convenient, we will switch from one point of view to the other (see Examples \ref{exs:fundamental_discrete} and \ref{exs:fundamental_continuous}).

\begin{prop}[Correspondance $(\T,\A)\leftrightarrow(\T,\pre)$]	\label{prop:lattice_correspondance}
	Suppose that $\pre$ is a partial order on $\T$ satisfying the following two conditions:
	\begin{enumerate}
		\item (Countable subsemilattice).
		Any countable subset $\{t_n:n\in\N\}$ of $\T$ admits a joint minimum denoted by $\bigwedge_{n\in\N}t_n\in\T.$
		\item (Separability from above).
		There exists a countable subset $\mathcal D\subseteq\T$ such that for any $t\in\T$, there exists a non-increasing sequence $(t_n)_{n\in\N}$ in $\mathcal D$ such that $t=\bigwedge_{n\in\N}t_n.$
	\end{enumerate}	
	Denote for all $t\in\T,\,$ $A(t)=\{s\in\T:s\pre t\}.$
	Then $\A=\{ A(t):t\in\T \}\cup\{\varnothing\}$ is an indexing collection.
	Conversely, if one considers an indexing collection $\A$ on $\T,$ then the order relation $\pre$ given by
	\[ \forall s,t\in\T,\quad s\pre t ~\Longleftrightarrow~ A(s)\subseteq A(t) \]
	satisfies conditions 1 and 2 above.
\end{prop}

In the sequel, $\A$ stands for such an indexing collection and $\B=\sigma(\A)$ for the $\sigma$-algebra it generates.

\medskip

The previous definition of $\A$ will be made clearer as properties are further derived from those four assumptions, but let us still make some preliminary remarks about each of them.
\begin{enumerate}
	\item 
	Being closed under countable intersections (\ie $\A$ is a $\delta$-ring) ensures that a lot of measure-theoretic constructions apply.
	 
	Also having $\A\setminus\{\varnothing\}$ as a $\delta$-ring ensures the well-posedness of the TIP assumption. Moreover, it is necessary in order to prove that the cartesian product of indexing collections is still an indexing collection.
	
	Another reason to require stability under intersections would be that filtrations play a crucial role while studying processes such as martingales or Markov processes. In order to impose intuitive 'time consistent' relations between the $\sigma$-algebras in a set-indexed filtration, stability under intersections is required. For more details, see \cite{ivanoff_set-indexed_1999}.
	
	\item 
	One key element arising in the study of $\R_+$-indexed \emph{càdlàg} (right continuous with left limits) stochastic processes is the use of dyadics. 
	They are extremely useful to get results in the continuous case from their discrete alter egos. 
	
	In the second assumption, the class $\A_n$ indeed plays a role similar to the dyadics of order $n$ in that endeavour.
	In the litterature of set-indexed processes, one usually imposes some topological structure on $\T$ so that $A$ lies in the interior of $g_n(A)$ for all $n\in\N$ instead. This implies a `separability strictly from above', but we chose against it here since $\A$ will be endowed with a metric in Section \ref{subsection:indexing_collection_geodesic} so that there is no competition with another topology.
	
	\item 
	The SHAPE condition has been first introduced in \cite[Assumption 1.1.5]{ivanoff_set-indexed_1999} as a sufficient condition to ensure the existence of \emph{increment maps} (see Proposition-Definition \ref{propde:increment_map_linear_extension}). Since it turns out that such a condition is also necessary (even though the proof is not detailed here, Lemma \ref{lem:shape_equivalent} is actually an equivalent formulation of the SHAPE condition), it has been added to the definition. In lattice-theoretic vocabulary, this condition is known as \emph{join irreducibility}. 
	
	\item 
	The TIP assumption has also been introduced in \cite[Assumption 2.4.2]{ivanoff_set-indexed_1999} and draws a clear correspondance between general processes $\{\texttt Z_t:t\in\T\}$ and set-indexed processes $\{ Z_A:A\in\A\}$ such that $Z_\varnothing=0$ through the relation $\texttt Z_t = Z_{A(t)}$ for all $t\in\T.$ 
	
	This bijection is the key element that enables the correspondance of Proposition \ref{prop:lattice_correspondance}.
\end{enumerate}

\begin{rem}	\label{rem:changes_not_issue}
	Although what we mean by `indexing collection' differs from parts of the litterature, we will stress each time a result from the litterature is used and why the conclusions still hold in our case.	
	
	For instance, an indexing collection is usually supposed to be closed under arbitrary intersections, but this property also holds in our framework.
	Indeed, consider a subcollection $\A'\subseteq\A.$ 
	Then, separability from above tells that $\bigcap_{A\in\A'}A = \bigcap_{n\in\N}\bigcap_{A\in\A'\cap\A_n}A$ so it still belongs to $\A$ by stability under countable intersections.
	Likewise, for all $t\in\T,$\, $\bigcap_{A\in\A:t\in A}A=\bigcap_{n\in\N}\bigcap_{A\in\A_n:t\in A}A=A(t)$, which links back to the usual meaning of $A(t)$ in the litterature.
\end{rem}

In particular, $\T$ and $\A$ both have a global minimum:
\begin{equation}	\label{eq:global_minimum}
	0_\T \speq \bigwedge_{t\in\T}t \qquad\text{ and }\qquad \varnothing' \speq A(0_\T) \speq \bigcap_{A\in\A:A\neq\varnothing}A.
\end{equation}
The notation will be consistent with the usual $0$ whenever $\T$ has one.
Without loss of generality, we suppose that both $\varnothing$ and $\varnothing'$ belong to $\A_n$ for all $n\in\N.$

\medskip

Let us present several examples in order to better understand what Definition \ref{de:indexing_collection} entails.
Since it has been explained in Proposition \ref{prop:lattice_correspondance} that it is equivalent to consider an indexing collection $\A$ or an order relation $\pre$ on $\T$ verifying certain properties, we will make full use of it depending on which is the most pratical in the actual context. 

\begin{exs}[Discrete indexing collections] 	\label{exs:fundamental_discrete}
	Contrary to the litterature, both discrete and continuous indexing collections are still available at this level of generality. It is the topological assumptions added in Section \ref{subsection:indexing_collection_geodesic} that will put the discrete case aside. 
	Since the following examples are at most countable, any increasing sequence $(\A_n)_{n\in\N}$ of finite subcollections of $\A$ closed under intersections such that $\bigcup_{n\in\N}\A_n=\A$ works out for the separability from above condition in Definition \ref{de:indexing_collection}. So we just need to specify the partial order $\pre$ of Proposition \ref{prop:lattice_correspondance}.
	\begin{enumerate}[$\diamond$]
		\item $\T = \N^p$ (for an integer $p\gs1$) endowed with the ususal componentwise partial order yields an indexing collection where
		\[ \forall t\in\N^p,\quad A(t) = \llbracket0,t\rrbracket = \big\{ s\in\N^p : \forall i\in\llbracket1,N\rrbracket,~s_i\ls t_i \big\}. \]
		
		\item $\T = \Z^p,$ once divided into $2^p$ 'quadrants', may be endowed with an order relation as follows:
		\[ \forall s,t\in\Z^p,\quad s\pre t ~\Longleftrightarrow~ \Big\{ \forall i\in\llbracket1,p\rrbracket,~s_it_i\gs0 ~\text{ and }~ s_i\ls t_i \Big\}. \] 
		Remark that using the usual componentwise partial order on $\Z^p$ would not yield an indexing collection since
		\[ \bigcap_{n\in\N} \rrbracket-\infty,-n\rrbracket^p \speq \varnothing \]
		which is not in accordance with the fact that $\A\setminus\{\varnothing\}$ should be closed under countable intersections.
		
		\item A tree $\T\subseteq \{\varnothing\}\cup\bigcup_{p=1}^\infty(\N^*)^p$ (see Neveu's convention for trees introduced in \cite{neveu_arbres_nodate}) is endowed with a natural partial order for which
		\[ \forall t = t_1...t_p\in\T,\quad A(t) \speq \big\{ \varnothing \big\} \cup \big\{ t_1...t_i : 1\ls i\ls p \big\} \speq \big\{ \text{ancestors of } t \big\} \cup \big\{ t \big\}. \]
		One may also adapt this idea to finite graphs through the choice of a spanning tree.
		
	\end{enumerate}	 
\end{exs}

\begin{exs}[Continuous indexing collections]	\label{exs:fundamental_continuous}
	\ 
	\begin{enumerate}[$\diamond$]
		\item $\T=\R_+^p$ (or $[0,1]^p$) endowed with the usual componentwise partial order goes the same way as $\N^p.$ 
		To check the separability condition, one may set $\A_n$ as the collection of $A(t)$'s where each component $t_i$ of $t$ is a dyadic of order $n$ between $0$ and $n,$ \ie $t_i=k_i\,2^{-n}$ where $k_i\in\llbracket0,n2^n\rrbracket.$ \\
		More generally, the positive cone $\T=E_+$ of a separated Riesz space $(E,\pre)$ such that $\C^0([0,1])$ or $L^q(0,1)$ (see \cite{aliprantis_infinite_2006} for more details and examples) may be endowed with an indexing collection. 
		
		\item $\T=\R^p$ is easily endowed with an indexing collection by combining the ideas from $\Z^p$ and $\R_+^p.$
		
		\item The space of $[0,1]$-valued sequences $\T=[0,1]^\N$ endowed with the usual componentwise partial order also yields an indexing collection where the elements of $\A_n$ are the sets of the form 
		\[ \prod_{i=1}^n[0,d_i] \times\prod_{i=n+1}^\infty[0,1]  \]
		where $d_1,...,d_n$ are dyadics of order $n$ in $[0,1],$ \ie $d_i=k_i\,2^{-n}$ where $k_i\in\llbracket0,2^n\rrbracket$ for all $i\ls n.$
		
		\item The hypersphere $\T=\Sph^p$ may be endowed with an indexing collection through the spherical coordinates map $\Sph^p\ni t\mapsto\varphi(t)=(\varphi_1(t),...,\varphi_p(t))\in[0,\pi]^{p-1}\times[0,2\pi)$ and the usual componentwise order $\pre$ on $\R_+^p$ by
		\[ \forall t\in\Sph^p,\quad A(t) = \big\{ s\in\Sph^p : \varphi(s)\pre\varphi(t) \big\}. \]
		Using coordinate maps (and partially ordering them), one may also endow more general manifolds with indexing collections.
		
		\item An $\R$-tree $\T$ (see \cite[Definition 3.15]{evans_probability_2008}) rooted at some $\rho\in\T$ may be endowed with an indexing collection $\A$ whose elements are the \emph{geodesic segments} $\llbracket\rho,t\rrbracket$ for all $t\in\T$ as long as it is separated from above.
	\end{enumerate}
\end{exs}

\medskip

Let us now recall for the sake of further reference some additional classes of subsets of $\T$ based on the definition of $\A.$ Definitions \ref{de:other_class_C} through \ref{de:other_class_C^ell} are taken from \cite{ivanoff_set-indexed_1999} apart from the notation $C_n(t)$ which comes from \cite{herbin_set-indexed_2013} and the classes $\C_k$ which are introduced here for the first time.
 
\begin{de}[Increment class $\C$] \label{de:other_class_C}
	The class of \emph{increment sets} is given by
	\[ \C \speq \big\{ A\setminus U : A\in\A, U\in\A(u) \big\}. \]
	For any $k\in\N,$ the subclass $\C_k$ of \emph{$k$-increments} of $\C$ is given by
	\[ \C_k \speq \left\{ A_0\setminus\bigcup_{i=1}^{k+1}A_i : A_0,...,A_{k+1}\in\A \right\}. \]
\end{de}

One obviously has $ \A \subseteq \C \subseteq \C(u) \subseteq \B $
where each inclusion is strict in general. 
The classes $\C$ and $\C(u)$ are a semiring of sets and
a ring of sets respectively (see \cite[Definitions 1.8 and 1.9]{klenke_probability_2014}). Thus they are well-adapted to measure-theoretic constructions, which will be made clearer in Section \ref{subsubsec:increment_map}.
The class $\C$ is also a natural extension of the rectangular increments of $\R_+^2$-indexed processes.

The subclasses $\C_k$ will play an important role to characterize a dimensional property of $\A$ (see Section \ref{subsection:indexing_collection_finite_dim}). They also are used to define regularity criterion for set-indexed maps (see Section \ref{sec:regularity_criterion}). 
The particular case of $\C_0$ has been used in \cite{herbin_set-indexed_2013} to characterize increment stationarity for set-indexed Lévy processes (see Section \ref{subsection:construction_integral} for a definition).

\begin{de}[Extremal representation]	\label{de:extremal_representation_C}
	Any $C\in\C$ may be written as
	\[ C \speq A_0\setminus\bigcup_{i=1}^nA_i \]
	where $n\in\N$ and $A_0,...,A_n\in\A$ are such that for all $i,j\in\llbracket1,n\rrbracket,$ $A_i\subseteq A_0$ and $A_i\subseteq A_j$ implies $i=j.$
	This representation, called \emph{extremal representation of $C$}, is unique up to relabelling $A_1,...,A_n.$
\end{de}

This representation is mainly a consequence of the SHAPE condition. 
We refer to \cite[Assumption 1.1.5]{ivanoff_set-indexed_1999} and the following comments for more details.

\medskip

The last class and its usefulness have been remarked early and may be found in \cite[Assumption 1.1.7]{ivanoff_set-indexed_1999}.

\begin{de}[Left neighborhoods $\C^\ell$] \label{de:other_class_C^ell}	
	The class of \emph{left neighborhoods} is given by
	\[ \C^\ell = \bigcup_{n\in\N}\C^\ell(\A_n) \]
	where for all integer $n,$ $\C^\ell(\A_n)$ is the collection of all sets of the form
	$ A\setminus\bigcup_{\substack{A'\in\A_n: A'\subsetneq A }} A' $
	where $A\in\A_n.$
	
	For all $t\in\bigcup_{C\in\C^\ell(\A_n)}C,~$ let $C_n(t)$ denote the unique element of $\C^\ell(\A_n)$ containing $t.$ 
	
	For $t\notin\bigcup_{C\in\C^\ell(\A_n)}C,$ set $C_n(t)=\T.$
\end{de}

The class $\C^\ell(\A_n)$ is made of the 'indivisible' (\ie smallest for the inclusion), pairwise disjoint elements of $\C$ that one can create from $\A_n.$
$\C^\ell$ is a subclass of $\C$ with the nice feature of being able to 'zoom in' on any point $t\in\T,$ which will prove to be crucial to define the pointwise jumps of $X$ (see Definition \ref{de:well-defined_jumps}) and derive a Lévy-Itô representation out of it (see (\ref{eq:levy-ito_decomposition_X})). 
For that, remark that the surjectivity in the TIP assumption ensures that $C_n(t)\subseteq A_n(t)$ for all $n\in\N$ big enough, so $C_n(t)$ belongs to $\C^\ell(\A_n)$. In turn, injectivity says that $\lp C_n(t)\rp_{n\in\N}$ decreases to $\{t\}.$

Actually, we can say a bit more than this.
Namely, the class $\C^\ell$ is a \emph{dissecting system}, meaning that for any $t\neq t'$ in $\T,$ there are $C,C'\in\C^\ell$ such that $t\in C,$ $t'\in C'$ and $C\cap C'=\varnothing.$
To check it, one just needs to take $C=C_n(t)$ and $C'=C_n(t')$ for a big enough $n\in\N.$

\begin{ex}
	Consider $\T=\R_+^p$ and its indexing collection given in Examples \ref{exs:fundamental_continuous}. Then the elements of $\C^\ell(\A_n)$ are the 'hypercubes' 
	\[ \R_+^p\cap\prod_{i=1}^p(d_i-2^{-n},d_i] \]
	where the $d_i$'s are dyadics of order $n$ in $[0,n].$
\end{ex}

\subsubsection{Increment map and linear functional}	\label{subsubsec:increment_map}

We are on our way to define the integral $\ds\int_\T f\,\text dX$ where $f:\T\rightarrow\R$ is a deterministic map and $X$ is a specific set-indexed process.
As it has already be explained earlier, one advantage of the set-indexed setting is that $X$ may already be considered as a kind of cumulative distribution function of a an additive map $\Delta X = \{ \Delta X_U:U\in\C(u) \}.$ 
The goal of this section is to reach Proposition-Definition \ref{propde:increment_map_linear_extension} where $\Delta X$ is properly defined and its link with $\ds\int_\T f\,\text dX$ clarified.

\begin{de}[Simple functions] \label{de:simple_functions}
	The space of \emph{simple functions} is the linear subspace $\E$ of $\R^\T$ spanned by the indicator functions $\myindic{A}$ where $A\in\A.$
\end{de}

\begin{rem} \label{rem:1_C(u)_in_E}
	By the usual inclusion-exclusion formula, we know that for all $C=A_0\setminus \bigcup_{i=1}^nA_i\in\C,$
	\begin{equation}	\label{eq:inclusion-exclusion_formula}
	\myindic{C} \speq \myindic{A_0} - \sum_{i=1}^n (-1)^i \sum_{j_1<...<j_i} \myindic{A_0\cap A_{j_1}\cap...\cap A_{j_i}},
	\end{equation}
	hence $\myindic{C}$ belongs to $\E.$
	Since any element of $U\in\C(u)$ may be written as a disjoint union of elements in $\C,$ its corresponding indicator function will also belong to $\E.$
	Hence $\big\{ \myindic{U}:U\in\C(u) \big\} \subseteq \E.$
\end{rem}

The following straightforward result highlights an important aspect of $\C$: it enables to write simple functions as sums of pairwise disjoint indicators.
The writing is not unique, but it could be made so using the class $\C(u).$

\begin{prop}[$\C$-representation of simple functions] 	\label{prop:C-representation}
	Any simple function $f\in\E$ may be written as
	\[ f \speq \sum_{i=1}^n c_i\myindic{C_i} \]
	where $n\in\N,$ $c_1,...,c_n\in\R^*$ and $C_1,...,C_n\in\C$ are pairwise disjoint.
\end{prop}

\begin{lem}	\label{lem:shape_equivalent}
	The family $\big\{ \myindic{A}:A\in\A\setminus\{\varnothing\} \big\}$ is linearly independent in $\R^\T,$ and thus forms a basis of $\E.$
\end{lem}

\begin{proof}
	Assume the family is linearly dependent.
	We can write a dependence relation 
	$ \sum_{i=1}^n \alpha_i\myindic{A_i}=0 $
	where $n\in\N^*,$ $\alpha_1,...,\alpha_n\in\R^*$ and $A_1,...,A_n\in\A\setminus\{\varnothing\}$ are pairwise distinct.
	
	By writing for all $j\in\llbracket1,n\rrbracket,$
	\[ \myindic{A_j} \speq \sum_{\substack{i=1 \\ i\neq j}}^n \frac{\alpha_i}{\alpha_j}\myindic{A_i}, \]
	we deduce that
	\[ A_j ~\subseteq~ \bigcup_{\substack{i=1 \\ i\neq j}}^n A_i. \]
	By the SHAPE condition (see Definition \ref{de:indexing_collection}), we get that for all $j\in\llbracket1,n\rrbracket,$ there exists $i_j\in\llbracket1,n\rrbracket\setminus\{j\}$ such that $A_j\subset A_{i_j}$. Notice that the inclusion is strict since the $A_i$'s are pairwise distinct.
	
	Let us show that this brings a contradiction.
	Since $i_1\neq1,$ we might as well suppose that $i_1=2$ so that $A_1\subset A_2.$
	Suppose that $A_1\subset...\subset A_j,$ then $i_j>j$ since $i_j\neq j$ by definition and $i_j<j$ would yield the contradiction $A_j=A_{i_j}.$ Hence, up to relabelling, we may suppose that $i_j=j+1$ so that $A_1\subset...\subset A_{j+1}.$
	By iteration, we get $A_1\subset...\subset A_n,$ but this is a contradiction since $A_n$ must also be included in some distinct $A_i$ for $i<n.$
\end{proof}

From this lemma, we deduce the existence of an \emph{additive extension} $\Delta h$ to $\C(u)$ of any map $h:\A\rightarrow\R$ in the sense of \cite{ivanoff_set-indexed_1999}.
Previously, it was known that SHAPE is a sufficient condition to the existence of such extensions, but it is also necessary to ensure the existence of all such extensions (the proof will be omitted here, but relies mainly on Lemma \ref{lem:shape_equivalent}).
It is one reason why SHAPE has been included in Definition \ref{de:indexing_collection} contrary to \cite{ivanoff_set-indexed_1999} where the existence of an additive extension is supposed whenever needed, which so turns out to be an equivalent point of view.

\begin{propde}[Increment map and linear functional] \label{propde:increment_map_linear_extension}
	Consider a map $h:\A\rightarrow\R$ such that $h(\varnothing)=0.$
	\begin{enumerate}[$\diamond$]
		\item There exists a unique additive extension $\Delta h : \C(u)\rightarrow\R$ of $h,$ \ie such that $\restr{\Delta h}{\A}=h$ and for all pairwise disjoint $U_1,U_2\in\C(u),\,$ $\Delta h(U_1\sqcup U_2)=\Delta h(U_1)+\Delta h(U_2).$

		The map $\Delta h$ is called the \emph{increment map of $h.$}
		
		\item There exists a unique linear map $\mathbf h:\E\rightarrow\R$ such that $\mathbf h(\myindic{A})=h(A)$ for all $A\in\A.$ 
		Moreover, $\mathbf h(\myindic{U})=\Delta h(U)$ for all $U\in\C(u).$
		
		The map $\mathbf h$ is called the \emph{linear functional associated with $h.$}
	\end{enumerate}
\end{propde}

\begin{proof}
	The existence and unicity of $\mathbf h$ are but a direct consequence of Lemma \ref{lem:shape_equivalent}.
	According to Remark \ref{rem:1_C(u)_in_E}, we may define $\Delta h(U)=\mathbf h(\myindic{U})$ for all $U\in\C(u),$ which is obviously additive. 
	
	It remains to prove uniqueness by induction.
	Suppose that such $\Delta h$ exists.
	First, remark that $\Delta h$ is uniquely determined on $\A$ since we must have $\restr{\Delta h}{\A}=h.$
	Suppose now that for a fixed integer $k\in\N,$ $\Delta h$ is uniquely determined on the class $\C_{k-1}$ (given in Definition \ref{de:other_class_C}) where we set $\C_{-1}=\A.$
	Let us consider an element $C_k=A_0\setminus\bigcup_{i=1}^{k+1}A_i\in\C_k$ and show that the value $\Delta h(C_k)$ is determined by $\restr{\Delta h}{\C_{k-1}}.$
	
	Denote $C_{k-1}=A_0\setminus\bigcup_{i=1}^kA_i.$
	Since $C_{k-1} = C_k\sqcup(A_{k+1}\cap C_{k-1}),$ the additivity of $\Delta h$ tells us that
	\[ \Delta h(C_k) \speq \Delta h(C_{k-1}) - \Delta h(A_{k+1}\cap C_{k-1}) \]
	where both $C_{k-1}$ and $A_{k+1}\cap C_{k-1}$ actually belong to $\C_{k-1}.$
	Hence $\Delta h(C_k)$ is uniquely determined by the induction hypothesis.
	
	Thus $\Delta h$ is uniquely determined on $\C=\A\cup\bigcup_{k\in\N}\C_k,$ but since any element of $\C(u)$ may be written as a disjoint union of elements of $\C,$ $\Delta h$ is uniquely determined on $\C(u)$ by additivity. 
\end{proof}

\subsection{Indexing collection as a metric space}	\label{subsection:indexing_collection_geodesic}

\subsubsection{Metric $d_\A$ on $\A$}

Since we want to have a look at the regularity of set-indexed processes, we require a (pseudo-)metric $d_\A$ on $\A.$ 
Moreover, we want $d_\A$ to interact well with the already existing order structure of the indexing collection $\A.$
A similar approach has been undertaken in \cite{herbin_local_2016} to obtain a set-indexed version of Kolmogorov-Chentsov's regularity theorem, but here we do not require any quantitative hypothesis.

\medskip

In the following, for any $A\subseteq A'$ in $\A,\,$ $[A,A']$ will denote the set $\big\{ A''\in\A:A\subseteq A''\subseteq A' \big\}.$

Moreover, whenever a metric $d_\A$ is given on $\A$, the TIP bijection of Definition \ref{de:indexing_collection} automatically induces a metric $d_\T$ on $\T$ by
\begin{equation}	\label{eq:d_T}
	\forall s,t\in\T,\quad d_\T(s,t) \speq d_\A(A(s),A(t)) 
\end{equation}
Conversely, whenever $d_\T$ is given, $d_\A$ is also characterized by (\ref{eq:d_T}). This pushes further the correspondance established by Proposition \ref{prop:lattice_correspondance}.
We underline the fact that even though $\T$ might sometimes by endowed with a natural metric, if $d_\A$ is given, then $d_\T$ will always be defined by (\ref{eq:d_T}) and vice versa.

\begin{de}[Set-indexed compatible metric] \label{de:compatible_metric}
	A metric $d_\A$ on $\A$ is said to be \emph{(set-indexed) compatible} if the following properties hold:
	\begin{enumerate}
		\item (Contractivity). 
		For any $A,A',A''\in\A,\,$ $d_\A(A\cap A'',A'\cap A'') \ls d_\A(A,A').$
		
		\item (Outer continuity).
		For any $(A_n)_{n\in\N}$ non-increasing sequence in $\A,\,$ $d_\A(A_n,A)\rightarrow0$ as $n\rightarrow\infty$ where $A=\bigcap_{n\in\N}A_n.$
		
		\item (Shrinking mesh property).
		The diameter of the left-neighborhoods tends to 0, \ie
		\[ \max_{C\in\C^\ell(\A_n)} \text{diam}(C) ~\longrightarrow~ 0 \spas n\rightarrow\infty \]
		where $\C^\ell(\A_n)$ has been given in Definition \ref{de:other_class_C^ell} and $\text{diam}(C)=\sup\{ d_\T(s,s'):s,s'\in C \}.$
				
		\item (Midpoint property).
		For any $A\subseteq A'$ in $\A,$ there exists $A''\in[A,A']$ such that $d_\A(A,A'')=d_\A(A'',A')=d_\A(A,A')/2.$
		
	\end{enumerate}
\end{de}

In the following, let us consider such a compatible metric $d_\A$.  
In particular, the metric $d_\T$ given by (\ref{eq:d_T}) endows $\T$ with a topology, relating back to the usual definition of indexing collection given in \cite{ivanoff_set-indexed_1999} where $\T$ is supposed to be a topological space from the start.

\medskip

Open balls for $d_\A$ and $d_\T$ will be denoted by $B_\A(A,\rho)$ and  $B_\T(t,\rho)$ respectively. 

\medskip

Let us briefly comment on this definition.
Assumptions 1 and 2 may be found in \cite[Definition 2.2]{herbin_local_2016} and ensure that $d_\A$ is compatible with the order structure on $\T.$
Assumption 3 is a way to further push that compatibility to the `meshes' $\A_n$. In particular, it is a qualitative counterpart to the quantitative hypothesis \cite[Assumption $\mathcal H_{\underline\A}$]{herbin_local_2016}. Remark that closely related assumptions may also be found in \cite{ivanoff_compensator_2007} where a lattice-indexed Poisson process is studied.
As for Assumption 4, it is an hypothesis reminiscent of the standard setting of geodesic spaces and its relevance will be further discussed after Proposition \ref{prop:geodesic}. 

\begin{exs} \label{exs:compatible_metric} \ 
	\begin{enumerate}[$\diamond$]
		\item In the case where $\T$ itself is endowed with a pseudo-metric $d,$ instead of directly defining $d_\A$ through $d_\T=d,$ we may also consider the induced Hausdorff metric given by
		\[ \forall A,A'\in\A\setminus\{\varnothing\},\quad d_{\mathcal H}(A,A') \speq \inf\big\{ \eps>0 : A'\subseteq A^\eps \text{ and } A\subseteq(A')^\eps \big\} \]
		where $A^\eps=\{ t\in\T : d(t,A)\ls\eps \}$ and $d_{\mathcal H}(\varnothing,\varnothing')=0$ by convention.
		
		$d_{\mathcal H}$ is always contractive, outer continuous if $d$ is with respect to the sub-semilattice $(\T,\pre).$
		Similarly, the shrinking mesh and midpoint properties may be formulated in terms of $d.$
		
		\item If $m$ is a measure on $(\T,\B)$ (recall that $\B=\sigma(\A)$) such that $m(A)<\infty$ for all $A\in\A,$ then we may also consider 
		\[ \forall A,A'\in\A,\quad d_m(A,A') \speq m(A\triangle A') \]
		where $A\triangle A'=(A\setminus A')\cup(A'\setminus A)$ is the symmetric set difference.
		$d_m$ is always contractive and outer continuous.
		A local version of the shrinking mesh property (which is what we truly need, but is harder to state) is true as long as the midpoint property is by mimicking the proof of \cite[Lemma 5.1.6]{ivanoff_set-indexed_1999}.
		The midpoint property works for all Examples \ref{exs:fundamental_continuous} as long as $m$ is absolutely continuous with respect to the `Lebesgue measure'.
	\end{enumerate}
\end{exs}

\subsubsection{Geodesics in $\A$}

Let us now delve into the consequences of the midpoint property of $d_\A $ (Definition \ref{de:compatible_metric}). They will prove to be important when establishing the Lévy-Itô decomposition in Section \ref{subsection:levy-ito_decomposition} and when giving upper bounds on the Hölder regularity in Section \ref{subsection:oji_function}.

\medskip

A \emph{(constant speed) geodesic} is a map $\gamma:I\rightarrow\A$ where $I$ is an interval of $\R$ such that
	\[ \forall x,y\in I,\quad d_\A(\gamma(x),\gamma(y)) \speq v\,|x-y| \]
where $v>0$ is a constant, called the \emph{speed} of the geodesic.	

\medskip

The following proposition should be compared to \cite[Lemma 5.1.6]{ivanoff_set-indexed_1999} which ensures the existence of so-called \emph{flows}. 
Those flows embody the intrisic `continuous' quality of classical indexing collections since they constitute continuous paths between elements of $\A.$ 
Flows are also a precious link with the usual one-dimensional theory: from a set-indexed process $Z$ and a flow $\gamma:\R_+\rightarrow\A,$ one may define a process $\left\{ \texttt Z^\gamma_t : t\in\R_+ \right\},$ called the \emph{projection of $Z$ along $\gamma$}, by setting $\texttt Z^\gamma_t = Z_{\gamma(t)}$ for all $t\in\R_+$ (see \cite{herbin_local_2016, ivanoff_set-indexed_1999} for more details and applications).
 
Here, our hypotheses on $d_\A$ enable a similar approach through increasing geodesics which share a lot of properties with those flows.
Compared to \cite{ivanoff_set-indexed_1999}, our approach seemed simpler to expose since we separated order properties (Section \ref{subsection:indexing_collection_poset}) from topological ones (Section \ref{subsection:indexing_collection_geodesic}). 
Moreover, as a consequence of the previously mentioned Lemma 5.1.6, the midpoint property (which is the key to Proposition \ref{prop:geodesic}) is true in the classical set-indexed setting for $d_\A=d_m$ and a $m$ a Radon measure, so our exposition is also more general.

\begin{prop}[Existence of increasing geodesics] \label{prop:geodesic}
	For any $A_0,A_1\in\A$ such that $A_0\subset A_1$ and $d_\A(A_0,A_1)>0,$ there exists an increasing geodesic $\gamma:[0,1]\rightarrow\A$ such that $\gamma(0)=A_0$ and $\gamma(1)=A_1.$
\end{prop}

\begin{proof}
	Denote for all non-negative integer $n$ the set $\mathcal D_n=\{ k\,2^{-n}:0\ls k\ls 2^n \}$ of dyadics of order $n$ in $[0,1]$ as well as $\mathcal D=\bigcup_{n\in\N}\mathcal D.$
	Directly iterating on the midpoint property in Definition \ref{de:compatible_metric} yields a collection $\{ A_d : d\in\mathcal D \}$ such that $d<d'$ implies $A_d\subset A_{d'}$ and
	\[ \forall d\in\mathcal D_{n+1}\setminus\mathcal D_n,\quad d_\A(A_d,A_{d\pm2^{-n}}) \speq \frac{1}{2}d_\A(A_{d-2^{-n}},A_{d+2^{-n}}). \]
	Then we define a map $\gamma:[0,1]\rightarrow\A$ as follows:
	\[ \forall x\in[0,1],\quad \gamma(x) \speq \bigcap_{\substack{d\in\mathcal D: \\ x\ls d}}A_d. \]
	
	The function $\gamma$ is increasing.
	Moreover, the iteration directly proves that for all $d,d'\in\mathcal D,\,$ $d_\A(\gamma(d),\gamma(d'))=d_\A(A_0,A_1)|d-d'|.$ 
	Outer continuity allows to extend this relation for all $x,y\in[0,1].$
\end{proof}

Before moving on, let us prove a useful application of increasing geodesics to the shrinking mesh property.

\begin{prop}[Weak shrinking mesh property]	\label{prop:weak_shrinking_mesh}
	The `mesh size' 
	\[ \delta_n \speq \max\big\{ d_\A(A,A') : A,A'\in\A_n,~ A \text{ maximal proper subset of } A' \big\} \longrightarrow0 \spas n\rightarrow\infty \] 
	where by `$A$ being a maximal proper subset of $A'$', we mean that $A\subset A'$ and there is no $A''\in\A_n$ such that $A\subset A''\subset A'$ (recall that $\subset$ stands for the strict inclusion).
\end{prop}

\begin{proof}
	Fix $n\in\N$ and consider $A,A'\in\A_n$ such that $A$ is a maximal propert subset of $A'$ and $\delta_n=d_\A(A,A').$ Then consider an increasing geodesic $\gamma:[0,1]\rightarrow\A$ such that $\gamma(0)=A$ and $\gamma(1)=A'.$
	
	Since for all $x>0,$ the tips (Definition \ref{de:indexing_collection}) of $\gamma(x) $ and $A'$ both belong to the same element in $\C^\ell(\A_n)$, we have
	\[ \forall x>0,\quad d_\A(\gamma(x),A') ~\ls~ \max_{C\in\C^\ell(\A_n)} \text{diam}(C). \] 
	Since $\gamma$ is continuous, letting $x\rightarrow0^+$ yields 
	\[ d_\A(A,A') ~\ls~ \max_{C\in\C^\ell(\A_n)} \text{diam}(C). \]
	The result then follows by letting $n\rightarrow\infty.$
\end{proof}

\subsubsection{Metric $d_\C$ on $\C$}

The metric $d_\A$ on $\A$ may be naturally extended to a metric $d_\C$ on $\C$ with corresponding interesting properties. In Section \ref{sec:regularity_criterion}, this metric is used to give an equivalent definition of the set-indexed Hölder exponents defined in \cite{herbin_local_2016}. 

\begin{de}[Metric $d_\C$ on $\C$]	\label{de:d_C}
	For any $C=A_0\setminus\bigcup_{i=1}^nA_i$ and $C'=A'_0\setminus\bigcup_{j=1}^{n'}A'_j$ in $\C$ written with their extremal representations given in Definition \ref{de:extremal_representation_C}, denote by $d_\C(C,C')$ the Hausdorff distance between the sets $\{ A_0,...,A_n \}$ and $\{ A'_0,...,A'_{n'} \},$ \ie
	\[ d_\C(C,C') \speq \max\left\{ \max_{0\ls i\ls n}\min_{0\ls j\ls n'}d_\A(A_i,A'_j),\, \max_{0\ls j\ls n'}\min_{0\ls i\ls n}d_\A(A_i,A'_j) \right\}. \]
	$(\C,d_\C)$ is a (pseudo-)metric space for which the canonical injection $(\A,d_\A)\hookrightarrow(\C,d_\C)$ is an isometry, \ie ${d_\C}(A,A')=d_\A(A,A')$ for any $A,A'\in\A.$	
\end{de}

In order to comprehend what is going on for $d_\C,$ the minimum is here in order to 'match' $A_i$ with the closest $A'_j$ and vice versa while the maximum takes the total error into account for the best matching.
$d_\C$ is well-defined since the extremal representation of an element of $\C$ is unique due to Definition \ref{de:extremal_representation_C}.

\medskip

The next lemma basically tells that all constitutive elements of $C_n(t)$ (Definition \ref{de:other_class_C^ell}) converge to $A(t)$ as $n$ tends to infinity and sheds a new light on the metric $d_\T$ given by (\ref{eq:d_T}).

\begin{lem} \label{lem:d_C}
	For all $t\in\T,\,$ $d_\C(C_n(t),A(t))\rightarrow0$ as $n\rightarrow\infty.$
	Moreover, 
	\begin{equation}	\label{eq:d_T-d_C_link}
		\forall s,t\in\T,\quad d_\T(s,t) \speq \lim_{n\rightarrow\infty}d_\C( C_n(s),C_n(t) ).
	\end{equation}
\end{lem}

\begin{proof}
	Let $t\in\T$ and $\eps>0.$
	By outer continuity of $d_\A,$ there exists an integer $n_0$ such that
	\begin{equation} 	\label{eq:d_C_bound1}
		\forall n\gs n_0,\quad A_n(t)\in\A_n \spand d_\A(A_n(t),A(t)) ~\ls~ \eps/2.
	\end{equation}
	According to Proposition \ref{prop:weak_shrinking_mesh}, there exists $n_1\gs n_0$ such that
	\begin{equation}	\label{eq:d_C_bound2}
		\forall n\gs n_1,\quad \delta_n\ls\eps/2
	\end{equation}
	where $\delta_n$ is the one from that very proposition.
	
	By (\ref{eq:d_C_bound1}), for all $n\gs n_0,\,$ $C_n(t)\in\C^\ell.$
	In particular, we may write its extremal representation $C_n(t)=A_n^0\setminus\bigcup_{i=1}^{j_n}A_n^i.$
	By (\ref{eq:d_C_bound2}), we get
	\begin{equation} 	\label{eq:d_C_bound3} 
		\forall n\gs n_1,\,\forall i\ls j_n,\quad d_\A(A_n(t),A_n^i) ~\ls~ \eps/2. 
	\end{equation}
	Hence it follows that for all $n\gs n_1,$
	
	\vspace{-0.3cm}
	\begin{spacing}{1.3}
	\[\begin{array}{rcll}
		d_\C(C_n(t),A(t))
		& ~\ls~	& d_\C(C_n(t),A_n(t)) \+ d_\A(A_n(t),A(t)) 	
		& ~\text{ by triangle inequality,}					\\
		& ~\ls~ & \delta_n \+ d_\A(A_n(t),A(t))		
		& ~\text{ by definition of $d_\C$ and } \delta_n, 	\\
		& ~\ls~	& \eps 
		& ~\text{ by (\ref{eq:d_C_bound1}) and (\ref{eq:d_C_bound3}).}
	\end{array}\]
	\end{spacing}
	\vspace{-0.2cm}
	
	Hence $d_\C(C_n(t),A(t))\rightarrow0$ as $n\rightarrow\infty.$
	
	\medskip
	
	Only (\ref{eq:d_T-d_C_link}) remains to prove.
	Let $s,t\in\T.$ Remark that $C_n(s)$ and $C_n(t)$ both belong to $\C$ for all $n$ big enough, say $n\gs n_2.$ 
	Then we can write for all $n\gs n_2,$
	\[ d_\T(s,t) \speq d_\C(A(s),A(t)) ~\ls~ d_\C(A(s),C_n(s)) \+ d_\C(C_n(s),C_n(t)) \+ d_\C(C_n(t),A(t)). \]
	Taking lower limits yields $d_\T(s,t) \ls \liminf_{n\rightarrow\infty}d_\C(C_n(s),C_n(t)).$ 
	Conversely, for all $n\gs n_2,$
	\[ d_\C(C_n(s),C_n(t)) ~\ls~ d_\C(C_n(s),A(s)) \+ d_\T(s,t) \+ d_\C(C_n(t),A(t)). \]
	Taking upper limits yields $\limsup_{n\rightarrow\infty}d_\C(C_n(s),C_n(t)) \ls d_\T(s,t).$ 
	The limit (\ref{eq:d_T-d_C_link}) follows.
\end{proof}

\subsubsection{Divergence $\qd$ on $\T\times\A$}	\label{subsubsec:divergence}

Our main focus being the study of the integral $Y_A=\ds\int_A f\,\text dX$ of a function $f$, defined on $\T$, against a set-indexed process $X$, defined on $\A$, it introduces an interaction between points $t\in\T$ and sets $A\in\A$ that the pseudo-metric $d_\A$, or equivalently $d_\T,$ fails to capture.
The goal of this part is to shed some light and put words on such phenomenon, which we illustrate in the case where $\T=\R_+^2$ from Examples \ref{exs:fundamental_continuous}. 

\medskip

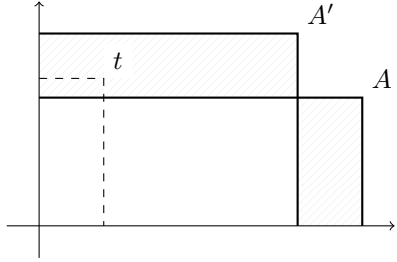
\begin{figure}[!h]
\centering
	\begin{tikzpicture}[scale=.85]	
		\draw (5,2)		node[above right] 	{$A$};
		\draw (4,3)		node[above right] 	{$A'$};
		\draw [->]		(-.5,0)--(5.5,0);
		\draw [->]		(0,-.5)--(0,3.5);	
		\draw [dashed]	(0,2.3)--(1,2.3)--(1,0);
		\draw [thick]	(0,2)--(5,2)--(5,0);
		\draw [thick]	(0,3)--(4,3)--(4,0);
		\fill [pattern=north east lines,opacity=.2] (0,2)--(0,3)--(4,3)--(4,2)--(0,2);
		\fill [pattern=north east lines,opacity=.2] (4,0)--(4,2)--(5,2)--(5,0)--(4,0);
		\draw (1,2.3)	node[above right, fill=white] 	{$t$};
	\end{tikzpicture}
	
	\caption{$t\in A\triangle A'$ where $d_\A(A,A')$ is much smaller than $d_\A(A,A(t))$.}
	\label{fig:symmetric_difference}
\end{figure}
	
Even if more details will be given in Section \ref{sec:regularity_criterion}, let us say for now that we will be interested in bounding the increments of $Y$ around some $A\in\A$ taking the form $Y_A-Y_{A'}$ where $A'$ is close to $A$ for $d_\A.$
Using the additive extension from Proposition \ref{propde:increment_map_linear_extension}, we have
\[ Y_A-Y_{A'} \speq \lp \Delta Y_{A\setminus A'}+Y_{A\cap A'} \rp - \lp \Delta Y_{A'\setminus A}+Y_{A'\cap A} \rp \speq \Delta Y_{A\setminus A'} - \Delta Y_{A'\setminus A}. \]

This tends to indicate that points $t\in A\triangle A'$ (\ie the hatched region in Figure \ref{fig:symmetric_difference}) do have an influence on the increment $Y_A-Y_{A'}$ where $d_\A(A,A')$ is small but $d_\A(A,A(t))$ might not be.
Hence the need to express `how close' such $t$ are to $A.$

\begin{de}[Victiny $\V$ and divergence $\qd$]	\label{de:divergence}
	For all $A\in\A,$ $t\in\T$ and $\rho>0,$ define
	\[ \V(A,\rho) \speq \bigcup_{A'\in B_\A(A,\rho)} \hspace{-.2cm} \big(A\triangle A'\big) \qquad\text{and}\qquad \qd(t,A) \speq \inf\Big\{ \rho>0 : t\in\V(A,\rho) \Big\}  \]
	with the convention $\V(A,\rho)=\varnothing$ for $\rho\ls0.$ 
	
	$\V(A,\rho)$ is called the \emph{victiny of $A$ of size $\rho$} and $\qd(t,A)$ the \emph{divergence between $t$ and $A$}.
\end{de}

This definition naturally yields two notions of `open balls' for $t\in\T,\,$ $A\in\A$ and $\rho>0$:

\vspace{-1cm}
\begin{spacing}{1.5}
\begin{eqnarray}
	\label{eq:victiny}
	\V(A,\rho) 	& = & \big\{ s\in\T : \qd(s,A)<\rho \big\},		\\
	\label{eq:dual_victiny}
	\V'(t,\rho)	& = & \big\{ A'\in\A : \qd(t,A')<\rho \big\}. 
\end{eqnarray}
\end{spacing}
\vspace{-.4cm}

where we check that (\ref{eq:victiny}) is coherent with Definition \ref{de:divergence}. $\V'(t,\rho)$ will be called \emph{dual victiny of $t$ of size $\rho$}.

\begin{ex} 	\label{ex:balls_qd}
	Although $\qd(t,A)=d_\A(A(t),A)$ in the case where $\T=\R_+$ (or a tree more generally, see Examples \ref{exs:fundamental_continuous}), other behaviors start to appear in higher-dimensional examples.

	When $\T=\R_+^2$ and $d_\T=d_2$ is the usual euclidean distance, the victinies are illustrated in Figure \ref{fig:victiny}.
	
\newpage

\begin{figure}[!h]
\centering
	\begin{tikzpicture}[scale=.85]	
		\draw [->]				(-.5,0)--(5.5,0);
		\draw [->]				(0,-.5)--(0,3.5);
		\draw [<->]				(2.7,2.9)--(4.3,2.9);
		\draw (3.5,2.9) 		node[above]				{$2\rho$};
		\draw [densely dotted]	(2.7,1.2)--(2.7,2.9);
		\draw [densely dotted]	(4.3,2)--(4.3,2.9);
		\draw [<->]				(4.4,1.2)--(4.4,2.8);
		\draw (4.4,2) 			node[right]				{$2\rho$};
		\draw [densely dotted]	(2.7,1.2)--(4.4,1.2);
		\draw [densely dotted]	(3.5,2.8)--(4.4,2.8);
		\fill [pattern=north east lines,opacity=.2] (4.3,0)--(4.3,2)--(4.3,2) arc (0:90:.8)--(3.5,2.8)--(0,2.8)--(0,1.2)--(2.7,1.2)--(2.7,0)--cycle;
		\fill [pattern=north west lines,opacity=.2] (3.5,2) circle (.8);
		\draw [thick]			(0,2.8)--(3.5,2.8);
		\draw [thick]			(4.3,2)--(4.3,0);
		\draw [thick]			(4.3,2) arc (0:90:.8);
		\draw [thick]			(0,1.2)--(2.7,1.2)--(2.7,0);
		\draw (3.5,2)			node[fill=white,above]	{$t$} node {$\times$};
		\draw [dashed,thick]	(0,2)--(3.5,2)--(3.5,0);
		\draw (2.5,-1)			node					{$\V(A(t),\rho)$ (hatched) and $B_\T(t,\rho)$ (crossed)};
		
		\draw [->]				(8.3,0)--(14.3,0);
		\draw [->]				(8.8,-.5)--(8.8,3.5);
		\draw[<->]				(11.4,.6)--(9.8,.6);
		\draw (10.6,.6) 		node[below]	{$2\rho$};
		\draw [densely dotted]	(11.4,2.3)--(11.4,.6);
		\draw [densely dotted]	(9.8,1.5)--(9.8,.6);
		
		\draw[<->]				(9.7,2.3)--(9.7,.7);
		\draw (9.7,1.5) 		node[left]	{$2\rho$};
		\draw [densely dotted]	(9.7,2.3)--(11.4,2.3);
		\draw [densely dotted]	(11.4,.7)--(9.7,.7);
		
		\fill [pattern=north west lines,opacity=.2] (10.6,1.5) circle (.8);
		\fill [pattern=north east lines,opacity=.2] (14.3,.7)--(10.6,.7) arc (270:180:.8)--(9.8,3.5)--(11.4,3.5)--(11.4,2.3)--(14.3,2.3)--cycle;
		\draw [thick]			(13.3,.7)--(10.6,.7) arc (270:180:.8)--(9.8,3);
		\draw [thick,dashed]	(14.3,.7)--(13.3,.7);
		\draw [thick,dashed]	(9.8,3.5)--(9.8,3);
		\draw [thick]			(13.3,2.3)--(11.4,2.3)--(11.4,3);
		\draw [thick,dashed]	(14.3,2.3)--(13.3,2.3);
		\draw [thick,dashed]	(11.4,3.5)--(11.4,3);
		\draw (10.6,1.5)		node[fill=white,above]	{$t$} node {$\times$};
		\draw (11.3,-1)			node					{$\V'(t,\rho)$ (hatched) and $B_\T(t,\rho)$ (crossed)};
	\end{tikzpicture}
	
	\caption{Victiny and dual victiny for $(\T,d_\T)=(\R_+^2,d_2)$.}
	\label{fig:victiny}
\end{figure}
\end{ex}

\vspace{-.5cm}

Some might wonder whether through the TIP bijection, one could obtain a metric on $\T$ with the formula $(s,t)\mapsto\qd(s,A(t)).$
However that is not the case since both symmetry and triangle inequality fail in general.
Indeed, if symmetry was true, then we would have $\V(A(t),\rho)=\V'(t,\rho)$ but Figure \ref{fig:victiny} strongly suggests that is generally not the case.
As for triangle inequality, Figure \ref{fig:qd_no_distance} illustrates a case where while both $\qd(s,A(t))$ and $\qd(t,A(u))$ are small --- both $s$ and $t$ are contained in corresponding small victinies --- $\qd(s,A(u))$ is big, so $\qd(s,A(u))\ls\qd(s,A(t))+\qd(t,A(u))$ cannot hold.

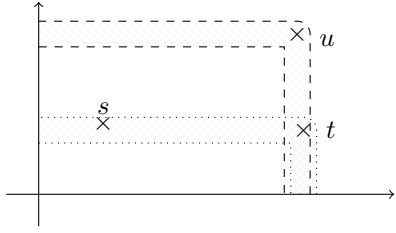
\begin{figure}[!h]
\centering
	\begin{tikzpicture}[scale=.85]	
		\draw [->]		(-.5,0)--(5.5,0);
		\draw [->]		(0,-.5)--(0,3);
		
		\draw (1,1.1)	node {$\times$} node[above] 	{$s$};	

		\draw (4.3,1)	node[right] 	{$t$};
		\draw (4.1,1)	node {$\times$};

		\draw [dotted]	(4.3,0)--(4.3,1)--(4.3,1) arc (0:90:.2)--(0,1.2);
		\draw [dotted]	(0,0.8)--(3.9,0.8)--(3.9,0);
		\fill [pattern=north east lines,opacity=.2] (4.3,0)--(4.3,1)--(4.3,1) arc (0:90:.2)--(0,1.2)--(0,0.8)--(3.9,0.8)--(3.9,0)--cycle;
		
		\draw (4.2,2.4)	node[right] 	{$u$};
		\draw (4,2.5)	node {$\times$};

		\draw [dashed] 	(4.2,0)--(4.2,2.5)--(4.2,2.5) arc (0:90:.2)--(0,2.7);
		\draw [dashed]	(0,2.3)--(3.8,2.3)--(3.8,0);
		\fill [pattern=north west lines,opacity=.2] (4.2,0)--(4.2,2.5)--(4.2,2.5) arc (0:90:.2)--(0,2.7)--(0,2.3)--(3.8,2.3)--(3.8,0)--cycle;
	\end{tikzpicture}

	\caption{The hatched regions represent victinies of $A(t)$ and $A(u)$.}
	\label{fig:qd_no_distance}
\end{figure}

Before moving on, let us briefly study $\V$ and $\qd$ in order to see that, even though there is no metric structure in general, those objects still retains some related nice geometric properties.

\begin{prop}[Properties of $\V$]	\label{prop:property_V}
	The following properties hold:
	\begin{enumerate}
		\item (Victinies behave like open balls).
		For all $A,A'\in\A$ and $\rho>0,$
		\begin{equation} 	\label{eq:prop_V_open_ball_1}
		\bigcup_{\rho'<\rho}\V(A,\rho') \speq \V(A,\rho) ~\subseteq~ \bigcap_{\rho'>\rho}\V(A,\rho')
		\end{equation}
		and 
		\begin{equation}	\label{eq:prop_V_open_ball_2}
			\V(A',\rho-d_\A(A,A')) ~\subseteq~ \V(A,\rho) ~\subseteq~ \V(A',\rho+d_\A(A,A')).
		\end{equation}
	
		\item (Discretization of the victiny).
		For all $A\in\A,$ $\rho>0$ and $n\in\N,$ denote
		
		\vspace{-.2cm}
		\begin{equation}	\label{eq:V_n}
			\V_n(A,\rho) \speq \bigcup_{\substack{\overline A,\underline A\in \A_n\cap B_\A(A,\rho): \\ \underline A\subset\overline A }} \hspace{-.2cm} \big(\overline A\setminus\underline A\big) \speq \overline V_n(A,\rho) \setminus \underline V_n(A,\rho)
		\end{equation}						
		\vspace{-.4cm}
		
		where $\overline V_n(A,\rho)$ (resp. $\underline V_n(A,\rho)$) is the union (resp. intersection) of all maximal (resp. minimal) elements for $\subseteq$ in $\A_n\cap B_\A(A,\rho).$
		Then
		\begin{equation}	\label{eq:V=U_nV_n}
			\V(A,\rho) ~=~ \bigcup_{n\in\N}\V_n(A,\rho). 
		\end{equation}
	\end{enumerate}	
\end{prop}

$\V_n(A,\rho)$ should be seen as a discretized version of $\V(A,\rho)$ and will prove to be useful twice in this article: for Lemma \ref{lem:generic_configuration_line} which is a step to deduce an upper bound on the Hölder regularity of $Y$ and for a set-indexed version of Doob's maximal inequality (Theorem \ref{theo:cairoli_inequality}) which will be used to give a lower bound for the same regularity.

\begin{proof}
	Let us fix $A,A'\in\A,$ $n\in\N$ and $\rho>0.$
	\begin{enumerate}  
		
		\item The relation (\ref{eq:prop_V_open_ball_1}) is a straightforward consequence of the definition.
		
		Let us prove the second inclusion of (\ref{eq:prop_V_open_ball_2}).
		Consider $A''\in B_\A(A,\rho).$
		Then
		
		\vspace{-.7cm}
		\begin{spacing}{1.5}
		\[\begin{array}{rcc}
			A\triangle A''
			& =		 		& \big( A\setminus A'' \big) \cup \big( A''\setminus A\big) 			\\
			& ~\subseteq~	& \big[\big(A\setminus A'\big) \cup \big(A'\setminus A''\big)\big] \cup \big[\big(A''\setminus A'\big) \cup \big(A'\setminus A\big)\big] 	\\
			& =				& \hspace{.3cm} \big(A\triangle A'\big)\cup\big(A'\triangle A''\big).
		\end{array}\]
		\end{spacing}
		\vspace{-.5cm}
		
		Since both $d_\A(A',A)$ and $d_\A(A',A'')$ are smaller than $\rho+d_\A(A,A'),$ we get
		\[ A\triangle A'' ~\subseteq~ \V(A',\rho+d_\A(A,A')). \]
		Hence $\V(A,\rho) \subseteq \V(A',\rho+d_\A(A,A')).$
		The first inclusion follows from this one by permuting $A$ and $A'$ as well as replacing $\rho$ by $\rho-d_\A(A,A').$

		\item Let us prove that the definition (\ref{eq:V_n}) of $\V_n(A,\rho)$ is consistent.
		
		Denote
		\[\begin{array}{rcl}
			\overline V_n(A,\rho)
			& \speq & \overline A_1\cup...\cup\overline A_k 	\\
			\underline V_n(A,\rho)
			& \speq & \underline A_1\cap...\cap\underline A_\ell
		\end{array}\]
		where the $\overline A_i$'s (resp. $\underline A_j$'s) are the maximal (resp. minimal) elements in $\A_n\cap B_\A(A,\rho)$.
		Then
		\[ \overline V_n(A,\rho)\setminus\underline V_n(A,\rho) ~=~ \bigcup_{\substack{1\ls i\ls k \\ 1\ls j\ls\ell}} \big(\overline A_i\setminus \underline A_j\big). \]
		From this expression, the converse inclusion in (\ref{eq:V_n}) is straightforward whereas the direct inclusion comes from the fact that any $\overline A\setminus \underline A$ is included in some $\overline A_i\setminus \underline A_j.$

		Let us prove the direct inclusion in (\ref{eq:V=U_nV_n}).
		For $A'\in B_\A(A,\rho),$ we have $A\triangle A' = \big[ A\setminus(A\cap A') \big] \cup \big[ A'\setminus(A\cap A') \big]$ by definition. 
		Hence, by separability from above (Definition \ref{de:indexing_collection}), we get
		\begin{equation}	\label{eq:symmetric_difference_discretized}
			\forall n_0\in\N,\quad A\triangle A' ~\subseteq~ \bigcup_{n\gs n_0} \big[ g_n(A)\setminus g_n(A\cap A') \big] \cup \big[ g_n(A')\setminus g_n(A\cap A') \big]
		\end{equation}
		which would then be included in $\bigcup_{n\in\N}\V_n(A,\rho)$ as long as there exists $n_0\in\N$ such that
		\begin{equation}	\label{eq:victiny_discretization_bound}
			\forall n\gs n_0,\quad \max\Big\{ d_\A(A,g_n(A)),\,  d_\A(A,g_n(A')),\,  d_\A(A,g_n(A\cap A')) \Big\} ~<~ \rho.
		\end{equation}
		So let us find such $n_0.$
		Using outer continuity (cf. Definition \ref{de:compatible_metric}) and contractivity for the last inequality, we get the following:
		\[\begin{array}{lclcll}
			d_\A(A,g_n(A)) 			& \underset{n\rightarrow\infty}{\longrightarrow}	
			& d_\A(A,A)			& =		& 0 			& ~<~\rho \\ 	
			d_\A(A,g_n(A'))			& \underset{n\rightarrow\infty}{\longrightarrow}	
			& d_\A(A,A')		&  		&  				& ~<~\rho \\
			d_\A(A,g_n(A\cap A')) 	& \underset{n\rightarrow\infty}{\longrightarrow}	
			& d_\A(A,A\cap A')	& ~\ls~	& d_\A(A,A')	& ~<~\rho.	 	
		\end{array}\]
		Thus (\ref{eq:victiny_discretization_bound}) is true for a big enough $n_0\in\N.$ Hence
		
		\vspace{-.1cm}	
		\[ \V(A,\rho) ~\subseteq~ \bigcup_{n\in\N}\V_n(A,\rho). \]
		\vspace{-.1cm}		
		
		For the converse inclusion in (\ref{eq:V=U_nV_n}), fix $\overline A,\underline A\in B_\A(A,\rho)$ such that $\underline A\subset\overline A.$
		Then,
		
		\vspace{-.7cm}
		\begin{spacing}{1.5}
		\[\begin{array}{rcl}
			\overline A\setminus\underline A
			& = 			& \big[ \big( A\cap\overline A\big)\setminus\underline A \big] \cup \big[ \overline A\setminus\big(A\cap\underline A\big) \big]	\\
			& ~\subseteq~	& \big[ \big( A\cap \overline A \big)\triangle\underline A \big] \cup \big[ \overline A\triangle\big(A\cap\underline A\big) \big].
		\end{array}\]
		\end{spacing}
		\vspace{-.3cm}
		
		Contractivity then shows that both $A\cap\overline A$ and $A\cap\underline A$ belong to $B_\A(A,\rho).$
		The result follows.  
	\end{enumerate}
\end{proof} 

\begin{prop}[Properties of $\qd$]	\label{prop:property_qd}
	The following properties hold:
	\begin{enumerate}
		\item For all $t\in\T,A\in A,~$ $\qd(t,A) ~\ls~ d_\A(A(t),A).$

		\item (Ersatz of triangle inequality). 
		For all $t\in\T,$ the map $\qd(t,.)$ is 1-Lipschitz, \ie
		\[ \forall A,A'\in\A,\quad |\qd(t,A)-\qd(t,A')| ~\ls~ d_\A(A,A'). \]
	\end{enumerate}
\end{prop}

\begin{proof} \ 
	\begin{enumerate}
		\item This property is just a consequence of the fact that, by the TIP bijection, $t\in A\triangle A(t)$ unless $A=A(t).$ 
		If $A=A(t),$ consider an increasing geodesic $\gamma:[0,1]\rightarrow\A$ from $\varnothing$ to $A.$ 
		Then $t\in A\setminus\gamma(x)$ for all $x<1,$ so $\qd(t,A)\ls d_\A(\gamma(x),A).$ Letting $x$ tends to $1^-$ yields $\qd(t,A)=0=d_\A(A(t),A).$ 
		
		\item 
		Let us fix $t\in\T$, $A,A'\in\A$ and $\rho>d_\A(A,A').$
		Denoting $\eps=\rho-d_\A(A,A')$ and using (\ref{eq:prop_V_open_ball_2}), we obtain 

		\vspace{-.7cm}
		\begin{spacing}{1.3}
		\[\begin{array}{rcccl}
			t 
			& ~\in~ 		& \V(A,\qd(t,A)+\eps)	& ~\subseteq~ & 
			\V(A',\qd(t,A)+\rho),	\\
			\V(A',\qd(t,A)-\rho)
			& ~\subseteq~	& \V(A,\qd(t,A)-\eps)	& ~\not\ni~ 	& 
			t. 
		\end{array}\]
		\end{spacing}
		\vspace{-.3cm}		
		
		Hence, by definition of $\qd(t,A'),$ for all $\rho>d_\A(A,A'),$
		\[ \qd(t,A)-\rho  ~\ls~ \qd(t,A') ~\ls~ \qd(t,A)+\rho. \]
		The result follows from taking $\rho\rightarrow d_\A(A,A')^+$ in the previous inequality.
	\end{enumerate}
\end{proof}

\subsection{Finite-dimensional hypotheses}
\label{subsection:indexing_collection_finite_dim}

Here, we make further assumptions on $(\T,\A)$ that have a finite-dimensional flavour. Those will be used in particular to establish a Lévy-Itô decomposition (Section \ref{subsection:levy-ito_decomposition}) and combined with the concepts of the previous section to prove some martingale related results (Section \ref{sec:regularity_thm}).

\medskip

If $(E,\pre_E)$ and $(F,\pre_F)$ are two partially ordered sets, an \emph{order embedding} $\phi:E\hookrightarrow F$ is a map such that
	\[ \forall x,y\in E,\quad x\pre_E y ~\Longleftrightarrow~ \phi(x)\pre_F\phi(y). \]
A partially ordered set $(E,\pre)$ has \emph{poset} (or \emph{order}) \emph{dimension} $\ls p$ where $p\in\N$ if there exists an \emph{order embedding} $\phi:E\hookrightarrow\N^p$ where $\N^p$ is endowed with the usual componentwise partial order.
Some authors prefer another definition based on the intersection of linear orders. We refer to \cite[Theorem 10.4.2]{ore_theory_1962} to see they are equivalent.

\begin{de}[Indexing collection of finite dimension]	\label{de:indexing_collection_finite_dimension}
	The indexing collection $\A$ is said to have \emph{finite dimension} if there exist $p\in\N^*$ such that the following properties hold:
	\begin{enumerate}
		\item The $\A_n$'s all have poset dimension $\ls p.$
		\item $\C^\ell\subseteq\C_{p-1}$, \ie all left-neighborhood $C\in\C^\ell$ may be written $C=A_0\setminus\bigcup_{i=1}^{p}A_i$ where $A_0,...,A_p\in\A.$
	\end{enumerate}
	The smallest of such integers will be called the \emph{dimension} of $\A$ and denoted $\dim\A.$
\end{de}

In the following, $\A$ will be an indexing collection of finite dimension $p.$

Let us briefly comment on those properties.

\begin{enumerate}
	\item The second property is required to bound the local number of elements in the `mesh' $\A_n$ as $n$ goes to infinity and so follows a similar philosophy as \cite[Assumption $\mathcal H_{\underline\A}$]{herbin_local_2016} even if neither of them implies the other.
	
	\item One might think that if the $\A_n$'s have poset dimension $\ls p$, then $\C^\ell\subseteq\C_{p'}$ for a possibly greater $p'.$
	Unfortunately, that is not the case as illustrated in Figure \ref{fig:finite_dimension_counter-example} where the elements of $\A_1,$ $\A_2$ and $\A_3$ are drawn. 
	Even though their poset dimension is $2,$ one may complete this construction such that the extremal representation of $C_n((1,1))$ is made from an ever increasing number of elements in $\A_n,$ and so cannot belong to some $\C_{p'}$ where $p'$ is independent from $n.$
	
	\vspace{-0.2cm}
	\begin{center}
	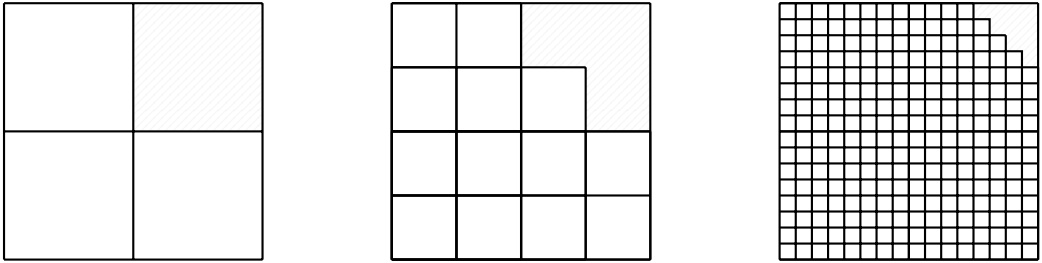
\begin{figure}[h!]
	
	\hspace{1.3cm}
	\begin{tikzpicture}[scale=.85]		
		\draw [thick] (0,0)	grid[step=2]	(4,4);
		\fill [pattern=north east lines,opacity=.2] (2,2)--(2,4)--(4,4)--(4,2)--(2,2);
		
		\draw [thick] (6,0)	grid[step=2]	(10,4);
		\draw [thick] (6,0)	grid[step=1]	(9,3);
		\draw [thick] (6,0)	grid[step=1]	(10,2);
		\draw [thick] (6,0)	grid[step=1]	(8,4);
		\fill [pattern=north east lines,opacity=.2] (8,4)--(8,3)--(9,3)--(9,2)--(10,2)--(10,4)--(8,4);
		
		\draw [thick] (12,0)	grid[step=2]	(16,4);
		\draw [thick] (12,0)	grid[step=0.25]	(15,4);
		\draw [thick] (15,0)	grid[step=0.25]	(16,3);
		\draw [thick] (15,3)	grid[step=0.25]	(15.5,3.5);
		\draw [thick] (15,3.75)--(15.25,3.75)--(15.25,3.5);
		\draw [thick] (15.5,3.25)--(15.75,3.25)--(15.75,3);
		\fill [pattern=north east lines,opacity=.2] (15,4)--(15,3.75)--(15.25,3.75)--(15.25,3.5)--(15.5,3.5)--(15.5,3.25)--(15.75,3.25)--(15.75,3)--(16,3)--(16,4)--(15,4);
	\end{tikzpicture}
	
	\caption{$C_n(1,1)$ (hatched region) for $n\in\{1,2,3\}$ and $\T=[0,1]^2$ endowed with a peculiar indexing collection.}
	\label{fig:finite_dimension_counter-example}
	\end{figure}
	\end{center}
	\vspace{-1cm}

	\item Conversely, we have not been able to prove that $\C^\ell\subseteq\C_{p-1}$ implies an upper bound on the poset dimension on the $\A_n$'s. To the best of our knowledge of the litterature on the subject, it seems to be a hard problem (see \cite{trotter_combinatorics_1992} and references therein for related results which do not quite fit our setting).
\end{enumerate}

\section{Integration with respect to a set-indexed Lévy process} \label{sec:construction_integral}

This section is devoted to give a sense to the integral $\ds\int_\T f\,\text dX$. 
According to \cite[Theorem 4.3]{herbin_set-indexed_2013}, the process $X=\{X_A:A\in\A\}$ against which we are going to integrate will end up being a particular case of \emph{Independently Scattered Random Measure} (ISRM) as defined by Rajput and Rosinski \cite{rajput_spectral_1989}. 
So why not use the integral built in \cite{rajput_spectral_1989} since those two agree whenever they are both defined?
As mentioned in the introduction, there are mainly two reasons: first, the stationarity of the increments of the set-indexed Lévy process enables a simpler construction and second, the set-indexed setting will prove to be instrumental in getting a Lévy-Itô decomposition for both processes $X$ and $Y$ that the more general setting of \cite{rajput_spectral_1989} does not provide.

\subsection{Density of simple functions}

Let us consider a $\sigma$-finite Borel measure $m$ defined on $(\T,\B)$ and such that $m(A)<\infty$ for all $A\in\A.$  
Since $\A\neq\B$ in general, it is not entirely obvious that $\E$ (Definition \ref{de:simple_functions}) is dense in $L^p(m)$ for $p\gs1.$ 
This part explains why it is nonetheless true in this case.

\begin{lem}[$\sigma$-finiteness]	\label{lem:sigma-finite}
	The increasing sequence $(\T_n)_{n\in\N}\in\A(u)^\N$ given by $\T_n = \bigcup_{A\in\A_n}A$ is such that $\T=\bigcup_{n\in\N}\T_n$ and $m(\T_n)<\infty$ for all $n\in\N.$
	In particular, $m$ is $\sigma$-finite.
\end{lem}

\begin{proof}
	$\T=\bigcup_{n\in\N}\T_n$ directly follows from the TIP bijection. 
	The $\sigma$-finiteness directly follows from the fact that $m$ is finite on $\A$ and that $\A_n$ is also finite.
\end{proof}

Just as in Examples \ref{exs:compatible_metric}, $d_m : (B,B')\mapsto m(B\triangle B')$ defines a metric on $\B_m = \{ B\in\B : m(B)<\infty \}.$  
That metric is natural in the sense that it is the restriction to indicator functions of the usual $L^1(m)$ metric, \ie
\begin{equation}	\label{eq:isometry_L^1_C(u)}
	\forall B,B'\in\B,\quad \| \myindic{B}-\myindic{B'} \|_{L^1(m)} \speq d_m(B,B').
\end{equation}
Littlewood's principles \cite{littlewood_lectures_1944} acts as guides to intuition with regards to measure theory and how can one apprehend some of its hardest concepts. In particular, the first principle tells that any Borel set $B\in\B(\R)$ such that $\mathbf{Leb}(B)<\infty$ --- where $\mathbf{Leb}$ is the Lebesgue measure --- can be approximated as a finite union of intervals with respect to the metric $d_{\mathbf{Leb}}.$ This principle is much more general and still holds in our setting when finite unions of segments are replaced by the elements of $\C(u).$

\begin{lem} \label{lem:distance_inequality}
	For any $(B_n)_{n\in\N},$ $(B'_n)_{n\in\N}\in\B^\N,$ 
	\[ d_m\Big( \bigcap_{n\in\N}B_n, \bigcap_{n\in\N}B'_n \Big) ~\ls~\sum_{n=0}^\infty\, d_m(B_n,B'_n).\]
\end{lem}

\begin{proof}
	Let $(B_n)_{n\in\N},$ $(B'_n)_{n\in\N}\in\B^\N.$
	Then, 
	\vspace{-0.5cm}
	\begin{spacing}{1.5}
	\[\begin{array}{rcl}
		d_m \Big( \bigcap_n B_n,~\bigcap_n B_n' \Big)
			& \speq 	& m \Big( \big( \bigcup_n B_n^\complement \big) ~\cap~ \big( \bigcap_n B_n' \big) \Big) \+ m \Big( \big( \bigcap_n B_n \big) ~\cap~ \big( \bigcup_n B_n'^\complement \big) \Big)	\\
			& ~\ls~	& {\ds\sum_j}\, m\lp B_j^\complement \cap \big( \bigcap_nB'_n \big) \rp \+ {\ds\sum_j}\, m\lp \big( \bigcap_n B_n \big) \cap  B_j'^\complement \rp	\\
			& \ls	& {\ds\sum_j}\, m(B_j' \setminus B_j) \+ {\ds\sum_j}\, m(B_j \setminus B_j')  						\\
			& =  	& {\ds\sum_j}\, d_m(B_j,B_j').
	\end{array}\]
	\end{spacing}
	\vspace{-0.7cm}
\end{proof}

\begin{prop} \label{prop:littlewood}
	The metric space $(\B_m,d_m)$ is complete and $\C(u)$ is a dense subset.
\end{prop}

\begin{proof}
	Let us show that $\B_m$ is complete. 
	Let $(B_n)_{n\in\N}$ be a Cauchy sequence in $\B_m$.
	According to (\ref{eq:isometry_L^1_C(u)}), $\lp\myindic{B_n}\rp_{n\in\N}$ is a Cauchy sequence in $L^1(m).$ Since $L^1(m)$ is complete, this sequence has a limit.
	We claim that this limit is necessarily of the form $\myindic{B}.$
	We only prove it for $m(\T)<\infty$ since $m$ is $\sigma$-finite (cf. Lemma \ref{lem:sigma-finite}). 
	Then apply Borel-Cantelli's lemma to obtain subsequence of $\lp\myindic{B_n}\rp_{n\in\N}$ converging $m$-a.e.
	Its limit must thus take its values in $\{0,1\}$ $m$-a.e., \ie be equal to $\myindic{B}$ for some $B\in\B.$
	
	Applying (\ref{eq:isometry_L^1_C(u)}) once more yields that $m(B)<\infty$ and $B_n\rightarrow B$ in $\B_m$ as $n\rightarrow\infty.$
	The completeness follows.
	
	\medskip
	
	Let us prove that $\C(u)$ is a dense subset of $\B_m,$ \ie $\overline{\C(u)}=\B_m$ where $\overline{\C(u)}$ is the closure in $\B_m$ of $\C(u).$
	Since $m$ is finite on $\A,$ $\C(u)\subseteq\B_m.$
	According to Lemma \ref{lem:sigma-finite}, we may assume $\T\in\A(u).$
	In this case, since $\B_m=\B=\sigma(\overline{\C(u)}),$ it is equivalent to show that $\overline{\C(u)}$ is a $\sigma$-algebra. 
	
	$\varnothing\in\overline{\C(u)}$ is trivial whereas stability by complement directly follows from the relation $m(B^\complement\triangle U^\complement) = m(B\triangle U)$ and the fact that $\C(u)$ is closed under complement since $\T\in\C(u).$
	
	It remains to prove stability under countable intersection.
	Let $(B_n)_{n\in\N}$ be a sequence in $\overline{\C(u)}$ and $\eps >0.$
	For any $n\in\N,$ consider $U_n\in\C(u)$ such that $d_m(B_n,U_n)\ls\eps\,2^{-(n+1)}.$ Then, according to Lemma \ref{lem:distance_inequality},
	\begin{equation} 	\label{eq:d_m_1} 	 
		d_m\Big( \bigcap_{n\in\N}B_n, \bigcap_{n\in\N}U_n \Big) ~\ls~ \sum_{n=0}^\infty \eps2^{-(n+1)} \speq \frac{\eps}{2}. 
	\end{equation}
	Moreover, since $m(\T)<\infty,$ the monotone continuity of $m$ implies that
	\begin{equation}	\label{eq:d_m_2}
		d_m\Big( \bigcap_{n\in\N}U_n, \bigcap_{n\ls k}U_n \Big) \speq m\Big( \bigcap_{n\ls k}U_n \setminus \bigcap_{n\in\N}U_n \Big) ~\underset{k\rightarrow\infty}{\longrightarrow}~ 0. 
	\end{equation}
	Hence (\ref{eq:d_m_1}), (\ref{eq:d_m_2}) and the triangle inequality imply that for $k$ big enough,
	\[ d_m\Big( \bigcap_{n\in\N}B_n, \bigcap_{n\ls k}U_n \Big) ~\ls~ \eps. \]
	Since $\C(u)$ is closed under finite intersection, $\bigcap_{n\ls k}U_n \in\C(u).$ Hence $\bigcap_{n\in\N}B_n \in \overline{\C(u)}.$ 
\end{proof}

\begin{rem}
	Actually, $(\B_m,d_m)$ is separable since one may prove that $\C^\ell(u)$ is countable and dense.
\end{rem}

\begin{cor}[Density of simple functions] \label{cor:density_simple_functions}
	For any $p\gs1,\,$ $\E=\text{Span}\{ \myindic{A}:A\in\A \}$ is dense in $L^p(m).$
\end{cor}

\begin{proof}
	Since the linear space $\text{Span}\{ \myindic{B}:B\in\B_m \}$ is dense in $L^p(m),$ we just need to check that $\myindic{B}$ belongs to the closure of $\E$ in $L^p(m)$ for any $B\in\B_m.$ 
	Given such $B\in\B_m,$ Proposition \ref{prop:littlewood} implies the existence of a sequence $(U_n)_{n\in\N}$ in $\C(u)$ such that $d_m(U_n,B)\rightarrow0$ as $n\rightarrow\infty.$
	Thus
	\[ \| \myindic{U_n}-\myindic{B} \|_{L^p(m)} \speq m(U_n\triangle B)^{1/p} ~\longrightarrow~ 0 \spas n\rightarrow\infty. \]
	Hence $\E$ is dense in $L^p(m).$
\end{proof}

\subsection{Construction of the integral}	\label{subsection:construction_integral}

In the following, let us consider a measure $m$ on $(\T,\B)$ such that $d_m$ is set-indexed compatible (Definition \ref{de:compatible_metric}) as well as a $L^2(\Omega)$ \emph{set-indexed Lévy process} (siLévy) $X$ as defined and studied in \cite{herbin_set-indexed_2013}. 
By that, we mean that $X$ checks the following properties:
\begin{enumerate}
	\item[0.] For all $A\in\A,~\esp{X_A^2}<\infty.$
	 
	\item (Independent increments). 
		For any positive integer $n$ and any pairwise disjoint $C_1,...,C_n\in\C$, the random variables $\Delta X_{C_1},...,\Delta X_{C_n}$ are independent where $\Delta X$ stands for the increment process from Proposition-Definition \ref{propde:increment_map_linear_extension}.
		
	\item (Stationary increments).
		For any $C_1,C_2\in\C$ such that $m(C_1)=m(C_2),$ $\Delta X_{C_1}$ and $\Delta X_{C_2}$ have the same distribution.
		
	\item (Outer continuity in distribution).
		For any non-increasing sequence $(A_n)_{n\in\N}$ in $\A,$ $X_{A_n}$ converges in distribution to $X_A$ where $A=\bigcap_{n\in\N}A_n.$
\end{enumerate}

This definition varies slightly from the one given in \cite{herbin_set-indexed_2013}.
First, the independence and stationary properties of the increments have been formulated as in \cite[Corollary 4.5]{herbin_set-indexed_2013} which is an equivalent definition.
Secondly, outer continuity in distribution replaces stochastic continuity to fit our setting with regards to the flows vs increasing geodesics discussion of Section \ref{subsection:indexing_collection_geodesic}. 
Theorem \ref{theo:integral_wrt_levy} will prove that this condition is actually equivalent to stochastic continuity.
Let us also mention that our regularity results will not depend on the behavior of the `big jumps', so imposing to have second-order moment is not truly a restriction per say.

\medskip

A nice feature is that $X$ shares a lot of common properties with the usual one-dimensional case, in particular with respect to its Fourier transform $\phi_{\Delta X_C}(\xi)=\esp{e^{i\xi\Delta X_C}}$ for which we can write a Lévy-Khintchine decomposition (see \cite{herbin_set-indexed_2013}), meaning that there exists a triplet $(b,\sigma^2,\nu)$ where $b\in\R,$ $\sigma^2\in\R_+$ and $\nu$ is a Borel measure on $\R$ such that $\nu(\{0\})=0$ and $\ds\int_\R(1\wedge x^2)\nu(\text dx)<\infty$ such that for all $C\in\C,$ $\phi_{\Delta X_C}=\exp[m(C)\psi]$ where $\psi$ is the \emph{Lévy-Khintchine exponent} of $X$ given by 
\begin{equation} \label{eq:levy-khintchine}
	\forall\xi\in\R,\quad \psi(\xi) \speq ib\xi \spm \frac{1}{2}\sigma^2\xi^2 \+ \int_\R \lp e^{i\xi x}-1-i\xi x\myindic{|x|\ls1} \rp \nu(\text dx).
\end{equation}

\begin{rem}
	The Lévy-Khintchine decomposition proved in \cite{herbin_set-indexed_2013} remains true in our modified setting since the argument relies on properties shared both by flows (namely \cite[Proposition 4.1]{herbin_set-indexed_2013}) and increasing geodesics (as used here).
\end{rem}

\begin{lem} \label{lem:propto_m}
	For all $C\in\C,$ we have
	\[ \esp{\Delta X_C} \speq \lp b+\int_{|x|>1} x\nu(\text dx) \rp m(C) \quad\text{and}\quad \text{Var}(\Delta X_C) \speq \lp \sigma^2 + \int_\R x^2\nu(\text dx) \rp m(C). \]
\end{lem}

\begin{proof}
	Let $C\in\C$ and denote $\psi_C = m(C)\psi.$ 
	Since $\Delta X_C\in L^2(\Omega),$ $\psi_C$ is twice differentiable and
	\[	\psi_C'(0)=i\esp{\Delta X_C} \quad\text{and}\quad \psi''_C(0)=-\text{Var}(\Delta X_C).	\]
	Since $\Delta X_C\in L^2(\Omega),$ we also know that $\ds\int_\R x^2\nu(\text dx)<\infty,$ so we are able to differentiate twice under the integral sign to get for all $\xi\in\R,$
	\[ \psi'_C(\xi) = \lp ib - \sigma^2\xi + \int_\R \lp ixe^{i\xi x}-ix\myindic{|x|\ls1} \rp\nu(\text dx)\rp m(C) \]
	and
	\[ \psi''_C(\xi) = \lp -\sigma^2 + \int_\R \lp -x^2e^{i\xi x} \rp\nu(\text dx) \rp m(C). \]
	The result follows from taking $\xi=0.$ 
\end{proof}

This lemma tells us that the expectation and variance of $\Delta X$ behave exactly as the `reference' measure $m$ that dictates the stationarity of the increments. 
It is a key step to define the integral with respect to $X$ through a Wiener type of construction.

\begin{de}[$X$-integrable functions]
	Assume that $X$ is a siLévy with Lévy-Khintchine triplet $(b,\sigma^2,\nu)$ given by (\ref{eq:levy-khintchine}).
	Let $L(X)$ denote the linear space of Borel functions $f:\T\rightarrow\R$ where $m$-a.e. equal maps are identified and such that
	\[ \|f\|_{L(X)} \speq \sqrt{ \lp b+\int_{|x|>1} x\nu(\text dx) \rp^2\|f\|_{L^1(m)}^2 + \lp \sigma^2 + \int_\R x^2\nu(\text dx) \rp\|f\|_{L^2(m)}^2 } \]
	is finite (with the convention $0\times\infty=0$). 
	
	Measurable functions $f:\T\rightarrow\R$ such that $f\myindic{A}\in L(X)$ for all $A\in\A$ are called \emph{locally $X$-integrable functions}.	
\end{de}

We remark that depending on whether the proportionality constants are zero or not, $L(X)$ is either $L^1(m),$ $L^2(m)$ or $L^1(m)\cap L^2(m)$ as long as $X$ is non zero.
It is then a classical result that $\lp L(X),\|.\|_{L(X)} \rp$ is a Banach space. Moreover, as a consequence of Corollary \ref{cor:density_simple_functions}, $\E$ is a dense subspace of $L(X).$

\begin{theo}[Integration with respect to $X$] \label{theo:integral_wrt_levy}
	There exists a unique continuous linear map 
	\[\begin{array}{cccc}
		\mathbf X : 	
		& L(X) 	& \longrightarrow 	& L^2(\Omega) 	\\
		& f		& \longmapsto 		& \mathbf X(f) = \ds\int_\T f\,\text dX
	\end{array}\]
	such that $\mathbf X(\myindic{A})=X_A$ for all $A\in\A.$
	$\mathbf X(f)$ is called the \emph{(stochastic) integral of $f$ with respect to $X$}.
\end{theo}

\begin{proof}
	We know by Proposition \ref{propde:increment_map_linear_extension} that $\restr{\mathbf X}{\E}$ must be the linear functional associated with $X.$ 
	Let $f=\sum_{i=1}^nc_i\myindic{C_i}$ be a simple function written with its $\C$-representation given by Proposition \ref{prop:C-representation}.
	Then, by Lemma \ref{lem:propto_m},
	

	\vspace{-1.7cm}
	\begin{spacing}{2.5}
	\[\begin{array}{rcl}
		|\esp{\mathbf X(f)}|
		& ~\ls~	& \ds\sum_{i=1}^n |c_i||\esp{\Delta X_{C_i}}|	\\
		& = 	& \ds\left| b+\int_{|x|>1} x\nu(\text dx) \right| \sum_{i=1}^n |c_i|m(C_i)	\\
		& = 	& \ds\left| b+\int_{|x|>1} x\nu(\text dx) \right| \|f\|_{L^1(m)}.
	\end{array}\]
	\end{spacing}
	\vspace{-1.1cm}
	
	and 
	
	\vspace{-1.7cm}
	\begin{spacing}{2.5}
	\[\begin{array}{rcl}
		\text{Var}(\mathbf X(f))
		& \speq	& \ds\sum_{i=1}^n c_i^2\text{Var}(\Delta X_{C_i}) 	\\
		& = 	& \ds\lp \sigma^2 + \int_\R x^2\nu(\text dx) \rp  \sum_{i=1}^n c_i^2m(C_i) 	\\
		& = 	& \ds\lp \sigma^2 + \int_\R x^2\nu(\text dx) \rp \|f\|_{L^2(m)}^2.
	\end{array}\]
	\end{spacing}
	\vspace{-0.8cm}	\ 
	
	Hence $\|\mathbf X(f)\|_{L^2(\Omega)}^2 = \text{Var}(\mathbf X(f)) + \esp{\mathbf X(f)}^2 \ls \|f\|_{L(X)}^2.$
	So $\mathbf X$ may be uniquely extended as intended.
\end{proof}

We remark that this construction enables to extend the process $\Delta X$ from $\C(u)$ to $\B_m$ by setting
\begin{equation}	\label{eq:increment_map_on_borel_sets}
	\forall B\in\B_m,\quad \Delta X_B \speq \mathbf X(f\myindic{B}) \speq \int_\T\myindic{B}\,\text dX.
\end{equation}

From Theorem \ref{theo:integral_wrt_levy}, we may finally define the process $Y$ by setting
\begin{equation}	\label{eq:definition_of_Y}
	\forall A\in\A,\quad Y_A \speq \mathbf X(f\myindic{A}) \speq \int_A f\,\text dX
\end{equation}
where $f$ is a given locally $X$-integrable function.

Likewise, we extend the increment process $\Delta Y$ from $\C(u)$ to $\B_m$ by setting 
\begin{equation}	\label{eq:Delta_Y_on_borelian}
	\forall B\in\B_m,\quad \Delta Y_B \speq \mathbf X(f\myindic{B}) \speq \int_\T f\myindic{B}\,\text dX
\end{equation}
for which we state and prove the natural property that pairwise disjoint increments of $\Delta Y$ remain independent.

\begin{prop}[Y has independent increments]	\label{prop:independent_increments}
	For all $k\gs1$ and pairwise disjoint $B_1,...,B_k\in\B_m$, the random variables $\Delta Y_{B_1},...,\Delta Y_{B_k}$ are independent.
\end{prop}

\begin{proof}
	Fix an integer $k\gs1$ and consider
	\[ \G \speq \Big\{ (B_1,...,B_k)\in(\B_m)^k : \Delta Y_{B_1},\Delta Y_{B_2\setminus B_1},...,\Delta Y_{B_k\setminus \cup_{j<k}B_j} \text{ are independent} \Big\} \]
	Let us show that $\G=(\B_m)^k$.
	Endow $(\B_m)^k$ with the metric
	\[ \big( (B_1,...,B_k),\, (B'_1,...,B'_k) \big) ~\longmapsto~ \sum_{j=1}^k d_m(B_j,B'_j). \]
	According to Proposition \ref{prop:littlewood}, $\C(u)^k$ is a dense subset of $(\B_m)^k$.
	Using the independence of the increments of $X$, $\G$ is easily shown to contain $\C(u)^k$.
	
	Moreover, by the continuity of $\mathbf X$ (Theorem \ref{theo:integral_wrt_levy}), $\G$ is also closed for such a metric. 
	
	Hence $\G=(\B_m)^k.$
	The result follows.
\end{proof}

\subsection{Lévy-Itô decomposition}	\label{subsection:levy-ito_decomposition}

Since our goal is to study the regularity of the trajectories of the set-indexed process $Y$ as defined in (\ref{eq:definition_of_Y}), we must first find a version with well-defined, nice sample paths.
In the one-dimensional case of $X,$ this hurdle has already been dealt with a long time ago with the Lévy-Itô decomposition, which states that any Lévy process $X$ may be written as the independent sum of a drift, a Brownian motion and a compensated Poisson process. 
This classical result has been adapted to the set-indexed case in \cite{herbin_set-indexed_2013}. 

\medskip

The Gaussian part being the easiest, let us start with it. 
A \emph{set-indexed Brownian motion (siBm)} is a centered Gaussian process $B=\{B_A:A\in\A\}$ with covariance function given by
\begin{equation} 	\label{eq:gaussian_part} 
	\forall A_1,A_2\in\A,\quad \esp{B_{A_1}B_{A_2}} \speq m(A_1\cap A_2). 
\end{equation}
Such processes and their fractional generalization have been studied in \cite{herbin_set-indexed_2006} in the set-indexed setting.

\medskip

The Poissonian part requires a bigger setup.
Since the universe $\Omega$ was supposed rich enough to support the existence of the process $X,$ it supports the following construction whose details may be found in \cite{kingman_poisson_1992}.

The measure $m\otimes\nu$ being $\sigma$-finite, we can consider a Poisson random set $\widetilde\Pi$ on $\T\times\R$ of intensity $m\otimes\nu.$ 

Define then
\begin{equation}	\label{eq:poisson_part_0}
	\forall\eps>0,\,\forall A\in\A,\quad N_A^\eps \speq \sum_{\substack{(t,J)\in\widetilde\Pi: \\ t\in A,~|J|\gs\eps}}J \spm m(A)\int_{\eps\ls|x|\ls1}x\nu(\text dx).
\end{equation}
where the sum is finite by construction of $\widetilde\Pi.$

According to \cite[Theorem 7.9]{herbin_set-indexed_2013}, the set-indexed process $N^\eps$ converges almost surely uniformly in $A\subseteq A_\text{max}$ (for any given $A_\text{max}\in\A$) to a set-indexed Lévy process $N$ with Lévy-Khintchine triplet $(0,0,\nu).$
So if we also consider the siBm $B$ from (\ref{eq:gaussian_part}) independent from $N,$ the set-indexed process $bm(.)+\sigma B+N$ actually is a siLévy with Lévy-Khintchine triplet $(b,\sigma^2,\nu).$
Since such a triplet characterizes the finite-dimensional distributions of $X,$ we might as well suppose that
\begin{equation}	\label{eq:levy-ito_decomposition_X}
	\forall A\in\A,\quad X_A \speq bm(A) \+ \sigma B_A \+ N_A.
\end{equation}

The biggest advantage of this representation lies in the fact that we are now able to talk about the `jump structure' of $X$ as follows.
Recall that the \emph{point-mass jump at $t\in\T$} of a map $h:\A\rightarrow\R$ is given by $J_t(h)=\lim_{n\rightarrow\infty}\Delta h(C_n(t))$ as long as the limit is defined.

\begin{de}[Well-defined jumps] \label{de:well-defined_jumps}
	A map $h:\E\rightarrow\R$ has \emph{well-defined jumps} if $J_t(h)$ is well-defined for all $t\in\T$ and $\{t\in A:|J_t(h)|>\eps\}$ is finite for all $A\in\A$ and $\eps>0.$
\end{de}

For any map $h:\A\rightarrow\R$ with well-defined jumps, denote its \emph{jump set} by
\begin{equation}	\label{eq:jump_set_h}
	\Pi(h) \speq \big\{ t\in\T : J_t(h)\neq0 \big\}.
\end{equation}
According to \cite[Theorem 7.3]{herbin_set-indexed_2013}, $X$ has well-defined jumps and $\Pi(X)=\Pi$ almost surely where $\Pi$ the projection of $\widetilde\Pi$ onto $\T.$
In particular, (\ref{eq:poisson_part_0}) has the `nicer' writing:
\begin{equation} 	\label{eq:poisson_part}
	\forall\eps>0,\,\forall A\in\A,\quad N_A^\eps \speq \sum_{\substack{t\in\Pi\cap A: \\ |J_t(X)|\gs\eps}} J_t(X) \spm m(A)\int_{\eps\ls|x|\ls1}x\nu(\text dx)
\end{equation}
from which it is possible to derive a `nicer' version of the integral defined in Theorem \ref{theo:integral_wrt_levy}:
\begin{equation}	\label{eq:integral_N^eps} 
	\forall\eps>0,\, \forall A\in\A,\quad \int_Af\,\text dN^\eps \speq \sum_{\substack{t\in\Pi\cap A: \\ |J_t(X)|\gs\eps}} f(t)J_t(X) \,-\, \int_A f\,\text dm\int_{\eps\ls|x|\ls1}x\nu(\text dx).
\end{equation}

The right-hand side constitutes a set-indexed process with well-defined jumps. Fortunately, this property may be kept when taking $\eps\rightarrow0^+$.


\begin{theo}[Lévy-Itô decomposition] 	\label{theo:levy-ito_decomposition}
	The process $Y$ may be written as a sum of three independent set-indexed processes:
	\begin{equation}	\label{eq:levy-ito_decomposition}
		\forall A\in\A,\quad Y_A \speq b\int_A f\,\text dm \+ \sigma\int_A f\,\text dB \+ \int_A f\,\text dN \quad\text{a.s.}
	\end{equation}
	where $\ds\int_.f\,\text dN = \lim_{\eps\rightarrow0^+}\ds\int_.f\,\text dN^\eps$. The limit as $\eps\rightarrow0^+$ happens almost surely uniformly in $[\varnothing,A]$ for all $A\in\A.$
%
\end{theo}

\begin{proof}
	The proof goes just as in \cite[Theorem 7.9]{herbin_set-indexed_2013} apart from a mistake that has been made in the application of Wichura's maximal inequality from \cite{wichura_inequalities_1969}.
	Indeed, we need to have the $\A_n$ of bounded poset dimension in order to have a correct upper bound similar to the proof of \cite[Theorem 4.6]{adler_representations_1983}.
	But since $\A$ has been supposed to have finite dimension (Definition \ref{de:indexing_collection_finite_dimension}), it is indeed the case.
\end{proof}

In the case when $\ds\int_{|x|\ls1}|x|\nu(\text dx)<\infty,$ we define
\begin{equation}
	\forall\eps>0,\,\forall A\in\A,\quad \widetilde N^\eps_A \speq \sum_{\substack{t\in\Pi\cap A: \\ |J_t(X)|\gs\eps}} f(t)J_t(X) \spand \int_A f\,\text d\widetilde N^\eps \speq \sum_{\substack{t\in\Pi\cap A: \\ |J_t(X)|\gs\eps}} f(t)J_t(X)
\end{equation}
for which another Lévy-Itô decomposition can be stated.

\begin{cor}[Lévy-Itô decomposition for integrable jumps]	\label{cor:levy-ito_decomposition_integrable_case}
	If $\ds\int_{|x|\ls1}|x|\nu(\text dx)<\infty,$ (\ref{eq:levy-ito_decomposition}) may be rewritten as:
	\begin{equation}	\label{eq:levy-ito_decomposition_integrable_case}
		\forall A\in\A,\quad Y_A \speq \widetilde b\int_A f\,\text dm \+ \sigma\int_A f\,\text dB \+ \int_A f\,\text d\widetilde N
	\end{equation}
	where $\widetilde b = b-\ds\int_{|x|\ls1}x\nu(\text dx),$ and $\ds\int_.f\,\text d\widetilde N = \lim_{\eps\rightarrow0^+}\ds\int_.f\,\text d\widetilde N^\eps$.
	The limit as $\eps\rightarrow0^+$ happens almost surely uniformly in $[\varnothing,A]$ for all $A\in\A.$ 
%
\end{cor}

So far, it is not clear to us whether one can obtain such a representation for the whole process $\mathbf X,$ \ie if
	\[ \forall f\in L(X),\quad \mathbf X(f) \speq b\int_\T f\,\text dm + \sigma\int_\T f\,\text dB + \lim_{\eps\rightarrow0^+}\int_\T f\,\text dN^\eps \]
where the convergence as $\eps\rightarrow0^+$ happens almost surely in some pathwise sense. 
However, one may easily establish this representation in $L^2(\Omega).$

\medskip

As we did for $X$ in (\ref{eq:levy-ito_decomposition_X}), we will only consider the versions of $Y$ given by (\ref{eq:levy-ito_decomposition}) or (\ref{eq:levy-ito_decomposition_integrable_case}) depending on the context.
A first important consequence of the Lévy-Itô decomposition is that it gives the jump structure of $Y.$
	
\begin{cor}[Jump structure]	\label{cor:jump_structure_Y}
	The following holds with probability one: $Y$ has well-defined jumps and for all $t\in\T,$ $J_t(Y)=f(t)J_t(X).$
	In particular, $\Pi(Y) = \Pi\cap\{ f\neq0 \}$ where $\{f\neq0\} = \{ t\in\T:f(t)\neq0 \}.$
\end{cor}

\begin{proof}
	According to Lemma \ref{lem:sigma-finite}, it is enough to prove that for any $n\in\N,$ the following holds with probability one:
	\[ \forall t\in\T_n,\quad J_t(Y) \speq f(t)J_t(X) \]
	where $\T_n$ is the one from that very lemma.
	Let $n\in\N.$
	We write the Lévy-Itô decomposition of $Y$ and determine de point-mass jump of its components.
	Since $d_m$ is compatible, from the shrinking mesh property and $f\in L(X)$ follows that
	\begin{equation}	\label{eq:jump_drift}
		\forall t\in\T_n,\quad J_t\lp b\int_. f\,\text dm\rp \speq 0.
	\end{equation}
	
	According to Theorem \ref{theo:levy-ito_decomposition}, there exists an event $\Omega_N$ of probability one such that for all $\omega\in\Omega_N$, $\ds\int_A f\,\text dN^\eps(\omega)$ converges to $\ds\int_A f\text dN(\omega)$ uniformly in $A\subseteq\T_n$ as $\eps\rightarrow0^+.$
	Since $\A$ has finite dimension, this convergence also happens uniformly in $C\in\C^\ell$ such that $C\subseteq\T_n.$ In particular, since $\C^\ell$ is a dissecting system,
	\begin{equation}	\label{eq:jump_poisson}
		\forall\omega\in\Omega_N,\,\forall t\in\T_n,\quad J_t\lp\int_. f\,\text dN(\omega)\rp \speq f(t)J_t(X(\omega)).
	\end{equation}
	
	Only the Gaussian part remains, but it may be proven by the same method as \cite[Theorem 7.3]{herbin_set-indexed_2013}, so there exists an event $\Omega_B$ of probability one such that for all $\omega\in\Omega_B,\,$ $\sigma\ds\int_.f\,\text dB(\omega)$ has well-defined jumps and 
	\begin{equation}	\label{eq:jump_gaussian}
		\forall\omega\in\Omega_B,\,\forall t\in\T_n,\quad J_t\lp \sigma\int_.f\,\text dB(\omega)\rp \speq 0.
	\end{equation}
	The result follows from (\ref{eq:jump_drift}), (\ref{eq:jump_poisson}) and (\ref{eq:jump_gaussian}).
\end{proof}

\section{Regularity criteria} \label{sec:regularity_criterion}

Hölder regularity is expressed in terms of exponents and may vary depending on the context and the behavior one wishes to capture. 
In Section \ref{subsection:ptw_exponent}, we provide the necessary definitions to this effect.

In Section \ref{subsection:oji_function}, we further push ideas from \cite{jaffard_multifractal_1999, jaffard_old_1997} to obtain (deterministic) upper bounds for the Hölder regularity of a deterministic function $h:\A\rightarrow\R$ based on its pointwise jumps. The notion of \emph{victiny} developed in Section \ref{subsubsec:divergence} is the key concept to improve on the `naive' upper bound.

\subsection{Set-indexed pointwise Hölder exponents} \label{subsection:ptw_exponent}

In \cite{herbin_local_2016}, Herbin and Richard defined a number of Hölder exponents which localize different visions of the continuity property in the set-indexed case.
In a general fashion, those Hölder exponents are used to finely study the regularity of maps $h:\A\rightarrow\R$ around a given $A\in\A$ or equivalently through the TIP bijection, around a given $t\in\T.$
We will be using the convention $\sup\varnothing=0$ which is usual for regularity exponents.

\subsubsection{Hölder exponent}

First, the \emph{(pointwise) Hölder exponent} constitutes the natural generalization of its one-dimensional analog to the metric space $(\A,d_\A)$:
\begin{equation} 	\label{eq:holder_exponent}
	\forall A\in\A,\quad \alpha_h(A) \speq \sup\left\{ \alpha\gs0 \,:\, \limsup_{\rho\rightarrow0^+} \sup_{A'\in B_\A(A,\rho)}\frac{\big| h(A)-h(A') \big|}{\rho^\alpha} <\infty\right\}.
\end{equation}

If positive, for any $\alpha\in(0,\alpha_h(A)),$ the pointwise Hölder exponent yields the following control of $h$ in the neighborhood of $t$:
\begin{equation} 	\label{eq:holder_exponent_estimate}
	\exists\rho_\alpha>0 : \forall A'\in B_\A(A,\rho_\alpha),\quad \big| h(A)-h(A') \big| ~\ls~ d_\A(A,A')^\alpha.
\end{equation}
Conversely, the estimate (\ref{eq:holder_exponent_estimate}) implies $\alpha\ls\alpha_h(A).$

\medskip

In modern litterature, one usually uses a slight modification of the above definition where one substracts the smooth part of the function --- its Taylor expansion --- before comparing it to a power of the radius (see \cite{balanca_2-microlocal_2012} for an in-depth comparison in the case $\T=\R_+$). 
However, the set-indexed setting does not seem to have any natural substitutes for polynomials, hence the definition. 
Moreover, keeping the polynomial part has even proven to be useful to study stochastic processes when $\T=\R_+$ (see \cite{balanca_2-microlocal_2012}).

\medskip

\cite{herbin_local_2016} also introduced the \emph{(pointwise) Hölder $\C$-exponent} in order to look at the variation of $h$ in terms of the class $\C$ and the associated increment map $\Delta h$ indexed by $\C.$
\cite[Proposition 3.2]{herbin_local_2016} actually proves that the following definition does not depend on the choice of $k\in\N$:
\begin{equation}	\label{eq:holder_C-exponent}
	\forall A\in\A,\quad \alpha_{h,\C}(A) \speq \sup\left\{ \alpha\gs0 \,:\, \limsup_{\rho\rightarrow0^+} \sup_{C\in \C_k\cap B_\C(A,\rho)}\frac{\big| \Delta h(C) \big|}{\rho^\alpha} <\infty \right\}
\end{equation}

where the class $\C_k$ has been given in Definition \ref{de:other_class_C}.
Such a definition leads to the apparently stronger than (\ref{eq:holder_exponent_estimate}) corresponding estimate for $\alpha\in(0,\alpha_{h,\C}(A))$: 
\begin{equation}	\label{eq:holder_C-exponent_estimate}
	\forall k\in\N,\,\exists\rho_{\alpha,k}>0 : \forall C\in\C_k\cap B_\C(A,\rho_{\alpha,k}),\quad \big| \Delta h(C) \big| ~\ls~ d_\C(A,C)^\alpha.
\end{equation}

A reason why it is preferred over a more natural definition on $\C$ is that $\C$ is not a Vapnick-\v Cervonenkis class since $\C^\ell$ is a dissecting system, so $\C$-indexed processes are far from having continuous sample paths in general (see \cite{adler_random_2007} for more details). 

Finally, it has not been seen in \cite{herbin_local_2016} that the usual Hölder exponent and the $\C$-exponent actually coincide.

\begin{prop}	\label{prop:holder_exponents_equivalence}
	For all $A\in\A,\,$ $\alpha_h(A) \,=\, \alpha_{h,\C}(A).$
\end{prop}

In particular, we will only mention $\alpha_h(A)$ in the following and still use both estimates (\ref{eq:holder_exponent_estimate}) and (\ref{eq:holder_C-exponent_estimate}).

\begin{proof}
	Let $A\in\A.$
	Then
	
	\vspace{-0.5cm}
	\begin{spacing}{1.5}
	\[\begin{array}{rcl}
		\ds\sup_{C\in \C_0\cap B_\C(A,\rho)} \big| \Delta h(C) \big|\rho^{-\alpha}
		& \speq	&
		\ds\sup_{\substack{A_0,A_1\in B_\A(A,\rho): \\ A_0\subseteq A_1}} \big| \Delta h(A_1\setminus A_0) \big|\rho^{-\alpha} 	\\
		& =   	& 
		\ds\sup_{\substack{A_0,A_1\in B_\A(A,\rho): \\ A_0\subseteq A_1}}\big| h(A_1)-h(A_0) \big|\rho^{-\alpha} 				\\
		& ~\ls~	& 
		\ds\sup_{\substack{A_0,A_1\in B_\A(A,\rho): \\ A_0\subseteq A_1}} \Big[\big| h(A_1)-h(A) \big|+\big| h(A)-h(A_0) \big|\Big] \rho^{-\alpha} 													\\
		& \ls 	&
		2 \ds\sup_{A'\in B_\A(A,\rho)} \big| h(A)-h(A') \big|\rho^{-\alpha}.
	\end{array}\]
	\end{spacing}
	\vspace{-0.2cm}
	
	Hence taking $k=0$ in (\ref{eq:holder_C-exponent}) immediately yields that $\alpha_h(A)\ls\alpha_{h,\C}(A).$
	
	Conversely, if $\alpha_{h,\C}(A)=0,$ then equality immediately holds. 
	Otherwise, take $\alpha\in(0,\alpha_{h,\C}(A)).$
	Fix $\rho_{\alpha,0}>0$ just as in the estimate (\ref{eq:holder_C-exponent_estimate}) for $k=0.$
	Let $A'\in B_\A(A,\rho_{\alpha,0}).$ 
	Then,
	\begin{equation}	\label{eq:increment_bound}
		\big| h(A)-h(A') \big| ~\ls~ \big| \Delta h(A\setminus A') \big| \+ \big| \Delta h(A'\setminus A) \big|.
	\end{equation}
	Since the extremal representation of $A\setminus A'$ is $A\setminus (A\cap A'),$ it follows by definition of $d_\C$ that
	\[ d_\C(A\setminus A',A) \speq \max\Big\{ d_\A(A,A'),\, d_\A(A\cap A',A) \Big\}. \]
	By contractivity, it follows that $d_\A(A\cap A',A)\ls d_\A(A,A'),$ so
	$d_\C(A\setminus A',A) = d_\A(A,A').$
	Similarly, $d_\C(A'\setminus A,A) = d_\A(A,A').$
	In particular, we may apply the estimate (\ref{eq:holder_C-exponent_estimate}) to (\ref{eq:increment_bound}) and get
	\[ \forall A'\in B_\A(A,\rho_{\alpha,0}),\quad |h(A)-h(A')| ~\ls~ 2d_\A(A,A')^\alpha. \]
	Hence $\alpha\ls\alpha_h(A),$ the result follows.
\end{proof}

\begin{rem}		\label{rem:local_exponent_no_need}
	Taking $\alpha\in(0,\widetilde\alpha_h(A))$ where $\widetilde\alpha_h(A)$ is the local analog of $\alpha_h(A)$ (see \cite{herbin_local_2016} and references therein for a precise definition) would yield an estimate similar to (\ref{eq:holder_exponent_estimate}) for all $\alpha\in(0,\widetilde\alpha_h(A))$
	\[ \forall A_0,A_1\in B_\A(A,\rho_\alpha),\quad \big| h(A_0)-h(A_1) \big| ~\ls~ d_\A(A_0,A_1)^\alpha \]
	from which one could deduce an estimate similar to (\ref{eq:holder_C-exponent_estimate}) of the form
	\[ \forall C,C'\in \C_k\cap B_\C(A,\rho_{\alpha,k}),\quad \big| \Delta h(C)-\Delta h(C') \big| ~\ls~ d_\C(C,C')^\alpha. \]
	In particular, if $h$ has well-defined jumps (Definition \ref{de:well-defined_jumps}) and $\A$ has finite dimension, it follows from the previous estimate and Lemma \ref{lem:d_C} that $s\mapsto J_s(h)=\lim_{n\rightarrow\infty}\Delta h(C_n(s))$ is Hölder-continuous in a neighborhood of $A$ (\ie for points $s$ such that $A(s)$ is in a neighborhood of $A$ for $d_\A$).
	However, from Corollary \ref{cor:jump_structure_Y}, $Y$ has well-defined jumps, so $s\mapsto J_s(Y)$ cannot be continuous unless it is null everywhere. 
	This happens only if $\nu=0,$ in which case $Y$ is Gaussian and \cite{herbin_local_2016} shows that it generally implies that $\alpha_Y=\widetilde\alpha_Y.$
	So local exponents do not constitute the right tools here.
\end{rem}

\subsubsection{$d_\T$-localized exponent}

In Section \ref{subsection:ptw_exponent}, we talked about the problem that characterizing the regularity of $Y$ through increments of the form $Y_A-Y_{A'}$ --- which is the case for $\alpha_Y(A)$ --- requires to take non-local information into account. 
We introduced in Definition \ref{de:divergence} the notion of \emph{victiny} especially to tackle this issue.
Another way to solve the problem is to swap the increment $Y_A-Y_{A'}$ by $\Delta Y_C$ for $C\in\C$ close to $A$ and of small diameter.
Not only the shrinking mesh property (Definition \ref{de:compatible_metric}) will ensure that this definition is well-posed, but also $\Delta Y_C$ actually constitutes the right notion of increment in the set-indexed setting. This fact has already been noted in the study of two-parameter processes (see the introduction of Section \ref{sec:set-indexed}). 

\medskip

With that in mind, we define the \emph{$d_\T$-localized exponent} of $h:\A\rightarrow\R$:
\begin{equation}	\label{eq:localized_exponent}
	\forall A\in\A,\quad \alpha_{h,d_\T}(A) \speq \sup\left\{ \alpha\gs0 \,:\, \limsup_{\rho\rightarrow0^+}\sup_{\substack{C\in\C_{p-1}\cap B_\C(A,\rho): \\ \text{diam}(C)<\rho}}\frac{|\Delta h(C)|}{\rho^\alpha}<\infty \right\}.
\end{equation}
where we recall that $p=\dim\A$ (Definition \ref{de:indexing_collection_finite_dimension}) and $\text{diam}(C) = \sup\{ d_\T(s,s'):s,s'\in C \}.$
Remark that, contrary to (\ref{eq:holder_C-exponent}) which defines the $\C$-exponent, this definition does depend on $p$ due to the condition on the diameter.
Indeed, elements of $\C_{k+1}$ with diameter smaller than $\rho$ cannot in general be only expressed using elements of $\C_k$ with diameter smaller than $\rho$.

\medskip

The following proposition gives some properties about the $d_\T$-localized exponent in order to get a better feel for it.

\begin{prop}	\label{prop:localized_exponent_properties}
	Let $A=A(t)\in\A.$ The following properties hold:
	\begin{enumerate}
		\item (Equivalent definition using the TIP bijection).
			\begin{equation}	\label{eq:localized_exponent_equivalent_definition}
				\alpha_{h,d_\T}(A(t)) \speq \sup\left\{ \alpha\gs0 : \limsup_{\rho\rightarrow0^+}\sup_{\substack{C\in\C_{p-1}\cap B_\C(A(t),\rho): \\ C\subseteq B_\T(t,\rho)}}\frac{|\Delta h(C)|}{\rho^\alpha}<\infty \right\}.
			\end{equation}		
		\item (Corresponding estimate).
			For all $\alpha\in(0,\alpha_{h,d_\T}(A))$, there exists $\rho_\alpha>0$ such that for all $C\in\C_{p-1}\cap B_\C(A,\rho_{\alpha}),$
			\begin{equation}	\label{eq:localized_exponent_estimate}
				C\subseteq B_\T(t,\rho_{\alpha}) \quad\Longrightarrow\quad \big| \Delta h(C) \big| ~\ls~ \Big( \max\big\{ d_\C(A,C),~ \text{diam}(C) \big\} \Big)^\alpha.
			\end{equation}
			Conversely, (\ref{eq:localized_exponent_estimate}) implies $\alpha\ls\alpha_{h,d_\T}(A).$

		\item (Comparison to the Hölder exponent).
		$\alpha_h(A) \ls \alpha_{h,d_\T}(A).$
	\end{enumerate}
\end{prop}

\begin{proof}
	Let us fix $A=A(t)\in\A.$
	\begin{enumerate}
		\item It is just a consequence of the fact that for all $\rho>0$ and $C\in B_\C(A,\rho),$
		\[\begin{array}{rcl}
			\text{diam}(C)<\rho 	
			& ~\Longrightarrow~ & C\subseteq B_\T(t,2\rho)	\\
			C\subseteq B_\T(t,\rho)	
			& ~\Longrightarrow~	& \text{diam}(C) \ls2\rho
		\end{array}\]
		
		\item It is straightforward consequence of (\ref{eq:localized_exponent_equivalent_definition}). 
		
		\item Since $\C^\ell\subseteq\C_{p-1},$ it is a but a simple comparison between (\ref{eq:holder_C-exponent}) for $k=p-1$, Proposition \ref{prop:holder_exponents_equivalence} and (\ref{eq:localized_exponent}).
	\end{enumerate}
\end{proof}

This exponent should also be compared to the \emph{pointwise continuity exponent} introduced in \cite[Definition 3.4]{herbin_local_2016}:

\begin{equation} 	\label{eq:ptw_continuity_exponent}
	\forall t\in\T,\quad \alpha_h^{pc}(t) \speq \sup\left\{ \alpha\gs0 \,:\, \limsup_{n\rightarrow\infty} \frac{|\Delta h(C_n(t))|}{m(C_n(t))^\alpha}<\infty \right\}.
\end{equation}

However, we found our definition easier to work with since it is more closely linked to a metric, does not directly rely on the countable class $\C^\ell$ and yields a more powerful estimate at the end while still answering our need to replace $Y_A-Y_{A'}$ with a better, more local, notion of increment.

\subsection{Regularity of generic OJI functions}	\label{subsection:oji_function}

In \cite{jaffard_old_1997}, Jaffard used the discontinuities of a càdlàg function $h:\R_+\rightarrow\R$ to obtain the following upper bound on its Hölder exponent:

\begin{equation}	\label{eq:jaffard_lemma_1d}
	\forall t\in\R_+\setminus\Pi(h),\quad \alpha_h(t) ~\ls~ \liminf_{s\in\Pi(h)\rightarrow t} \frac{\log|J_s(h)|}{\log|s-t|}
\end{equation}

where $J_s(h)=h(s)-h(s^-)$ for all $s\in\R_+.$
In this section, we strive to generalize this approach in order to use it in a similar fashion as \cite{jaffard_multifractal_1999}.
We start by treating the Hölder exponent only. We explain how to do a similar (and simpler) study of the $d_\T$-localized exponent in Section \ref{subsection:localized_regularity}.

\subsubsection{Jump sets in generic configuration}

In order to adapt (\ref{eq:jaffard_lemma_1d}) to a more general setting, we could consider the point-mass jumps $J_s(h)=\lim_{n\rightarrow\infty}\Delta h(C_n(t))$ of a map $h:\A\rightarrow\R$ with well-defined jumps (Definition \ref{de:well-defined_jumps}) in a neighborhood of $t$ and reproduce Jaffard's proof.
This would yield the following bound:

\begin{equation} 	\label{eq:jaffard_lemma_naive}
	\forall t\in\T,\quad \alpha_h(t) ~\ls~ \liminf_{\substack{s\in\Pi(h): \\ d_\T(s,t)\rightarrow0^+}} \frac{\log|J_s(h)|}{\log d_\T(s,t)}  \vee0.
\end{equation}

Recall that $\alpha_h(A)$ and $\alpha_h(t)$ are the same --- provided that $A=A(t)$ --- due to the correspondance (\ref{eq:d_T}) between $d_\A$ and $d_\T$.
Moreover, the `$\vee\,0$' ensures that this inequality holds for all $t\in\T$. 

\medskip

Alas, unless $\T$ is one-dimensional, this upper bound turns out not to be so sharp.
The reason is that we failed to consider the majority of the point-mass jumps contributing to lessen the regularity, \ie the ones in the \emph{victiny} of $A=A(t)$ (Definition \ref{de:divergence}).
Continuing our illustration with $(\T,d_\T)=(\R_+^2,d_2)$ from Example \ref{ex:balls_qd}, the area of $B_\T(t,\rho)$ in Figure \ref{fig:victiny} is of order $\rho^2$ whereas the area of $\V(A,\rho)$ is of order $\rho$ as $\rho\rightarrow0,$ so not taking the jumps in $\V(A,\rho)\setminus B_\T(t,\rho)$ into account incurs severe losses in the sharpness of the argument.

So in order to obtain a better upper bound on $\alpha_h(A),$ we need to be able to `fetch' the jumps of $h$ in the victiny of $A$ while only using elements in $B_\A(A,\rho)$ for small $\rho.$

\begin{de}[OJI function]
	A map $h:\A\rightarrow\R$ with well-defined jumps is said to be \emph{only jump-irregular (OJI)} if it can be written in the form:
	\[ \forall A\in\A,\quad h(A)\speq \lim_{\eps\rightarrow0^+} \Bigg[ \sum_{\substack{s\in\Pi(h)\cap A: \\ |J_s(h)|\gs\eps  }} J_s(h) - a(A,\eps) \Bigg] \]
	where $a:\A\times(0,1)\rightarrow\R$ is a continuous function (where $\A$ is endowed with the metric $d_m=m(.\triangle.)$) and the limit as $\eps\rightarrow0^+$ happens almost surely uniformly in $[\varnothing,A]$ for all $A\in\A.$
\end{de}

The term `only jump-irregular' is meant to indicate that OJI functions are only allowed to have discontinuities in the form of point-mass jumps.
 
Using the uniform convergence as $\eps\rightarrow0^+$ and the fact that $\A$ is finite-dimensional (Definition \ref{de:indexing_collection_finite_dimension}), one readily checks that for any OJI function $h$ and $s\in\T,\,$ $\Delta h(C_n(s))\rightarrow J_s(h)$ as $n\rightarrow\infty.$

\medskip

In order to improve on (\ref{eq:jaffard_lemma_naive}) and take all the jumps of $h$ in the victiny of $A\in\A$ into account, it may happen that some jumps cannot be picked separately from the viewpoint of a given $A$.
It is the case for instance whenever $J_s(h)$ and $J_{s'}(h)$ are both non-zero in Figure \ref{fig:aligned_jumps}.

%

\begin{figure}[!h]
\centering
	\begin{tikzpicture}[scale=.85]	
		\draw [->]				(-.5,0)--(5.5,0);
		\draw [->]				(0,-.5)--(0,3.5);
		\draw [<->]				(2.7,2.9)--(4.3,2.9);
		\draw (3.5,2.9) 		node[above]				{$2\rho$};
		\draw [densely dotted]	(2.7,1.2)--(2.7,2.9);
		\draw [densely dotted]	(4.3,2)--(4.3,2.9);
		\draw [<->]				(4.4,1.2)--(4.4,2.8);
		\draw (4.4,2) 			node[right]				{$2\rho$};
		\draw [densely dotted]	(2.7,1.2)--(4.4,1.2);
		\draw [densely dotted]	(3.5,2.8)--(4.4,2.8);
		\fill [pattern=north east lines,opacity=.2] (4.3,0)--(4.3,2)--(4.3,2) arc (0:90:.8)--(3.5,2.8)--(0,2.8)--(0,1.2)--(2.7,1.2)--(2.7,0)--cycle;
		\fill [pattern=north west lines,opacity=.2] (3.5,2) circle (.8);
		\draw (.5,2.3)		node[above,fill=white]{$s$}		node{$\times$} ;
		\draw (1.75,2.3)	node[above,fill=white]{$s'$}	node{$\times$};
		\draw (-.7,2.3)		node[left]	{$\mathcal L(s)=\mathcal L(s')$};
		\draw [->]				(-.7,2.3)--(-.2,2.3);
		\draw 					(0,2.3)--(3,2.3)--(3,0);		
		\draw [thick]			(0,2.8)--(3.5,2.8);
		\draw [thick]			(4.3,2)--(4.3,0);
		\draw [thick]			(4.3,2) arc (0:90:.8);
		\draw [thick]			(0,1.2)--(2.7,1.2)--(2.7,0);
		\draw (3.5,2)			node[fill=white,above]	{$t$} node {$\times$};
		\draw [dashed,thick]	(0,2)--(3.5,2)--(3.5,0);
	\end{tikzpicture}
	
	\caption{$s,s'\in\V(A(t),\rho)$ cannot be isolated from one another using elements in $B_\A(A(t),\rho)$.}
	\label{fig:aligned_jumps}
\end{figure}
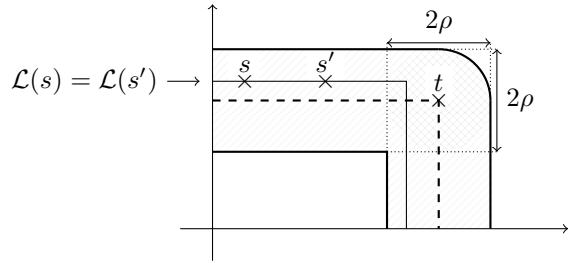

\begin{de}[Generic configuration]	\label{de:generic_configuration}
	For a given $A\in\A,$ an OJI function $h:\A\rightarrow\R$ is said to be \emph{in generic configuration in the victiny of $A$} if there exists $\rho_0>0$ such that for all for all $\rho\in(0,\rho_0),$ $s\in\Pi(h)\cap\V(A,\rho)$ and $\eps>0,$
	\[ \exists \underline A,\overline A\in B_\A(A,\rho) :\quad \underline A\subseteq \overline A \quad\text{ and }\quad \big|\Delta h(\overline A\setminus\underline A) - J_s(h)\big| \ls\eps. \]
\end{de}

Remark that they are OJI functions that are not in generic configuration in the victiny of $A.$ For instance, take the function defined for all $A\in\A$ by $h(A)=\myindic{s\in A}+\myindic{s'\in A}$ for $s$ and $s'$ as in Figure \ref{fig:aligned_jumps}.

\begin{theo}[Generic configuration for $X$ and $Y$]	\label{theo:generic_configuration}
	For all $A\in\A,$ the Poissonian parts of $X$ and $Y$ are almost surely in generic configuration in the victiny of $A.$
\end{theo}

Before proving this theorem, we need to introduce some notations in order to correctly define:
\begin{enumerate}[$\diamond$]
	\item The `boundary' --- later called $\mathcal L(s)$ --- on which $s$ (and $s'$) stands in Figure \ref{fig:aligned_jumps}.
	\item Approximating sequences $(\underline A_n(s))_{n\in\N}$ and $(\overline A_n(s))_{n\in\N}$ such that $\Delta h(\overline A_n(s)\setminus\underline A_n(s))\rightarrow J_s(h)$ as $n\rightarrow\infty.$ The sets $\overline A_n(s)\setminus\underline A_n(s)\in\C_0$ should be thought of as `thick' versions of $\mathcal L(s)$ converging to $\mathcal L(s)$ as $n\rightarrow\infty.$
\end{enumerate}

\medskip

Let $A\in\A,$ $\rho>0$ and $s\in\V(A,\rho)$ be fixed.
By Proposition \ref{prop:property_V}, we may write $\V(A,\rho)=\bigcup_{n\in\N}\V_n(A,\rho)$ where
\[ \forall n\in\N,\quad \V_n(A,\rho) ~= \bigcup_{\substack{\overline A,\underline A\in \A_n\cap B_\A(A,\rho): \\ \underline A \text{ maximal proper subset of } \overline A }} \hspace{-.2cm} \big(\overline A\setminus\underline A\big) \]
and the `maximal proper subset' condition is not a restriction since it just eliminates redundancy in the definition of $\V_n(A,\rho).$

From this expression, there exists a non-decreasing sequence $(\overline A_n(s)\setminus\underline A_n(s))_{n\gs n_0}$ such that for all $n\gs n_0,~$ $\underline A_n(s),\overline A_n(s)\in\A_n\cap B_\A(A,\rho),$ $\underline A_n(s)$ is a maximal proper subset of $\overline A_n(s)$ and $s\in\overline A_n(s)\setminus\underline A_n(s).$
Remark that the number of choices to define such sequences for all $s\in\V(A,\rho)$ can be made countable by using the (arbitrary) total order on the $\A_n$'s given in Definition \ref{de:indexing_collection}. 

\medskip

Denote, as intended,
\[ \mathcal L(s) \speq \mathcal L(A,\rho,s) \speq \bigcap_{n\in\N} \big(\overline A_n(s)\setminus\underline A_n(s)\big). \]
Since $d_m$ is compatible (Definition \ref{de:compatible_metric}), Proposition \ref{prop:weak_shrinking_mesh} ensures that $m(\mathcal L(s))=0.$

\begin{lem}	\label{lem:generic_configuration_line}
	A sufficient condition for an OJI function $h$ to be in generic configuration in the victiny of $A\in\A$ is that there exists $\rho_0>0$ such that for all $\rho\in(0,\rho_0)\cap\Q$ and $s\in\Pi(h)\cap\V(A,\rho),$ 
	\[ \Pi(h)\cap\mathcal L(A,\rho,s) \speq \{ s \}. \]
\end{lem}

\begin{proof}
	Let us fix $\rho\in(0,\rho_0)\cap\Q$ and $s\in\Pi(h).$
	Since $h$ is OJI, we may write
	\begin{equation} 	\label{eq:OJI_thick_line} 
		\Delta h\big( \overline A_n(s)\setminus\underline A_n(s) \big) ~=~ \lim_{\eps\rightarrow0^+} \Bigg[ \sum_{\substack{t\in\Pi(h)\cap( \overline A_n(s)\setminus\underline A_n(s)): \\ |J_t(h)|\gs\eps  }} \hspace{-.3cm} J_t(h) - \Delta a\big(\overline A_n(s)\setminus\underline A_n(s),\eps\big) \Bigg].  
	\end{equation}
	where the convergence as $\eps\rightarrow0^+$ happens uniformly in $n\gs n_0.$
	
	Since $\overline A_n(s)$ decreases to $\bigcap_{n\in\N}\overline A_n(s),$ we deduce from the outer continuity of $d_m$ that
	\[ d_m \Big( \overline A_n(s),\bigcap_{n\in\N}\overline A_n(s) \Big) \longrightarrow0 \spas n\rightarrow\infty. \]
	Moreover, the shrinking mesh property tells that $d_m(\underline A_n(s),\overline A_n(s))$ and hence
	\[ d_m \Big( \underline A_n(s),\bigcap_{n\in\N}\overline A_n(s) \Big) \longrightarrow0 \spas n\rightarrow\infty. \]
	
	Hence, due to the continuity of $a(.,\eps)$ with respect to $d_m$, taking $n\rightarrow\infty$ in (\ref{eq:OJI_thick_line}) and using $\Pi(h)\cap\mathcal L(s)=\{ s \}$ yields
	\[ \lim_{n\rightarrow\infty} \Delta h\big( \overline A_n(s)\setminus\underline A_n(s) \big) ~=~ \lim_{\eps\rightarrow0^+} \sum_{\substack{t\in\Pi(h)\cap\mathcal L(s): \\ |J_t(h)|\gs\eps  }} \hspace{-.3cm} J_t(h) ~=~ J_s(h).  \]
	\vspace{-.2cm}
	
	Hence $h$ is in generic configuration in the victiny of $A.$
\end{proof}

\begin{proof}[Proof of Theorem \ref{theo:generic_configuration}]
	We only prove it for $X,$ the proof for $Y$ is exactly the same.

	By the Lévy-decomposition (\ref{eq:levy-ito_decomposition_X}), the Poissonian part of $X$ is OJI.

	Fix $A\in\A$ and $\rho>0.$ 
	Denote $\Pi=\{ s_n : n\in\N \}$ and the event
	\[ \Omega_\rho \speq \Big\{ \exists s,s'\in\Pi\cap\V(A,\rho) : s\neq s' \text{ and } s'\in\mathcal L(s,A,\rho) \Big\}. \]
	Since $m(\mathcal L(s))=0$ for all $s\in\V(A,\rho),$
	\[ \prob{\Omega_\rho} ~\ls~ \sum_{i\neq j} \prob{s_j\in\mathcal L(s_i,A,\rho)} ~=~ 0. \]
	Hence $\Omega^*=\bigcap_{\rho\in\Q_+^*}\Omega_\rho^\complement$ is an event of probability one.
	Lemma \ref{lem:generic_configuration_line} then implies that for all $\omega\in\Omega^*,$ $X(\omega)$ is in generic configuration in the victiny of $A\in\A.$	
\end{proof}

\subsubsection{Bounding regularity with point-mass jumps}

As promised, the study of generic OJI functions yields a better bound than (\ref{eq:jaffard_lemma_naive}).

\begin{theo}	\label{theo:jaffard_lemma}
	Let $A\in\A,~$ $h:\A\rightarrow\R$ an OJI function in generic configuration in the victiny of $A$ and a sequence $(s_n)_{n\in\N}\in\Pi(h)^\N$ such that $\qd(s_n,A)\rightarrow0$ as $n\rightarrow\infty.$
	Then,
	\[ \alpha_h(A) ~\ls~ \liminf_{n\rightarrow\infty} \frac{\log|J_{s_n}(h)|}{\log \qd(s_n,A)}\vee0. \]
\end{theo}

\begin{proof}	
	The case $\alpha_h(A)=0$ is trivial.
	Otherwise, take $\alpha\in(0,\alpha_h(A)).$
	Let us consider $\rho_{\alpha,0}>0$ such that the estimate (\ref{eq:holder_C-exponent_estimate}) holds for $k=0.$
	
	Without loss of generality, we may suppose that for all $n,$ $\qd(s_n,A)<\min\{1,\rho_0,\rho_{\alpha,0}\}$ where $\rho_0$ is the one of Definition \ref{de:generic_configuration}.
	
	Let $\eps>0$ and $n\in\N.$ Estimate (\ref{eq:holder_C-exponent_estimate}) yields
	\begin{equation} 	\label{eq:holder_estimate_qd_jaffard} 
	\forall \underline A,\overline A\in B_\A(\qd(s_n,A)+\eps),\quad \big| \Delta h\big(\overline A\setminus\underline A\big) \big| ~\ls~ \big( \qd(s_n,A)+\eps \big)^\alpha. 
	\end{equation}
	
	Since $s_n\in\Pi(h)\cap\V(A,\qd(s_n,A)+\eps)$ and $h$ is in generic configuration, we may find $\underline A$ and $\overline A$ in $B_\A(\qd(s_n,A)+\eps)$ such that
	\begin{equation}	\label{eq:generic_configuration_jaffard}
		\big|\Delta h(\overline A\setminus\underline A) - J_{s_n}(h)\big| \ls\eps.
	\end{equation}
	Combining (\ref{eq:holder_estimate_qd_jaffard}) and (\ref{eq:generic_configuration_jaffard}) and taking $\eps\rightarrow0^+$ yields $\big| J_{s_n}(h) \big| ~\ls~ \qd(s_n,A)^\alpha$ and thus
	\[ \alpha ~\ls~ \frac{\log|J_{s_n}(h)|}{\log\qd(s_n,A)}. \]
	The result follows from taking lower limit as $n\rightarrow\infty$ and then $\alpha\rightarrow\alpha_h(A)^-.$
\end{proof}

Following ideas from \cite{jaffard_multifractal_1999}, let us introduce for any given map $h:\A\rightarrow\R$ with well-defined jumps, measurable set $L\subseteq\T$ and $\delta>0,$
\begin{equation}	\label{eq:E^delta_j}
	\forall j\in\N,\quad E_{j|L}^\delta(h) ~= \bigcup_{\substack{s\in\Pi(h)\cap L: \\ |J_s(h)|\in\Gamma_j}} \V'\big( s,|J_s(h)|^\delta \big)
\end{equation}
where $\V'$ is the dual victiny given in (\ref{eq:dual_victiny}) and
\begin{equation}	\label{eq:Gamma_j}
	\forall j\in\N,\quad \Gamma_j \speq \left\{ x\in\R : 2^{-j}\ls|x|<2^{-(j-1)} \right\}.
\end{equation}

Let us also introduce
\begin{equation}	\label{eq:E^delta} 
	E_{|L}^\delta(h) \speq \limsup_{j\rightarrow\infty} E_{j|L}^\delta(h) ~=~ \bigcap_{k\in\N}\bigcup_{j\gs k}E_{j|L}^\delta(h).
\end{equation}

The set $L$ allows to select the jumps of $X$ falling into a specific region depending on the behavior of $f.$ It will be precised later on while proving the upper bound in Theorem \ref{theo:regularity_Y_poissonian}.

\medskip

Combined with Theorem \ref{theo:jaffard_lemma}, the following result yields an upper bound on the Hölder exponent (much like (11) in \cite{jaffard_multifractal_1999}).

\begin{prop}	\label{prop:jaffard_bound}
	If $h:\A\rightarrow\R$ has well-defined jumps and $A\in\A,$ then
	\[ A\in E_{|L}^\delta(h) \quad\Longrightarrow\quad \liminf_{\substack{s\in\Pi(h)\cap L: \\ \qd(s,A)\rightarrow0}} \frac{\log|J_s(h)|}{\log\qd(s,A)} ~\ls~ \frac{1}{\delta}. \]
\end{prop}

\begin{proof}
	Let $A\in E_{|L}^\delta(h).$
	Then there exists an increasing sequence $(j_k)_{k\in\N}$ in $\N$ and a sequence $(s_k)_{k\in\N}\in \lp\Pi(h)\cap L\rp^\N$ such that
	\[ \forall k\in\N,\quad 2^{-j_k}<|J_{s_k}(h)|\ls2^{-(j_k-1)} \quad\text{and}\quad \qd(s_k,A) \ls |J_{s_k}(h)|^\delta. \]
	So $\qd(s_k,A)\rightarrow 0$ and $|J_{s_k}(h)|\rightarrow0^+$ as $k\rightarrow\infty.$
	In particular, there exists $k_0$ such that 
	\[ \forall k\gs k_0,\quad \frac{\log|J_{s_k}(h)|}{\log\qd(s_k,A)} ~\ls~ \frac{1}{\delta}. \]
	The result follows from taking the lower limit in the above inequality.
\end{proof}

\section{Regularity of the process $Y=\ds\int_.f\,\text dX$} \label{sec:regularity_thm}

The goal of this section is to characterize the almost sure regularity of $Y$ --- as defined in Section \ref{sec:regularity_criterion} --- in terms of the behaviors of both $f$ and $X$.
In Section \ref{subsection:0-1_law}, we establish a 0-1 law of independent interest to simplify the proofs that follows in the subsequent parts. 
Section \ref{subsection:ptw_regularity} gives bounds on the Hölder regularity of $Y$. 
Section \ref{subsection:localized_regularity} does the same work for the $d_\T$-localized exponent.
Finally, Section \ref{subsection:examples} is devoted to some examples and applications of the main results.

\subsection{A 0-1 law}	\label{subsection:0-1_law}

Let us start by recalling the Lévy-Itô decomposition (\ref{eq:levy-ito_decomposition_X}) of $X$ given by
\[ \forall A\in\A,\quad X_A \speq bm(A) \+ \sigma B_A \+ N_A \]
where such a writing is uniquely characterized (distribution-wise) by the Lévy-Khintchine triplet $(b,\sigma^2,\nu)$ of $X.$ From it, we established a Lévy-Itô decomposition (Theorem \ref{theo:levy-ito_decomposition}) for $Y$ given by

\vspace{-1.1cm}
\begin{spacing}{2.1}
\[	
	\forall A\in\A,\quad Y_A \speq
	\left\{
	\begin{array}{ll}
		\ds b\int_A f\,\text dm \+ \sigma\int_A f\,\text dB \+ \int_A f\,\text dN			
		& \quad\text{ if } \ds\int_{|x|\ls1}|x|\nu(\text dx)=\infty, 	\\
		\ds \widetilde b\int_A f\,\text dm \+ \sigma\int_A f\,\text dB \+ \int_A f\,\text d\widetilde N	
		& \quad\text{ if } \ds\int_{|x|\ls1}|x|\nu(\text dx)<\infty.
	\end{array}
	\right. 
\]
\end{spacing}
\vspace{-0.4cm}

where $\widetilde N$ is the non-compensated version of $N$ defined in Corollary \ref{cor:levy-ito_decomposition_integrable_case} and $\widetilde b$ is modified accordingly.

\medskip

In the case where $\T=\R_+$ and $m$ is the Lebesgue measure, asking the regularity of the drift part of $Y$ (the first term) is the same as asking the regularity of a primitive of $f$. 
This problem has been entirely  dealt with through the use of a tool called the \emph{2-microlocal frontier} which characterizes how the regularity evolves when one takes fractional integrals and/or derivatives of $f$.
The 2-microlocal formalism dates back to \cite{mizohata_second_1986} and has seen a lot of developments throughout the years (see for instance \cite{jaffard_pointwise_1991, levy_vehel_2-microlocal_2004, meyer_wavelets_1998} and references therein).
In order to do the same in the set-indexed setting, one would need to develop an analog for the 2-microlocal frontier. 
Being an entirely deterministic endeavor, we chose to push it aside for this article. 

Namely, we are going to cancel the drifts in the expression of $Y$ (\ie take either $b=0$ or $\widetilde b=0$ depending on the case), giving the following simpler expressions:

\vspace{-1.1cm}
\begin{spacing}{2.1}
\begin{equation}	\label{eq:Y_simplified}	
	\forall A\in\A,\quad Y_A \speq
	\left\{
	\begin{array}{ll}
		\ds \sigma\int_A f\,\text dB \+ \int_A f\,\text dN			
		& \quad\text{ if } \ds\int_{|x|\ls1}|x|\nu(\text dx)=\infty, 	\\
		\ds \sigma\int_A f\,\text dB \+ \int_A f\,\text d\widetilde N	
		& \quad\text{ if } \ds\int_{|x|\ls1}|x|\nu(\text dx)<\infty.
	\end{array}
	\right. 
\end{equation}
\end{spacing}
\vspace{-0.4cm}

Hence we are going to express the regularity of $Y$ as given in (\ref{eq:Y_simplified}) in terms of $(\sigma^2,\nu)$ --- which characterizes the law of $X$ --- and $f$.

\medskip

The following 0-1 law, which happens to be Blumenthal's when $\T=\R_+$, is a first clue which states that the regularity of $Y$ at a fixed $A\in\A$ must be determinisitic.

\begin{theo}[Set-indexed 0-1 law]	\label{theo:0-1_law}
	Let $A\in\A$ and define
	\[ \forall\rho>0,\quad \F_{(A,\rho)} \speq \sigma\big( Y_A-Y_{A'} : 0<d_\A(A,A')<\rho \big) \spand \F_{A+} \speq \bigcap_{\rho>0}\F_{(A,\rho)}. \]
	If $m(\V(A,\rho))\rightarrow0$ as $\rho\rightarrow0$, then any event in $\F_{A+}$ has either probability 0 or 1.
\end{theo}

\begin{proof}
	Define the $\sigma$-algebra
	\[ \F_{(A,\infty)} \speq \bigvee_{\rho>0}\F_{(A,\rho)}. \]
	Since $\F_{A+}\subseteq\F_{(A,\infty)},$ it is enough to prove that $\F_{A+}$ is independent from $\F_{(A,\infty)}.$
	
	First, remark that the family of cylinders 
	\[ \Big\{ Y_A-Y_{A_1}\in B_1,...,Y_A-Y_{A_k}\in B_k \Big\} \]
	where $k\in\N$, $A_1,...,A_k\in\A$, $B_1,...,B_k\in\B(\R)$ and $d_\A(A,A_j)>0$ for all $j\in\llbracket1,k\rrbracket$ is a $\pi$-system that generates $\F_{(A,\infty)}.$
	So by a monotone class argument, it is enough to show that $\F_{A+}$ is independent from $Y_A-Y_{A_1},...,Y_A-Y_{A_k}$ for some fixed $A_1,...,A_k\in\A$ such that $d_\A(A,A_j)>0$ for all $j\in\llbracket1,k\rrbracket.$ 
	
	Let $\rho>0$. Since $\F_{A+}\subseteq\F_{(A,\rho)}$ and $Y$ has independent increment (Proposition \ref{prop:independent_increments}), $\F_{A+}$ is independent from the random variables
	\[ \Delta Y_{A\setminus(A_j\cup\V(A,\rho))} - \Delta Y_{A_j\setminus(A\cup\V(A,\rho))} \]
	for all $j\in\llbracket1,k\rrbracket.$
	Since $m(\V(A,\rho))\rightarrow0$ as $\rho\rightarrow0$, $\F_{A+}$ is actually independent from the random variables
	\[ \lim_{\rho\rightarrow0} \big[\Delta Y_{A\setminus(A_j\cup\V(A,\rho))} - \Delta Y_{A_j\setminus(A\cup\V(A,\rho))}\big] \speq Y_A-Y_{A_j} \]
	for all $j\in\llbracket1,k\rrbracket$ where the limit holds in $L^2(\Omega)$ according to Theorem \ref{theo:integral_wrt_levy}.
	The result follows.
\end{proof}

A stronger condition than $m(\V(A,\rho))\rightarrow0$ as $\rho\rightarrow0$ is usually true in our setting and describes the local geometry around $A\in\A$. The speed at which the convergence holds actually influences the final regularity (see Corollaries \ref{cor:regularity_Y_poissonian_=} and \ref{cor:loc_regularity_Y_poissonian_=}).

\medskip

As an application of the 0-1 law, the following proposition ensures that both terms in (\ref{eq:Y_simplified}) may be treated separately.

\begin{prop}	\label{prop:separate_exponents}
	Let $A\in\A.$ If  $m(\V(A,\rho))\rightarrow0$ as $\rho\rightarrow0$, then
	\[ \alpha_Y(A) \speq \alpha_{\sigma\int_.f\,\text dB}(A) \wedge \alpha_{\int_.f\,\text dN}(A)  \quad\text{a.s.} \]
	The same holds if one replaces $N$ by $\widetilde N$ and/or all the Hölder exponents $\alpha_.(A)$ by the corresponding $d_\T$-localized exponents $\alpha_{.,d_\T}(A)$.
\end{prop}

\begin{proof}
	The other proofs being similar, we only prove it for $N$ and $\alpha_Y$. 
	
	Let $A\in\A.$	
	Since $\alpha_Y(A)$ is $\F_{A+}$-measurable, it is deterministic according to Theorem \ref{theo:0-1_law}. Replacing $Y$ by $\sigma\ds\int_.f\,\text dB$ (resp. $\ds\int_.f\text dN$) yields that $\alpha_{\sigma\int_.f\text dB}(A)$ (resp. $\alpha_{\int_.f\,\text dN}(A)$) is also constant. 
	Thus, there exist $\alpha_Y,\alpha_B,\alpha_N\in\R_+\cup\{\infty\}$ such that the event
	\[ \Omega^* \speq \Big\{ \alpha_Y(A)=\alpha_Y \Big\} \cap \Big\{ \alpha_{\sigma\int_.f\,\text dB}(A)=\alpha_B \Big\} \cap \Big\{ \alpha_{\int_.f\,\text dN}(A)=\alpha_N \Big\} \]
	happens with probability one.
	
	It is a classical (deterministic) result that 
	$\alpha_Y \gs \alpha_B\wedge\alpha_N$
	always holds and that there is equality whenever $\alpha_B \neq \alpha_N.$
	Suppose that $\alpha_Y>\alpha_B=\alpha_N$ and consider $\alpha=(\alpha_Y+\alpha_B)/2.$
	Since $\alpha>\alpha_N,$ for all $\omega\in\Omega^*,$ there exists a sequence $(A'_n(\omega))_{n\in\N}\in\A^\N$ such that $d_\A(A,A_n'(\omega))\rightarrow0^+$ as $n\rightarrow\infty$ and
	\[ \lim_{n\rightarrow\infty}~ \frac{1}{d_\A(A,A'_n(\omega))^\alpha}\left| \int_A f\,\text dN(\omega)-\int_{A'_n(\omega)}f\,\text dN(\omega) \right| \speq +\infty. \]
	Instead of extracting a subsequence and potentially replacing $Y$ by $-Y$, we might as well suppose that there exists an event $\Omega_N\subseteq\Omega^*$ of positive probability such that for all $\omega\in\Omega_N,$
	\begin{equation}	\label{eq:div_N} 
		\lim_{n\rightarrow\infty}~ \frac{1}{d_\A(A,A'_n(\omega))^\alpha}\left[ \int_A f\,\text dN(\omega)-\int_{A'_n(\omega)}f\,\text dN(\omega) \right] \speq +\infty.
	\end{equation}
	
	Let $\omega\in\Omega_N.$
	Since $\alpha<\alpha_Y$ and $\omega\in\Omega^*,$ (\ref{eq:div_N}) implies that
	\[ \lim_{n\rightarrow\infty}~ \frac{1}{d_\A(A,A'_n(\omega))^\alpha}\left[ \sigma\int_A f\,\text dB(\omega)-\sigma\int_{A'_n(\omega)} f\,\text dB(\omega) \right] \speq -\infty \]
	in order to compensate for the divergence.
	
	Since $B$ is symmetric and independent from $N$ and the sequence $(A'_n)_{n\in\N}$ only depends on $N,$ we may replace $B$ by $-B$ in the previous relation and obtain that for $\myprob$-almost every $\omega\in\Omega_N,$ 
	\[ \lim_{n\rightarrow\infty}~ \frac{1}{d_\A(A,A'_n(\omega))^\alpha}\left[ \sigma\int_A f\,\text dB(\omega)-\sigma\int_{A'_n(\omega)} f\,\text dB(\omega) \right] ~=~+\infty \]
	which is a contradiction since $\prob{\Omega_N}>0.$
	Hence $\alpha_Y=\alpha_N=\alpha_B.$
\end{proof}

\subsection{Hölder regularity}	\label{subsection:ptw_regularity}

As explained in Proposition \ref{prop:separate_exponents}, treating separately the two cases $\nu=0$ and $\sigma^2=0$ is enough to obtain a complete characterization of $Y$ apart from the drift, which is always supposed to be zero.

\subsubsection{The Gaussian part}

The case where $\nu=0$ has already been treated at great lengths in the litterature. 
For the set-indexed case, \cite[Corollary 5.3]{herbin_local_2016} ensures that under some entropic condition similar to Dudley's, for all $A\in\A,$ the following holds with probability one:
\begin{equation}	\label{eq:Y_gaussian_regularity}
	\alpha_Y(A) \speq \frac{1}{2}\alpha_{\sigma\int_{A\triangle.} f^2\,\text dm}(A).
\end{equation}

As already pointed out in Section \ref{subsection:0-1_law}, the Hölder exponent for $A'\mapsto\ds\int_{A\triangle A'}f^2\,\text dm$ cannot be deduced solely from the Hölder exponent of $f$. 
The knowledge of some kind of 2-microlocal frontier is required, hence (\ref{eq:Y_gaussian_regularity}) cannot be readily improved.
More precise results for the one-dimensional gaussian case concerning the 2-microlocal frontier especially adapted to our exponent are given in \cite{balanca_2-microlocal_2012}.

\subsubsection{The Poissonian part}

In this part, let us suppose that $Y$ is \emph{purely Poissonian}, \ie that we take $\sigma^2=0$ in the Lévy-Itô decomposition (\ref{eq:Y_simplified}).
The first study of Hölder regularity of a purely Poissonian Lévy process happened in \cite{blumenthal_sample_1961} where Blumenthal and Getoor determined the value of the pointwise Hölder exponent $\alpha_X(0)$ as defined by (\ref{eq:holder_exponent}). 
As explained in Section \ref{subsection:ptw_exponent}, there seem not to exist any natural extension of polynomials to the set-indexed setting, hence our choice to substract by hand the `polynomial part' in the same way as \cite{blumenthal_sample_1961}.
The authors also introduced the so-called \emph{Blumenthal-Getoor exponent}:

\begin{equation}	\label{eq:blumenthal-getoor_exponent}
	\beta \speq \inf\left\{ \delta>0 : \int_{|x|\ls1}|x|^\delta\nu(\text dx)<\infty \right\}.
\end{equation}

Remark that since $\nu$ is a Lévy measure, $\beta\in[0,2].$
For $\T=\R_+,$ Blumenthal and Getoor proved in particular that $\alpha_Y(0)=1/\beta$ almost surely together with the convention $1/0=+\infty$.
This result has been extended in \cite{balanca_fine_2014, jaffard_multifractal_1999, pruitt_growth_1981} in much greater detail.
Our goal is to extend those results in the case of the integral process $Y$, which to our knowledge has not been done even in $\R_+,$ and for possibly different spaces than $\R_+$.

\medskip

In the following, let us fix $A\in\A.$

Recall that the divergence $\qd$ and the victiny $\V(A,\rho)$ have been given in Definition \ref{de:divergence}. 

Theorem \ref{theo:regularity_Y_poissonian} will tell us that the Hölder exponent of $Y$ at $A$ is governed by the regularity of both $f$ and $X.$ Since $X$ has stationary increments, the exponent $\beta$ and some information about the victiny of $A$ will be enough to understand the regularity of $X$ (Corollary \ref{cor:regularity_Y_poissonian_=}). 
However for $Y,$ we need to know more about the behavior of $f$ in the victiny of $A.$ 
That is the reason why $L_{f,\alpha}(A)$ and $L_{f,\alpha}(A,\rho)$ are introduced below and correspond to the `irregular part' of $f.$
As for the sets $\overline R_{f}(A)$ and $\underline R_{f}(A),$ they determine the proportion of the victiny where $f$ is indeed irregular. 

Thus, we define for all $\alpha\gs0$ and $\rho>0,$

\vspace{-.5cm}
\begin{spacing}{1.8}
\[\begin{array}{ccc}
	L_{f,\alpha}(A) 
	& = & \Big\{ s\in\T : |f(s)|>\qd(s,A)^\alpha \Big\},	\\
	L_{f,\alpha}(A,\rho) 
	& = & \V(A,\rho)\cap L_{f,\alpha}(A),					\\
	\vspace{.2cm} L_{f,\alpha}^\complement(A,\rho) 
	& = & \V(A,\rho)\setminus L_{f,\alpha}(A),				\\
	\vspace{.2cm} \overline R_f(A)
	& = & \ds\left\{ (\alpha,q)\in\R_+^2 ~:~ \liminf_{\rho\rightarrow0^+}~ \frac{m(L_{f,\alpha}(A,\rho))}{\rho^q}>0 \right\}, 			\\
	\underline R_f(A)
	& = & \ds\left\{ (\alpha,q,q')\in\R_+^3 ~:~ \limsup_{\rho\rightarrow0^+}~ \left[ \frac{m(L_{f,\alpha}(A,\rho))}{\rho^q} + \frac{m(L_{f,\alpha}^\complement(A,\rho))}{\rho^{q'}}  \right] <\infty \right\}.
\end{array}\]
\end{spacing}
\vspace{-.2cm}

It is not too surprising that we take the irregularity of $f$ into account only through the measure $m$ since the integral with respect to $X$ does not differentiate between $m$-a.e. equal functions.

\begin{theo}	\label{theo:regularity_Y_poissonian}
	Let $A\in\A.$
	Suppose that $\sigma^2=0$ and $f$ is bounded in the victiny of $A$ (\ie $\sup_{s\in\V(A,\rho_0)}|f(s)|<\infty$ for some $\rho_0>0$).
	Then the following holds with probability one:
	\[ \sup_{(\alpha,q,q')\in\underline R_f(A)}\min\left\{ \frac{q}{\beta},\frac{q'}{\beta}+\alpha \right\} ~\ls~ \alpha_Y(A) ~\ls~ \inf_{(\alpha,q)\in\overline R_f(A)} \left\{\frac{q}{\beta}+\alpha\right\} \]
	with the conventions that $\inf\varnothing=1/0=+\infty.$
\end{theo}

We recall that $\A$ is supposed to be finite dimensional as well (Definition \ref{de:indexing_collection_finite_dimension}). We will comment and apply this result in Section \ref{subsection:examples}.

\bigskip
\paragraph{Proof of the upper bound}

The upper bound is very similar in spirit as \cite{jaffard_multifractal_1999} and the key step is the covering argument given by Proposition \ref{prop:covering_argument}.

Remark that if $\beta=0,$ there is nothing to prove. 
So let us suppose in the following that $\beta>0.$

\medskip

For now, let us consider any measurable set $L\subseteq\T$ of finite positive $m$-measure and denote for all $j\in\N,$
	
\vspace{-1.1cm}
\begin{spacing}{1.5}
\begin{align}
	\Pi_{j|L}	& \speq L\cap \Big\{ s\in\Pi : J_s(X)\in\Gamma_j \Big\}	\\
	\nu_j		& \speq 	\nu(\Gamma_j)		
\end{align}
\end{spacing}
\vspace{-.3cm}
	
where $\Gamma_j = \big\{ x\in\R : 2^{-j}\ls|x|<2^{-(j-1)} \big\}$ has been introduced in (\ref{eq:Gamma_j}).	
	
\begin{lem}	\label{lem:nu_j_estimates}
	Fix $\gamma<\beta.$ 
	There exists an increasing sequence $(j_k)_{k\in\N}$ in $\N$ such that
	\begin{align}
		\label{eq:nu_j_upper_estimate}
		\nu_j \speq \bigO\big( 8^j \big) \spas j\rightarrow\infty	\\
		\label{eq:nu_j_lower_estimate}
		2^{j_k\gamma} \speq \bigO\big(\nu_{j_k}\big) \spas k\rightarrow\infty
	\end{align}
\end{lem}

\begin{proof}
	Since $\beta>0$, the convergence of $\ds\int_{|x|\ls1}|x|^\gamma\nu(\text dx)$ is equivalent to the convergence of $\ds\sum_{j=0}^\infty\nu_j2^{-j\gamma}.$ 
	The Cauchy-Hadamard formula for the radius of convergence of power series then gives
	\begin{equation}	\label{eq:blumenthal-getoor_jaffard}	 
		\beta \speq \limsup_{j\rightarrow\infty} \frac{\log_2\nu_j}{j} 
	\end{equation}
	which is a relation that was already noted in \cite{jaffard_multifractal_1999}.
	The estimate (\ref{eq:nu_j_upper_estimate}) (resp. (\ref{eq:nu_j_lower_estimate})) then follows from this formula and the fact that $\beta<3$ (resp. $\gamma<\beta$).
\end{proof}
	
Recall that the random sets $E_{j|L}^\delta(h)$ and $E_{|L}^\delta(h)$ have been introduced in (\ref{eq:E^delta_j}) and (\ref{eq:E^delta}) respectively.

\begin{lem}	\label{lem:computation_borel_0-1}
	Let $A\in\A,$ then for all $\delta>0$ and $j\in\N,$
	\[ \exp\left[ -\nu_jm\big( L\cap\V(A,2^{-j\delta}) \big) \right] ~\ls~ \prob{A\notin E_{j|L}^\delta(X)} ~\ls~ \exp\left[ -\nu_jm\big( L\cap\V(A,2^{-(j+1)\delta}) \big) \right].  \]
\end{lem} 
	
\begin{proof}
	We only prove the upper bound, the lower bound being proven in exactly the same way.
	
	Fix $\delta>0$ and $j\in\N.$ 
	Then, by definition of $\V$ and $\V',$
	
	\vspace{-.7cm}
	\begin{spacing}{2}
	\[\begin{array}{rcl}
		\prob{A\notin E_{j|L}^\delta(X)}
		& \speq	& \prob{ \forall s\in\Pi_{j|L},~ A\notin\V'(s,|J_s(X)|^\delta) }	\\
		& ~\ls~	&  \prob{ \forall s\in\Pi_{j|L},~ A\notin\V'(s,2^{-(j+1)\delta}) }	\\
		& = 	& \prob{ \forall s\in\Pi_{j|L},~ s\notin\V(A,2^{-(j+1)\delta}) }.
	\end{array}\]
	\end{spacing}
	\vspace{-.5cm}	
	
	Since $\Pi_{j|L}$ is a Poisson random set of intensity measure $\nu_jm(L\cap.)$, we may write $\Pi_{j|L}=\{ s_1,...,s_{M_j} \}$ where $M_j$ is a Poisson random variable of intensity $\nu_jm(L)$ and the $s_i$'s are iid variables of distribution $m(L\cap.)/m(L)$ independent from $M_j.$
	Hence, conditioning with respect to $M_j$ yields
	
	\vspace{-.7cm}
	\begin{spacing}{2}
	\[\begin{array}{rcl}
		\prob{ \forall s\in\Pi_{j|L},~ s\notin\V(A,2^{-(j+1)\delta}) }
		& \speq	& \esp{ \prob{ \forall i\in\llbracket1,M_j\rrbracket,~s_i\notin\V(A,2^{-(j+1)\delta}) ~|~ M_j} } \\
		& = 	& \esp{ \prob{ s_1\notin\V(A,2^{-(j+1)\delta}) }^{M_j} } 			\\
		& = 	& \exp\left[ \nu_jm(L)\lp \prob{s_1\notin\V(A,2^{-(j+1)\delta})}-1 \rp \right]	\\
		& = 	& \exp\left[ -\nu_j m\big( L\cap\V(A,2^{-(j+1)\delta}) \big) \right].
	\end{array}\]
	\end{spacing}
	\vspace{-.7cm}
	
	The result follows.
\end{proof}

\begin{prop}	\label{prop:covering_argument}
	Let $A\in\A$ and suppose that there exists $q>0$ such that
	\begin{equation}	\label{eq:lower_estimate_L}
		\rho^q ~=~ \bigO\big( m(L\cap\V(A,\rho)) \big) \quad\text{as}\quad \rho\rightarrow0.
	\end{equation}	
	Then for all $\delta<\beta/q,$ $A$ belongs to $E^\delta_{|L}(X)$ with probability one.
\end{prop}

\begin{proof}
	Fix $\delta<\beta/q$ and $A\in\A.$
	Then, by (\ref{eq:lower_estimate_L}) and Lemma \ref{lem:computation_borel_0-1}, there exists $\kappa>0$ such that
	\[ \forall j\in\N,\quad \prob{A\notin E_{j|L}^\delta(X)} ~\ls~ \exp\lp -\kappa\nu_j2^{-j\delta q} \rp.  \]
	Hence, taking $\gamma\in(\delta q,\beta)$ and using the sequence $(j_k)_{k\in\N}$ of Lemma \ref{lem:nu_j_estimates}, we get
	\[ \forall k\in\N,\quad \prob{A\notin E_{j_k|L}^\delta(X)} ~\ls~ \exp\lp -\kappa2^{j_k(\gamma-\delta q)} \rp  \]
	which is a convergent series. 
	The result follows from Borel-Cantelli lemma. 
\end{proof}

\begin{prop}	\label{prop:upper_bound_Y_poissonian}
	For all $(\alpha,q)\in\overline R_f(A),$ we have
	\[ \alpha_Y(A) ~\ls~ \frac{q}{\beta}+\alpha \quad\text{a.s.}  \]
\end{prop}

\begin{proof}
	The following holds with probabiliy one.
	
	\vspace{-.7cm}
	\begin{spacing}{2}
	\[\begin{array}{rcll}
		\alpha_Y(A)
		& ~\ls~	& \ds\liminf_{\substack{s\in\Pi(Y): \\ \qd(s,A)\rightarrow0}} \frac{\log|J_s(Y)|}{\log\qd(s,A)} 	
		& \text{by Theorem \ref{theo:jaffard_lemma},}		\\
		& \speq	&  \ds\liminf_{\substack{s\in\Pi(Y): \\ \qd(s,A)\rightarrow0}} \left[ \frac{\log|J_s(X)|}{\log\qd(s,A)} + \frac{\log|f(s)|}{\log\qd(s,A)} \right]
		& \text{by Corollary \ref{cor:jump_structure_Y},}	\\
		& \ls 	& \underset{\ls1/\delta}{\underbrace{\ds\liminf_{\substack{s\in\Pi(Y)\cap L_{f,\alpha}(A): \\ \qd(s,A)\rightarrow0}} \frac{\log|J_s(X)|}{\log\qd(s,A)}}} ~+~ \underset{\ls\alpha}{\underbrace{\limsup_{\substack{s\in L_{f,\alpha}(A): \\ \qd(s,A)\rightarrow0}} \frac{\log|f(s)|}{\log\qd(s,A)}}}				
	\end{array}\]
	\end{spacing}
	\vspace{-.5cm}
	
	where the last two inequalities are due to Propositions \ref{prop:jaffard_bound} and \ref{prop:covering_argument} and the definition of $L_{f,\alpha}(A)$.
\end{proof}

The upper bound of Theorem \ref{theo:regularity_Y_poissonian} may be readily deduced from Proposition \ref{prop:upper_bound_Y_poissonian} by taking a relevent sequence converging to the claimed upper bound.

\medskip

The covering argument could be modified to give an upper bound which would hold for all $A\in\A$ with probability one. It would basically only require suitable assumptions in order to be able to apply \cite[Theorem 2]{hoffmann-jorgensen_covering_1973}.
However, since such a modification would make the assumptions much heavier, we chose against it. Moreover, this improvement cannot be carried over to the lower bound argument due to the multifractal nature of the regularity (see \cite{jaffard_multifractal_1999}).

\bigskip
\paragraph{Proof of the lower bound when $\beta\gs1$}

For this part of the proof, we will require some concentration in the form of a set-indexed version of Cairoli's inequality (Theorem \ref{theo:cairoli_inequality}), which is a multiparameter generalization of Doob's maximal inequality.

The key step in proving it is to discretize the set-indexed martingale $Y$ on $\A_n,$ then use the fact that $\A_n$ embeds itself into $\N^p$ for some integer $p$ and finally apply already known results about multiparameter martingales. 
This is what Lemma \ref{lem:cairoli_inequality} contains.

\begin{lem}[Discrete set-indexed maximal inequality]	\label{lem:cairoli_inequality}
	For all $\gamma>1,$ there exists a constant $\kappa_{p,\gamma}>0$ such that for all $n\in\N$ and $\rho>0,$	
	\[ \esp{ \sup_{A'\in \A_n\cap B_\A(A,\rho)} \left| Y_A-Y_{A'} \right|^\gamma } ~\ls~ \kappa_{p,\gamma} \lp \esp{ \left| \Delta Y_{\V_n(A,\rho)} \right|^\gamma } \+ \esp{ \left| \Delta Y_{A\setminus\underline V_n(A,\rho)} \right|^\gamma } \rp \]
	where $\V_n(A,\rho)=\overline V_n(A,\rho)\setminus\underline V_n(A,\rho)$ has been introduced in (\ref{eq:V_n}).
\end{lem}

\begin{rems} \ 
	\begin{enumerate}[$\diamond$]
		\item Since $\V_n(A,\rho)$ (resp. $A\setminus\underline V_n(A,\rho)$) belongs to $\C(u)$ (resp. $\C_0$), Proposition \ref{propde:increment_map_linear_extension} ensures that $\Delta Y_{\V_n(A,\rho)}$ (resp. $\Delta Y_{A\setminus\underline V_n(A,\rho)}$) is a well-defined expression.
		Such a point also holds for the expression $\Delta Y_{A\setminus\underline V(A,\rho)}$ appearing later in Theorem \ref{theo:cairoli_inequality}.
 
		\item As it will made clearer in the proof, such a result holds for more general processes than $Y$ and are closely linked to a class of $\N^p$-indexed processes called \emph{orthosubmartingales} as defined in \cite[p.16-17]{khoshnevisan_multiparameter_2002}. For more information about set-indexed martingales, we refer to \cite{ivanoff_set-indexed_1999}.
			
		\item Generalizing \cite[Theorem 3.3]{walsh_convergence_1979} might lead to a weak $L^1$ version of such inequality, but we do not need it here.
	\end{enumerate}
\end{rems}

\begin{proof}
	Fix $\gamma>1$, $n\in\N$ and $\rho>0$.
	Since $\A_n$ is of poset dimension smaller than $p$ (see Section \ref{subsection:indexing_collection_finite_dim}), there exists an order embedding $\phi:\A_n\cap B_\A(A,\rho)\hookrightarrow\llbracket0,k\rrbracket$ for some $k\in\N^p$ where $\N^p$ is endowed with the usual componentwise partial order denoted here by $\pre.$ 
	Denote by $I$ the range of $\phi$ and define
	\[ \forall j\in\llbracket0,k\rrbracket,\quad M_j ~=~ \left| \Delta Y_{\bigcup_{i\in I:i\pre j}\phi^{-1}(i) \setminus \underline V_n(A,\rho)} \right|. \]
	Checking that $M=\{ M_j:j\in\llbracket0,k\rrbracket \}$ is a non-negative orthosubmartingale in the sense of \cite[p.16-17]{khoshnevisan_multiparameter_2002} is straightforward. Using \cite[Theorem 2.3.1]{khoshnevisan_multiparameter_2002} leads to
	\begin{equation}	\label{eq:cairoli_1}
		\esp{\sup_{j\pre k} M_j^\gamma} ~\ls~ \lp\frac{\gamma}{\gamma-1}\rp^{\gamma p} \esp{M_k^\gamma}.
	\end{equation}
	By definition of $M$, we have
	\begin{equation}	\label{eq:cairoli_2}
		M_k \speq \big|\Delta Y_{\overline V_n(A,\rho) \setminus \underline V_n(A,\rho)}\big| \speq \big|\Delta Y_{\V_n(A,\rho)}\big|.
	\end{equation}
	Moreover, 
	
	\vspace{-.8cm}
	\begin{spacing}{1.5}
	\[\begin{array}{rcl}
		\ds\sup_{A'\in\A_n\cap B_\A(A,\rho)} |Y_A-Y_{A'}|
		& ~\ls~	& |Y_A-Y_{\underline V_n(A,\rho)}| \+ \ds\sup_{A'\in\A_n\cap B_\A(A,\rho)} |Y_{A'}-Y_{\underline V_n(A,\rho)}| 	\\
		& \speq	& |\Delta Y_{A\setminus\underline V_n(A,\rho)}| \+ \ds\sup_{j\pre k} M_j.
	\end{array}\]
	\end{spacing}
	\vspace{-.5cm}
	
	Hence
	\begin{equation}	\label{eq:cairoli_3}
		\esp{ \sup_{A'\in \A_n\cap B_\A(A,\rho)} \left| Y_A-Y_{A'} \right|^\gamma } ~\ls~ 2^\gamma \lp \esp{ |\Delta Y_{A\setminus\underline V_n(A,\rho)}|^\gamma } \+ \esp{ \sup_{j\pre k} M_j^\gamma } \rp.
	\end{equation}
	The result follows from combining (\ref{eq:cairoli_1}), (\ref{eq:cairoli_2}) and (\ref{eq:cairoli_3}).
	
\end{proof}

In the following theorem, since $\V(A,\rho)$ does not necessarily belong to $\C(u),$ the expression $\Delta Y_{\V(A,\rho)}$ in what follows should be understood in the same sense as (\ref{eq:Delta_Y_on_borelian}), \ie
\[ \Delta Y_{\V(A,\rho)} ~=~ \int_\T f\myindic{\V(A,\rho)}\,\text dX. \]		

\begin{theo}[Set-indexed maximal inequality]	\label{theo:cairoli_inequality}
	For all $\gamma>1,$ there exists a constant $\kappa_{p,\gamma}>0$ such that for all $\rho>0,$	
	\[ \esp{ \sup_{A'\in B_\A(A,\rho)} \left| Y_A-Y_{A'} \right|^\gamma } ~\ls~ \kappa_{p,\gamma} \left[ \esp{ \left| \Delta Y_{\V(A,\rho)} \right|^\gamma } + \esp{ \left| \Delta Y_{A\setminus\underline V(A,\rho)} \right|^\gamma } \right] \]
	where $\underline V(A,\rho) = \bigcap_{n\in\N} \underline V_n(A,\rho).$
\end{theo}

\begin{proof}
	Just like the one-dimensional setting, obtaining continuous versions of martingale inequalities is a consequence of their discrete counterparts (here, Lemma \ref{lem:cairoli_inequality}) and outer-continuity of sample paths (ensured for $Y$ by its Lévy-Itô decomposition (\ref{eq:levy-ito_decomposition})).
\end{proof}

This maximal inequality encourages us to understand the $L^\gamma$-norm of variables like $\Delta Y_U = \mathbf X(f\myindic{U})$.
In this endeavor, we will need to dabble a bit with Orlicz spaces to obtain relevant estimates.

For $\gamma\in(1,2],$ let $\Phi_\gamma$ be the Orlicz function defined by
	\[\begin{array}{cccc}
	\Phi_\gamma : 	
	& \R 	& \longrightarrow 	& \R_+ 		\\
	& u 	& \longmapsto 		& \ds\int_\R \lp |xu|^2\wedge|xu|^\gamma \rp\nu(\text dx).
	\end{array}\]

Denote $\|.\|_{\Phi_\gamma}$ the associated \emph{Luxemburg norm} defined for all measurable map $g:\T\rightarrow\R$ by
\begin{equation} \label{eq:luxemburg_norm}
	\| g \|_{\Phi_\gamma} = \inf\left\{ c>0 ~:~ \int_\T \Phi_\gamma(c^{-1}g(s))m(\text ds)\ls1 \right\}. 
\end{equation}

We refer to \cite{harjulehto_orlicz_2019} for a modern exposition of the general theory of (generalized) Orlicz spaces. 
However, we will only be using the fact that $\|.\|_{\Phi_\gamma}$ induces a (quasi)norm on the linear space $L^{\Phi_\gamma}$ of all measurable maps $g:\T\rightarrow\R$ (where two $m$-a.e. maps are identified) such that $\|g\|_{\Phi_\gamma}<\infty.$

\medskip

More specifically of interest to us is that \cite[Theorem 3.3]{rajput_spectral_1989} proves that the stochastic integral
\[\begin{array}{ccc}
	L^{\Phi_\gamma}	& \longrightarrow 	& L^\gamma(\Omega) 	\\
	g 				& \longmapsto 		& \ds\int_\T g\,\text dX
\end{array}\]
is continuous.
In particular, there exists a constant $\kappa_{\Phi_\gamma}>0$ such that:
\begin{equation}	\label{eq:orlicz_control}
	\forall g\in L(X),\quad  \|\mathbf X(g)\|_{L^\gamma(\Omega)} ~=~ \esp{\left| \int_{\T} g\,\text dX \right|^\gamma}^{1/\gamma} ~\ls~ \kappa_{\Phi_\gamma}\|g\|_{\Phi_\gamma}.
\end{equation}

The following lemma will simplify further use of the Luxemburg norm.

\begin{lem}	\label{lem:orlicz_lebesgue_control}
	For any $\gamma\in(1,2]$ and measurable map $g:\T\rightarrow\R,$
	\[ \|g\|_{\Phi_\gamma} ~\ls~ \left[ \int_\R|x|^\gamma\nu(\text dx) \right]^{1/\gamma}\|g\|_{L^\gamma(m)}. \]
\end{lem}

\begin{proof}
	In the formulation of the $L^\gamma$-space as an Orlicz space, we have
	\[ \| g \|_{L^\gamma(m)} ~=~ \inf\left\{ c>0 ~:~ \int_\T c^{-\gamma}|g(s)|^\gamma m(\text ds)\ls1 \right\}. \]
	Comparing this norm with (\ref{eq:luxemburg_norm}), it follows that it is enough to prove the following:
	\begin{equation} 	\label{eq:CNS_norms} 
		\forall c>0,\quad \int_\T\Phi(c^{-1}g(s))m(\text ds) ~\ls~ c^{-\gamma} \int_\T |g(s)|^\gamma m(\text ds) \int_\R|x|^\gamma\nu(\text dx). 
	\end{equation}
	
	Let $c>0$ and $s\in\T.$ Then,
	
	\vspace{-1.1cm}
	\begin{spacing}{2}
	\[\begin{array}{rcl}
		\Phi(c^{-1}g(s)) 
		& \speq	& \ds\int_{|xg(s)|\ls c} |c^{-1}xg(s)|^2 \nu(\text dx) \+ \int_{|xg(s)|>c} |c^{-1}xg(s)|^\gamma\nu(\text dx) 	\\
		& = 	& \ds c^{-2} \int_{|xg(s)|\ls c} |xg(s)|^{2-\gamma}|xg(s)|^\gamma \nu(\text dx) \+ c^{-\gamma} \int_{|xg(s)|<c}|xg(s)|^\gamma \nu(\text dx)						\\
		& ~\ls~	& \ds c^{-2}c^{2-\gamma} \int_{|xg(s)|\ls c} |xg(s)|^\gamma\nu(\text dx) \+ c^{-\gamma} \int_{|xg(s)|> c} |xg(s)|^\gamma\nu(\text dx) 		\\
		& = 	& \ds c^{-\gamma}|g(s)|^\gamma \int_\R|x|^\gamma\nu(\text dx).
	\end{array}\]
	\end{spacing}
	\vspace{-0.5cm}
	
	Integrating with respect to $m$ yields (\ref{eq:CNS_norms}), from which the result follows.
\end{proof}

Recall that we supposed $\beta\gs1.$
From (\ref{eq:orlicz_control}) and Lemma \ref{lem:orlicz_lebesgue_control} follows that for all $\gamma\in(\beta,2]$ (or $\gamma=2$ if $\beta=2$), there exists a finite constant $\kappa_\gamma>0$ such that
\begin{equation}	\label{eq:control_Delta_Y}
	\forall B\in\B_m,\quad \esp{|\Delta Y_B|^\gamma} ~\ls~ \kappa_\gamma\|f\myindic{B}\|_{L^\gamma(m)}^\gamma
\end{equation}
where $\Delta Y_B = \mathbf X(f\myindic{B}).$

\medskip

We are now ready to proceed to the lower bound itself.

\begin{prop}	\label{prop:regularity_Y_beta>1}
	Suppose that the hypotheses of Theorem \ref{theo:regularity_Y_poissonian} and $\beta\gs1$ hold.  
	
	Then, for all $(\alpha,q,q')\in\underline R_f(A),$
	\[ \alpha_Y(A) ~\gs~ \min\left\{ \frac{q}{\beta},\frac{q'}{\beta}+\alpha \right\} \quad\text{a.s.} \]
\end{prop}

\begin{proof}
	Fix $(\alpha,q,q')\in\underline R_f(A).$ 
	We will only prove the result in the case $q'/\beta+\alpha\ls q/\beta$. 
	The second case is proven in exactly the same fashion, one just has to replace $q'$ by $q$ and take $\alpha=0$ in the following. 
	
	Let $\delta>\beta/q'$ and $\eta=1/\delta+\alpha$.
	By Borel-Cantelli, it is enough to prove that
	\begin{equation}	\label{eq:CS_lower_bound}
		\sum_{j=1}^\infty\, \prob{ \sup_{d_\A(A,A')<2^{-j}} |Y_A-Y_{A'}| > 2^{-j\eta} } ~<~ \infty.
	\end{equation}
	Fix $\gamma\in(\beta,2]$ (or $\gamma=2$ if $\beta=2$).
	By Theorem \ref{theo:cairoli_inequality}, for all $j\in\N,$
	\begin{equation}	\label{eq:cairoli_applied}
		\prob{ \sup_{d_\A(A,A')<2^{-j}} |Y_A-Y_{A'}| > 2^{-j\eta} } ~\ls~ \kappa_{p,\gamma} 2^{j\eta\gamma} \left[ \esp{ \left| \Delta Y_{\V(A,2^{-j})} \right|^\gamma } + \esp{ \left| \Delta Y_{A\setminus\underline V(A,2^{-j})} \right|^\gamma } \right].
	\end{equation}

	
	Since $A\setminus\underline V(A,2^{-j})\subseteq\V(A,2^{-j}),$ it follows from (\ref{eq:control_Delta_Y}) and (\ref{eq:cairoli_applied}) that
	
	\begin{equation}	\label{eq:control_by_f_gamma}
		\prob{ \sup_{d_\A(A,A')<2^{-j}} |Y_A-Y_{A'}| > 2^{-j\eta} } \speq \bigO\lp2^{j\eta\gamma}\|f\myindic{\V(A,2^{-j})}\|_{L^\gamma(m)}^\gamma\rp \spas j\rightarrow\infty.
	\end{equation}
	
	On the one hand, since $(\alpha,q,q')\in\underline R_f(A),$ we have $m(L_{f,\alpha}(A,\rho))=\bigO(\rho^q)$ as $\rho\rightarrow0$. Moreover, $f$ is also bounded in the victiny of $A$.
	Thus
	\begin{equation}	\label{eq:estimate_f_0}
		\|f\myindic{L_{f,\alpha}(A,\rho)}\|_{L^\gamma(m)}^\gamma \speq \bigO\lp\rho^{q}\rp \spas \rho\rightarrow0.
	\end{equation}
	
	On the other hand, using $(\alpha,q,q')\in\underline R_f(A)$ once more yields $m(L_{f,\alpha}^\complement(A,\rho))=\bigO(\rho^{q'})$ as $\rho\rightarrow0$.
	Hence
	\begin{equation}	\label{eq:estimate_f_1}
		\|f\myindic{L_{f,\alpha}^\complement(A,\rho)}\|_{L^\gamma(m)}^\gamma \speq \bigO\lp\rho^{q'+\alpha\gamma}\rp \spas \rho\rightarrow0.
	\end{equation}
	Since we supposed that $q'/\beta+\alpha\ls q/\beta$ and if we take $\gamma$ close enough to $\beta$, the estimates (\ref{eq:estimate_f_0}) and (\ref{eq:estimate_f_1}) give together
	\begin{equation}	\label{eq:estimate_f_0+1}
		\|f\myindic{\V(A,\rho)}\|_{L^\gamma(m)}^\gamma \speq \bigO\lp\rho^{q'+\alpha\gamma}\rp \spas \rho\rightarrow0.
	\end{equation}
	Combining (\ref{eq:control_by_f_gamma}) and (\ref{eq:estimate_f_0+1}) yields
	\begin{equation}	\label{eq:CNS_ls1_borel-cantelli_estimate}
		\prob{ \sup_{d_\A(A,A')<2^{-j}} |Y_A-Y_{A'}| > 2^{-j\eta} } \speq \bigO\lp2^{-j(q'+(\alpha-\eta)\gamma)}\rp \spas j\rightarrow\infty.
	\end{equation}
	Since 
	\[ q'+(\alpha-\eta)\gamma \speq q'-\frac{\gamma}{\delta} ~\longrightarrow~ q'-\frac{\beta}{\delta}  \spas \gamma\rightarrow\beta^+ \]
	and $q'-\beta/\delta>0,$ (\ref{eq:CNS_ls1_borel-cantelli_estimate}) proves that (\ref{eq:CS_lower_bound}) holds.
	The result follows.
\end{proof}

\bigskip
\paragraph{Proof of the lower bound when $\beta<1$}

For this part of the proof, we suppose that $\beta<1$.
In particular, $\ds\int_{|x|\ls1}|x|\nu(\text dx)<\infty$ so that (\ref{eq:Y_simplified}) gives the following expression:
\begin{equation}	\label{eq:Y_integrable_case}
	\forall A\in\A,\quad Y_A \speq \lim_{\eps\rightarrow0^+} \bigg[ \sum_{\substack{t\in\Pi\cap A: \\ |J_t(X)|\gs\eps}} f(t)J_t(X) \bigg]
\end{equation}
where the limit as $\eps\rightarrow0^+$ happens almost surely uniformly in $[\varnothing,A]$ for all $A\in\A.$

The problems in trying to copy the proof for $\beta\gs1$ is that $Y$ is not a (orthosub)martingale anymore (as mentioned in the proof of Lemma \ref{lem:cairoli_inequality}). 
Instead, we introduce the set-indexed process $Z_{|L}$ given by
\begin{equation}	\label{eq:Z_L}
	\forall A\in\A,\quad (Z_{|L})_A \speq \lim_{\eps\rightarrow0^+} \bigg[ \sum_{\substack{t\in\Pi\cap L\cap A: \\ |J_t(X)|\gs\eps}} |J_t(X)| \bigg]
\end{equation}
where $L\subseteq\T$ is once more a free parameter set that will be chosen later on. Remark that $Z_{|L}$ is also a set-indexed Lévy process, but with respect to $m(L\cap.)$ instead of $m$ as `reference measure'.

We first establish a lower bound on the regularity of $Z_{|L}$ and then deduce one for the regularity of $Y$.

\medskip

The following method is inspired from \cite{balanca_fine_2014}. 
Denote for all $\eta\gs0,$ the set-indexed process $(Z_{|L})^\eta$ given by
\begin{equation}	\label{eq:Z^eta}
	\forall A\in\A,\quad (Z_{|L})_A^\eta ~=~ \lim_{\eps\rightarrow0^+} \bigg[ \sum_{\substack{t\in\Pi\cap L\cap A: \\ \eps\ls|J_t(X)|< 2^{-\eta}}} |J_t(X)| \bigg]
\end{equation}

\begin{lem}	\label{lem:balanca_lemma}
	Let $\A\in\A$ and suppose that there exists $q>0$ such that
	\begin{equation}	\label{eq:estimate_L}
		m(L\cap\V(A,\rho)) ~=~ \bigO\big( \rho^q \big) \qquad\text{as}\qquad \rho\rightarrow0.
	\end{equation}
	Then for all $\delta>\beta/q,$ there exists a constant $\kappa_\delta>0$ such that
	\[ \forall j\in\N,\quad \prob{ \sup_{A'\in B_\A(A,2^{-j})} \Delta (Z_{|L})_{A\triangle A'}^{j/\delta} \gs j2^{-j/\delta} } ~\ls~ \kappa_\delta\, e^{-j}. \]
\end{lem}

\begin{proof}
	The proof is inspired from \cite[Lemma 2.1]{balanca_fine_2014}, but we chose to give the details for the sake of completeness.
	Fix $j\in\N.$ Then, since $Z$ is non-negative, 
	\[ \forall A'\in B_\A(A,2^{-j}),\quad \Delta (Z_{|L})_{A\triangle A'}^{j/\delta} ~\ls~ \Delta (Z_{|L})_{\V(A,2^{-j})}^{j/\delta}. \]
	
	Hence, by Markov's inequality,
	\[ \prob{ \sup_{A'\in B_\A(A,2^{-j})} \Delta (Z_{|L})_{A\triangle A'}^{j/\delta} \gs j2^{-j/\delta} } ~\ls~ e^{-j}\esp{ \exp\lp 2^{j/\delta}\Delta (Z_{|L})_{\V(A,2^{-j})}^{j/\delta} \rp }. \]
	
	So we just need to prove that $\esp{ \exp\lp 2^{j/\delta}\Delta (Z_{|L})_{\V(A,2^{-j})}^{j/\delta} \rp }=\bigO(1)$ as $j\rightarrow\infty$ to conclude.
	
	Using (\ref{eq:Z^eta}), we may compute the Laplace transform of $\Delta (Z_{|L})_{\V(A,2^{-j})}^{j/\delta}$ and get
	
	\[ \esp{ \exp\lp 2^{j/\delta}\Delta (Z_{|L})_{\V(A,2^{-j})}^{j/\delta} \rp } \speq \exp\lp 2^{j/\delta} m(L\cap \V(A,2^{-j})) \int_{|x|\ls2^{-j/\delta}} \lp e^{|x|}-1 \rp\nu(\text dx) \rp.  \]	

	
	Hence, due to (\ref{eq:estimate_L}) and the fact that $e^{|x|}-1\ls2|x|$ for $|x|\ls1,$ there exists a constant $\kappa>0$ such that
	\[ \esp{ \exp\lp 2^{j/\delta}\Delta (Z_{|L})_{\V(A,2^{-j})}^{j/\delta} \rp } ~\ls~ \exp\lp \kappa 2^{-j(q-1/\delta)}\int_{|x|\ls2^{-j/\delta}} |x|\nu(\text dx) \rp. \]
	Taking $\gamma\in(\beta,1\wedge \delta q)$ gives	
	
	\vspace{-1cm}
	\begin{spacing}{2.5}
	\[\begin{array}{rcl}
		\esp{ \exp\lp 2^{j/\delta}\Delta (Z_{|L})_{\V(A,2^{-j})}^{j/\delta} \rp }
		& ~\ls~	& \ds\exp\lp \kappa 2^{-j(q-1/\delta)}\int_{|x|\ls2^{-j/\delta}} |x|^{1-\gamma}|x|^\gamma\nu(\text dx) \rp			\\ 
		& \ls	& \ds\exp\lp \kappa 2^{-j(q-1/\delta+(1-\gamma)/\delta)}\int_{|x|\ls2^{-j/\delta}} |x|^\gamma\nu(\text dx) \rp	\\
		& \ls	& \ds\exp\lp \kappa 2^{-j(q-\gamma/\delta)}\int_{|x|\ls1} |x|^\gamma\nu(\text dx) \rp.
	\end{array}\]
	\end{spacing}
	\vspace{-.7cm}
	
	Since $q-\gamma/\delta\rightarrow q-\beta/\delta$ as $\gamma\rightarrow\beta^+$ and $q-\beta/\delta>0$, we may find $\gamma>\beta$ showing that the above expression is bounded as $j\rightarrow\infty$.
	The result follows.
\end{proof}

\begin{lem}	\label{lem:regularity_Z_L}
	Let $A\in\A$ and suppose that (\ref{eq:estimate_L}) holds.
	Then, the following holds with probability one: for all $\delta>\beta/q$, there exists $\rho_\delta>0$ such that
	\[ \forall\rho\in(0,\rho_\delta),\,\forall A'\in B_\A(A,\rho),\quad (\Delta Z_{|L})_{A\triangle A'} ~\ls~ \rho^{1/\delta}. \]
\end{lem}

\begin{proof}
	Let us fix $\delta>\beta/q.$
	It is enough to prove that the following holds with probability one: there exists $k\in\N$ such that 
	\[ \forall j\gs k,\,\forall A'\in B_\A(A,2^{-j}),\quad (\Delta Z_{|L})_{A\triangle A'} ~\ls~ 2^{-j/\delta}. \]
	According to Lemma \ref{lem:balanca_lemma}, this is already true if we replace $(\Delta Z_{|L})$ by $(\Delta Z_{|L})^{j/\delta}$.
	Thus, it is enough to prove that the following holds with probability one: there exists $k\in\N$ such that for all $j\gs k$ and $A'\in B_\A(A,2^{-j}),\,$ $(\Delta Z_{|L})_{A\triangle A'} = (Z_{|L})_{A\triangle A'}^{j/\delta}.$
	In other words, we want to show that $A\notin E^\delta_{|L}(X)$ almost surely.
	By Lemma \ref{lem:computation_borel_0-1}, we have for all $j\in\N,$
	
	\vspace{-.5cm}
	\begin{spacing}{1.5}
	\[\begin{array}{rcll}
		\prob{A\in E_{j|L}^\delta(X)}
		& ~\ls~	& 1-\exp\left[ -\nu_jm\big( L\cap\V(A,2^{-j\delta}) \big) \right] 	\\
		& \ls	& \nu_jm\big( L\cap\V(A,2^{-j\delta}) \big)
		& ~\text{ since } 1-e^{-x}\ls x, 	\\
		& \speq	& \bigO\big( \nu_j2^{-j\delta q} \big)
		& ~\text{ by (\ref{eq:estimate_L}).}
	\end{array}\]
	\end{spacing}
	\vspace{-.2cm}
	
	Taking $\gamma\in(\beta,\delta q)$ and using (\ref{eq:blumenthal-getoor_jaffard}) yields
	\[ \prob{A\in E_{j|L}^\delta(X)} \speq \bigO\big( 2^{-j(\delta q-\gamma)} \big) \spas j\rightarrow\infty \]
	which then is a convergent series. 
	Thus, by Borel-Cantelli, $A\notin E_{|L}^\delta(X)$ with probability one.
	The result follows.
\end{proof}

\begin{prop}	\label{prop:lower_bound_Y_poissonian}
	Suppose that the hypotheses of Theorem \ref{theo:regularity_Y_poissonian} and $\beta<1$ hold. 
	
	Then, for all $(\alpha,q,q')\in\underline R_f(A),$
	\[ \alpha_Y(A) ~\gs~ \min\left\{ \frac{q}{\beta},\frac{q'}{\beta}+\alpha \right\} \quad\text{a.s.} \]
\end{prop}

\begin{proof}
	Fix $(\alpha,q,q')\in\underline R_f(A).$ 
	By definition of $L_{f,\alpha}^\complement(A,\rho),$ we have
	\begin{equation}	\label{eq:estimate_f_0}
		\forall s\in L_{f,\alpha}^\complement(A,\rho),\quad |f(s)|~\ls~\rho^{\alpha}
	\end{equation}
	for all $\rho>0$ small enough.
	
	Likewise, since $f$ is bounded in the victiny of $A$, there exists $\kappa_f>0$ such that
	\begin{equation}	\label{eq:estimate_f_1}
		\forall s\in L_{f,\alpha}(A,\rho),\quad |f(s)|~\ls~\kappa_f
	\end{equation}
	for all $\rho>0$ small enough.
	
	Combining those estimates on $f$ with the expression (\ref{eq:Y_integrable_case}) of $Y$ yields for all $\rho>0$ small enough
	\[ \forall A'\in B_\A(A,\rho),\quad |Y_A-Y_{A'}| ~\ls~ \rho^\alpha (\Delta Z_{|L_{f,\alpha}(A)^\complement})_{A\triangle A'} \+ \kappa_f (\Delta Z_{|L_{f,\alpha}(A)})_{A\triangle A'}.  \]
	
	Since $(\alpha,q,q')\in\underline R_f(A),$ we may apply Lemma \ref{lem:regularity_Z_L} to both $L=L_{f,\alpha}$ and $L=L_{f,\alpha}^\complement$ (for which (\ref{eq:estimate_L}) holds if we replace $q$ by $q'$). 
	Thus, the following holds with probability one: for all $\eps>0,$ there exists $\rho_0>0$ such that
	\[ \forall\rho<\rho_0,\,\forall A'\in B_\A(A,\rho),\quad |Y_A-Y_{A'}| ~\ls~ \rho^{\alpha+q'/\beta-\eps} \+ \kappa_f\,\rho^{q/\beta-\eps}.  \]
	The result follows immediately.	
\end{proof}

Just as for the upper bound and Proposition \ref{prop:upper_bound_Y_poissonian}, the lower bound of Theorem \ref{theo:regularity_Y_poissonian} is deduced from Proposition \ref{prop:lower_bound_Y_poissonian} by taking a relevent subsequence.

\subsection{$d_\T$-localized regularity}	\label{subsection:localized_regularity}

\subsubsection{The Gaussian part}

Using the same method as in \cite{herbin_local_2016} and under the same entropic conditions, one is able to determine the regularity of $Y$ in the case where $\nu=0$. Namely, for all $A\in\A,$ the following holds with probability one:

\begin{equation}	\label{eq:Y_gaussian_loc_regularity}
	\alpha_{Y,d_\T}(A) \speq \frac{1}{2}\alpha_{\sigma\int_.f^2\,\text dm,d_\T}(A).
\end{equation}

\subsubsection{The Poissonian part}

Similarly to Section \ref{subsection:ptw_regularity}, we define for all $A=A(t)\in\A,$ $\alpha\gs0$ and $\rho>0,$

\vspace{-.5cm}
\begin{spacing}{1.8}
\[\begin{array}{ccc}
	L_{f,\alpha}'(A) 
	& \speq & \Big\{ s\in\T : |f(s)|>d_\T(s,t)^\alpha \Big\},	\\
	L_{f,\alpha}'(A,\rho) 
	& = 	& B_\T(t,\rho)\cap L_{f,\alpha}'(A),				\\
	\vspace{.2cm} L_{f,\alpha}'^\complement(A,\rho) 
	& = 	& B_\T(t,\rho)\setminus L_{f,\alpha}'(A),			\\
	\vspace{.2cm} \overline R_f'(A)
	& = 	& \ds\left\{ (\alpha,q)\in\R_+^2 ~:~ \liminf_{\rho\rightarrow0^+}~ \frac{m(L_{f,\alpha}'(A,\rho))}{\rho^q}>0 \right\}, \\
	\underline R_f'(A)
	& = 	& \ds\left\{ (\alpha,q,q')\in\R_+^3 ~:~ \limsup_{\rho\rightarrow0^+}~ \left[ \frac{m(L_{f,\alpha}'(A,\rho))}{\rho^q} + \frac{m(L_{f,\alpha}'^\complement(A,\rho))}{\rho^{q'}} \right] <\infty \right\}.
\end{array}\]
\end{spacing}
\vspace{-.5cm}

\begin{theo}	\label{theo:loc_regularity_Y_poissonian}
	Let $A\in\A.$
	Suppose that $\sigma^2=0$ and that $f$ is bounded in the neighborhood of $A.$
	Then, with probability one,
	\[ \sup_{(\alpha,q,q')\in\underline R_f'(A)}\min\left\{ \frac{q}{\beta},\frac{q'}{\beta}+\alpha \right\} ~\ls~ \alpha_{Y,d_\T}(A) ~\ls~ \inf_{(\alpha,q)\in\overline R_f'(A)} \left\{\frac{q}{\beta}+\alpha\right\} \]
	with the conventions that $\inf\varnothing=1/0=+\infty.$
\end{theo}

Such a result should be compared to Theorem \ref{theo:regularity_Y_poissonian}. It constitutes an adequate counterpart to the fact that the previous result must take non-local information into account. Here, we clearly see that only properties and behaviors around $t$ are considered.

\begin{proof}[Sketch of proof]
Such a result is similar to Theorem \ref{theo:regularity_Y_poissonian} and its proof may be done using similar ideas. 
We will only focus on highlighting the few differences that arise when applying the same method.

\medskip

For the upper bound, the key is proving an estimate using jumps similar to the one in Theorem \ref{theo:jaffard_lemma}. It so turns out that the upper bound coming from this approach is the left-hand side of (\ref{eq:jaffard_lemma_naive}), which should not come as a surprise since the $d_\T$-localized exponent only takes into account what happens in the neighborhood of $t.$ The rest of the computation is the same once one has replaced $\V(A,\rho)$ by $B_\T(t,\rho).$

\medskip

As for the lower bound, it is somewhat more involved.
For all $A=A(t)\in\A$ and $\rho>0,$ we introduce a localized version $Y_{(.,\rho)}=\{ Y_{(A',\rho)}:A'\in\A \}$ of $Y$ around $A$ as follows:
\[ \forall A'\in\A,\quad Y_{(A',\rho)} \speq \Delta Y_{A'\cap B_\T(t,\rho)}. \]
The set-indexed processes $Y_{(.,\rho)}$ still have independent increments, and so the martingale arguments developed above will still apply. Let us show that.

\medskip

For the case where $\beta\gs1,$ the key argument lies in Proposition \ref{prop:regularity_Y_beta>1}. We claim that, up to some inconsequential constants, we can replace the probability in (\ref{eq:CS_lower_bound}) by
\[ \prob{ \sup_{C\in\C^\ell\cap B_\C(A,2^{-j}) ~:~ C\subseteq B_\T(t,2^{-j})} |\Delta Y_C| > 2^{-j\eta} } \]
and the rest of the proof would still follow once one replaces $\V(A,2^{-j})$ by $B_\T(t,2^{-j}).$

Indeed, for all $\rho>0,$

\vspace{-.3cm}
\begin{spacing}{1.4}
\[\begin{array}{rcl}
	\ds\sup_{\substack{ C\in\C^\ell\cap B_\C(A,\rho): \\ C\subseteq B_\T(t,\rho)}} |\Delta Y_C|
	& ~\ls~	& \ds\sup_{C\in\C^\ell\cap B_\C(A,\rho)} |\Delta Y_{(C,\rho)}|	
	\vspace{-.5cm} \\		
	& & \hspace{2.8cm} \text{where } \Delta Y_{(.,\rho)} \text{ is the increment map of } Y_{(,.\rho)},								\\
	& \ls 	& \ds\sup_{C\in\C_{p-1}\cap B_\C(A,\rho)} |\Delta Y_{(C,\rho)}| 
	\vspace{-.2cm} \\
	& & \hspace{2.8cm} \text{since } \C^\ell\subseteq\C_{p-1} \text{ where } p=\dim\A,			\\
	& \ls	& \ds 2^p \sup_{A_0,...,A_p\in B_\A(A,\rho)} |Y_{(A,\rho)}-Y_{(A_0\cap...\cap A_p,\rho)}|	\\
	& & \hspace{2.8cm} \text{by the inclusion-exclusion formula (\ref{eq:inclusion-exclusion_formula}).}
\end{array}\]
\end{spacing}
\vspace{-.2cm}

Moreover, for all $A_0,...,A_p\in B_\A(A,\rho),$
\[\begin{array}{rcll}
	d_\A(A,A_0\cap...\cap A_p)
	& ~\ls~	& d_\A(A,A_0\cap A) + d_\A(A_0\cap A,A_0\cap...\cap A_p) 	\\
	& \ls 	& d_\A(A,A_0) + d_\A(A,A_1\cap...\cap A_p)
	& \text{by contractivity,}	\\
	& <		& \rho + d_\A(A,A_1\cap...\cap A_p).
\end{array}\]

Thus, by induction, we deduce that
\[ \sup_{\substack{ C\in\C^\ell\cap B_\C(A,\rho): \\ C\subseteq B_\T(t,\rho)}} |\Delta Y_C| ~\ls~ 2^p \sup_{A'\in B_\A(A,(p+1)\rho)} |Y_{(A,\rho)}-Y_{(A',\rho)}|. \]
Hence for all $j\in\N,$
\begin{equation*}
\begin{array}{l}
	\ds\prob{ \sup_{C\in\C^\ell\cap B_\C(A,2^{-j}) ~:~ C\subseteq B_\T(t,2^{-j})} |\Delta Y_C| > 2^{-j\eta} }	\\
	\hspace{5cm} \ls~ \ds \prob{ \sup_{A'\in B_\A(A,(p+1)2^{-j})} |Y_{(A,2^{-j})}-Y_{(A',2^{-j})}| > 2^{-p}2^{-j\eta} }
\end{array}
\end{equation*}

which proves our claim, since Theorem \ref{theo:cairoli_inequality} still applies to $Y_{(.,2^{-j})}$ and $\Delta Y_{(\V(A,2^{-j}),2^{-j})}=\Delta Y_{B_\T(t,2^{-j})}.$

\medskip

For the case when $\beta<1,$ the trick of introducing the localized process $Y_{(.,\rho)}$ works in a similar fashion.

\end{proof}

\subsection{Examples and applications}	\label{subsection:examples}

In the following, we give some simple criteria when applying Theorems \ref{theo:regularity_Y_poissonian} and \ref{theo:loc_regularity_Y_poissonian}. 
We also give an example showing that the inequality is not always sharp.

\subsubsection{Cases of equality}

We will see that the local geometry of the victiny (resp. open ball) around $A\in\A$ plays a crucial role in order to determine the Hölder exponent (resp. the $d_\T$-localized exponent) of $Y$ at $A.$

\begin{cor}		\label{cor:regularity_Y_poissonian_=}
	Let $A\in\A.$ 
	Suppose that the following hypotheses hold:
	\begin{enumerate}[(i)]
		\item $\sigma^2=0.$
		\item There exists $q_\V>0$ such that for all $\eps>0$, there exists $\rho_{\V,\eps}>0$ such that:
			\begin{equation}	\label{eq:victiny_estimate}
				\forall\rho\in(0,\rho_{\V,\eps}),\quad \rho^{q_\V+\eps} ~\ls~ m(\V(A,\rho)) ~\ls~ \rho^{q_\V-\eps}.
			\end{equation}
		
		\item There exists $\alpha\gs0$ such that for all $\eps>0$, there exists $\rho_{\alpha,\eps}>0$ such that:
			\begin{equation}	\label{eq:f_estimate}
				\forall s\in\V(A,\rho_{\alpha,\eps}),\quad \qd(s,A)^{\alpha+\eps} ~\ls~ |f(s)| ~\ls~ \qd(s,A)^{\alpha-\eps}.
			\end{equation}
	\end{enumerate}
	
	Then, with probability one,
	\[ \alpha_X(A) \speq \frac{q_\V}{\beta} \qquad\text{and}\qquad \alpha_Y(A) \speq \frac{q_\V}{\beta}+\alpha. \]
\end{cor}

\begin{proof}
	The case of $X$ is just the particular case when $f=1,$ for which the estimate (\ref{eq:f_estimate}) works with $\alpha=0$.
	
	Using (\ref{eq:f_estimate}), we get
	\[ \forall\eps>0,\,\forall\rho\in(0,\rho_{\alpha,\eps}),\quad L_{f,\alpha+\eps}(\rho) \speq \V(A,\rho) \quad\text{and}\quad L_{f,\alpha-\eps/2}(\rho) \speq \varnothing. \]  
	Hence, according to (\ref{eq:victiny_estimate}) and Theorem \ref{theo:regularity_Y_poissonian}, for all $\eps>0,$ the following holds with probability one:
	\[ \min\left\{ \frac{\eps^{-1}}{\beta},\,\frac{q_\V-\eps}{\beta}+\alpha-\frac{\eps}{2} \right\} ~\ls~ \alpha_Y(A) ~\ls~ \frac{q_\V+\eps}{\beta}+\alpha+\eps. \]
	The result follows from taking $\eps\rightarrow0^+$ along a subsequence.
\end{proof}

Likewise, Theorem \ref{theo:loc_regularity_Y_poissonian} yields the following result. The proof is exactly the same.

\begin{cor}		\label{cor:loc_regularity_Y_poissonian_=}
	Let $A=A(t)\in\A.$ 
	Suppose that the following hypotheses hold:
	\begin{enumerate}[(i)]
		\item $\sigma^2=0.$
		\item There exists $q_B>0$ such that for all $\eps>0$, there exists $\rho_{\V,\eps}>0$ such that:
			\begin{equation}	\label{eq:ball_estimate}
				\forall\rho\in(0,\rho_{\V,\eps}),\quad \rho^{q_B+\eps} ~\ls~ m(B_\T(t,\rho)) ~\ls~ \rho^{q_B-\eps}.
			\end{equation}
		
		\item There exists $\alpha\gs0$ such that for all $\eps>0$, there exists $\rho_{\alpha,\eps}>0$ such that:
			\begin{equation}	\label{eq:f_loc_estimate}
				\forall s\in B_\T(t,\rho_{\alpha,\eps}),\quad d_\T(s,t)^{\alpha+\eps} ~\ls~ |f(s)| ~\ls~ d_\T(s,t)^{\alpha-\eps}.
			\end{equation}
	\end{enumerate}

	Then, with probability one,
	\[ \alpha_{X,d_\T}(A) \speq \frac{q_B}{\beta} \qquad\text{and}\qquad \alpha_{Y,d_\T}(A) \speq \frac{q_B}{\beta}+\alpha. \]
\end{cor}

The condition on $f$ for Corollary \ref{cor:loc_regularity_Y_poissonian_=} is actually equivalent to say that $f(t)=0$ and $\alpha_f(t)=\underline\alpha_f(t)=\alpha$ where $\alpha_f(t)$ is the pointwise Hölder exponent as given in (\ref{eq:holder_exponent}) and $\underline\alpha_f(t)$ is the \emph{pointwise Hölder subexponent} given by
\[
	\underline\alpha_f(t) \speq \inf\left\{ \alpha\gs0 \,:\, \liminf_{\rho\rightarrow0^+} \inf_{s\in B_\T(t,\rho)} \frac{|f(s)-f(t)|}{d_\T(s,t)^\alpha}>0 \right\}
\]

whenever it is defined. A slightly modified exponent of this kind has already been introduced in \cite{herbin_almost_2014} to study the local Hausdorff dimension of trajectories of Gaussian processes. We also remark one could also express the estimate on the local behavior of the victiny (or the ball) with an exponent-like vocabulary.

\medskip

As a nice consequence to the previous corollaries, we recover a fact that has already been observed in \cite{balanca_2-microlocal_2012, herbin_stochastic_2009} and many others, namely that even in the context of a stochastic integral, integrating still regularizes in some sense.

\begin{cor}[$Y$ is more regular than $X$]	\label{cor:more_regular}
	Let $A=A(t)\in\A$. 
	If the estimate (\ref{eq:victiny_estimate}) holds, then with probability one,
	\[ \alpha_Y(A) ~\gs~ \alpha_X(A). \]
	Similarly, if the estimate (\ref{eq:ball_estimate}) holds, then with probability one,
	\[ \alpha_{Y,d_\T}(A) ~\gs~ \alpha_{X,d_\T}(A) \+ \alpha_f(t)\myindic{f(t)=0}. \]
\end{cor}

Stating a better result for $\alpha_Y(A)$ similar to $\alpha_{Y,d_\T}(A)$ is quite straightforward, but would require introducing another exponent for $f$ considering $\qd$ instead of $d_\T.$ We chose against it since the $d_\T$-localized exponent already illustrates our point.

\begin{proof}
	Let $A=A(t)\in\A$.
	We only prove the result for the $d_\T$-localized exponent, the Hölder exponent being easier.
	 
	If $\nu=0,$ then Proposition \ref{prop:separate_exponents} and (\ref{eq:Y_gaussian_loc_regularity}) immediately yield the result.
	So we might as well suppose that $\sigma^2=0.$
	According to Theorem \ref{theo:loc_regularity_Y_poissonian}, it is enough to prove that with probability one,
	\begin{equation}	\label{eq:CS_more_regular} 
		\sup_{(\alpha,q,q')\in\underline R_f'(A)}\min  \left\{ \frac{q}{\beta},\,\frac{q'}{\beta}+\alpha \right\} ~\gs~\alpha_{X,d_\T}(A) \+ \alpha_f(t)\myindic{f(t)=0}. 
	\end{equation}
	
	Suppose that $f(t)\neq0$ or $\alpha_f(t)=0$ so that $\alpha_f(t)\myindic{f(t)=0}=0.$
	Then, remark that according to (\ref{eq:victiny_estimate}), for all $\eps>0$ and $(\alpha,q,q')\in\underline R_f'(A),$ we may always consider that both $q$ and $q'$ are greater than $q_B-\eps$ in the left-hand side of (\ref{eq:CS_more_regular}).
	In particular,
	\begin{equation}	\label{eq:CS_more_regular_X}
		\sup_{(\alpha,q,q')\in\underline R_f'(A)}\min  \left\{ \frac{q}{\beta},\,\frac{q'}{\beta}+\alpha \right\} ~\gs~ \frac{q_B}{\beta}.
	\end{equation}
	According to Corollary \ref{cor:loc_regularity_Y_poissonian_=},\, $q_B/\beta=\alpha_{X,d_\T}(A)$ almost surely. 
	Hence (\ref{eq:CS_more_regular}) follows from (\ref{eq:CS_more_regular_X}) in this case. 
	
	\medskip
	
	Now, suppose that $f(t)=0$ and $\alpha_f(t)>0$. Then, for all $\alpha\in(0,\alpha_f(t))$ and small enough $\rho>0,\,$ $L_{f,\alpha}'(A,\rho)=\varnothing.$ In particular, $(\alpha,q,q_B-\eps)\in\underline R_f'(A)$ for all $q,\eps>0.$ 
	Thus
	\[ \sup_{(\alpha,q,q')\in\underline R_f'(A)}\min  \left\{ \frac{q}{\beta},\,\frac{q'}{\beta}+\alpha \right\} ~\gs~ \min\left\{ \frac{q}{\beta},\, \frac{q_B-\eps}{\beta}+\alpha \right\}. \]	
	Taking $q\rightarrow+\infty$, $\eps\rightarrow0^+$ and $\alpha\rightarrow\alpha_f(t)^+$ yields
	\begin{equation}	\label{eq:CS_more_regular_f}
		\sup_{(\alpha,q,q')\in\underline R_f'(A)}\min  \left\{ \frac{q}{\beta},\,\frac{q'}{\beta}+\alpha \right\} ~\gs~ \frac{q_B}{\beta} + \alpha_f(t).
	\end{equation}
	Hence (\ref{eq:CS_more_regular}) follows from (\ref{eq:CS_more_regular_f}) in this case.
\end{proof}

We proceed to apply those results to a multiparameter Lévy process in order to show that various behaviors start to appear when $p>1.$
The coming example should be compared to the case $p=1$ where it was proven in \cite{blumenthal_sample_1961} that almost sure regularity is $1/\beta.$

\begin{ex}[Set-indexed Lévy process for $\T=\R_+^p$]
	Suppose that $\T=\R_+^p$ is endowed with its usual indexing collection $\A$ as given in Examples \ref{exs:fundamental_continuous}, $m$ is the Lebesgue measure and $d_\T=\|.-.\|$ is any distance induced by a norm on $\R^p$.
	
	Let also $X$ be a purely Poissonian set-indexed Lévy process on $\A$ and $t\in\R_+^p.$ 
	According to Corollaries \ref{cor:regularity_Y_poissonian_=} and \ref{cor:loc_regularity_Y_poissonian_=} which hold for $q_\V=1$ and $q_B=p$ respectively, the following holds with probability one:
	\[ \alpha_X(A(t)) \speq \left\{\begin{array}{ll}
				1/\beta	& \text{if } t\neq0, \\
				p/\beta & \text{if } t=0
			\end{array}\right.
		\spand \alpha_{X,\|.-.\|}(A(t)) \speq p/\beta.
	\]
		
	When $d_\A=d_m$ is taken instead, the values of $\alpha_X(A(t))$ and $\alpha_{X,d_\T}(A(t))$ do not change as long as $t=(t_1,...,t_p)$ is such that $t_1...t_p\neq0$ since $d_\T$ stays equivalent to $\|.-.\|$ on any compact set away from the coordinate hyperplanes (see \cite[Lemma 3.1]{herbin_almost_2014}).
	
	Let us now consider $t$ in such an hyperplane.
	If $p=1,$ there is nothing much to say and we recover the result of  \cite{blumenthal_sample_1961}, \ie $\alpha_X(0)=1/\beta$ almost surely.
	However, when $p>1,$ we have $m(B_\T(t,\rho))=\infty$ for all $\rho>0,$ so an argument based on Borel-Cantelli ensures that the following event holds with probability one:
	\[ \Omega^* \speq \bigcup_{\eps>0}\bigcap_{\rho>0} \big\{ \exists s\in\Pi\cap B_\T(t,\rho):|J_s(X)|\gs\eps \big\}. \]
	This means that there are sequences of macroscopic jumps converging to $0.$ Applying the estimate (\ref{eq:jaffard_lemma_naive}) yields $\alpha_X(t)=\alpha_{X,d_\T}(t)=0$ almost surely.
\end{ex}

\subsubsection{The one-dimensional case $\T=\R_+$}

When $\T=\R_+,$ both exponents are reduced to the usual pointwise one, yielding the following result.

\begin{cor}[Hölder regularity when $\T=\R_+$]	\label{cor:regularity_Y_1d}
	Suppose that $m$ is the usual Lebesgue measure on $\R_+$ and that $\sigma^2=0.$
	Then, for all $t\in\R_+,$ the following holds with probability one:
	\begin{equation}	\label{eq:regularity_Y_1d} 
		\sup_{(\alpha,q,q')\in\underline R_f(t)} \min\left\{ \frac{q}{\beta},\,\frac{q'}{\beta}+\alpha \right\} ~\ls~ \alpha_Y(t) ~\ls~ \inf_{(\alpha,q)\in\overline R_f(t)} \left\{ \frac{q}{\beta}+\alpha \right\}.
	\end{equation}
	
	Moreover, if there exists $\alpha\gs0$ such that for all $\eps>0,$ there exists $\rho_{\alpha,\eps}>0$ such that
	\[ \forall s\in(t-\rho_{\alpha,\eps},t+\rho_{\alpha,\eps}),\quad |s-t|^{\alpha+\eps} ~\ls~ |f(s)| ~\ls~ |s-t|^{\alpha-\eps}. \]
	then, the following holds with probability one:
	\[ \alpha_Y(A) \speq \frac{1}{\beta}+\alpha. \]
\end{cor}

Remark that in this case, by a similar argument to Corollary \ref{cor:more_regular}, we may always take $q\wedge q'=1$ in the left-hand side of (\ref{eq:regularity_Y_1d}) and $q\gs1$ in its right-hand side, simplifying the expression in practical applications.

\begin{exs} 
	We finally address the simplest case and another one where the upper and lower bounds do not coincide.
	\begin{enumerate}[$\diamond$]
		\item Suppose that $f(s)=s^\alpha$ for all $s\in\R_+$.
		Consider $t\in\R_+.$
		According to Corollary \ref{cor:regularity_Y_1d}, the following holds with probability one:
		\[ \alpha_Y(t) ~=~ \frac{1}{\beta}+\alpha\myindic{t=0}. \]
		 
		\item Let $q\gs1$.
		Consider a Borel set $E\subseteq\R_+$ such that
		\[ 0 ~<~ \liminf_{\rho\rightarrow0^+} \frac{m(E\cap[0,\rho])}{\rho^q} ~\ls~ \limsup_{\rho\rightarrow0^+} \frac{m(E\cap[0,\rho])}{\rho^q} ~<~ \infty. \]
		For instance, the set
		\[ E ~=~ \bigcup_{j\in\N} \lp 2^{-j}-2^{-jq},2^{-j} \right] \]
		works fine.
		Then consider $0\ls\alpha<\alpha'$ and define the function $f(s)=s^\alpha\myindic{E}+s^{\alpha'}\myindic{E^\complement}$ for all $s\in\R_+.$
		
		Applying Corollary \ref{cor:regularity_Y_1d} yields
		\[ \min\left\{ \frac{q}{\beta},\frac{1}{\beta}+\alpha \right\} ~\ls~ \alpha_Y(0) ~\ls~ \min\left\{ \frac{q}{\beta}+\alpha,\, \frac{1}{\beta}+\alpha' \right\} \quad\text{a.s.} \]
		which is not an equality for a large choice of $q,\alpha,\alpha'$ and $\beta.$
	\end{enumerate}
\end{exs}

\bibliographystyle{plain}



\end{document}